\documentclass[final]{amsart}

\usepackage[utf8]{inputenc}
\usepackage[T1]{fontenc}
\usepackage[english]{babel}
\usepackage{amsmath}
\usepackage{amssymb}
\usepackage{amsthm}
\usepackage{mathtools}
\usepackage{mathrsfs}
\usepackage{enumerate}
\usepackage{hyperref}
\usepackage{commath}
\usepackage{cleveref}
\usepackage{tikz}
\usetikzlibrary{tikzmark}
\usetikzlibrary{calc}
\usetikzlibrary{shapes.geometric, arrows}
\usepackage{dsfont,esint}
\usepackage{cancel}

\usepackage[sort, numbers]{natbib}
\usepackage{doi}

\tikzstyle{arrow} = [thick,->,>=stealth]
\tikzstyle{rec} = [rectangle, rounded corners, minimum width=1.4cm, minimum height=0.7cm,text centered, draw=black]

\newcommand{\N}{\mathbb{N}}
\newcommand{\R}{\mathbb{R}}
\newcommand{\Z}{\mathbb{Z}}
\newcommand{\Sph}{\mathbb{S}}

\newcommand{\ProRd}[1]{\mathcal{P}_{#1}(\R^d)}
\newcommand{\normTV}[1]{\norm{#1}}
\newcommand{\MeasRd}[1]{\mathcal{M}_{#1}(\R^d)}

\newcommand{\BorRd}{\mathscr{B}(\R^d)}
\newcommand{\LebRd}{\mathscr{L}^d}

\newcommand{\1}{\mathds{1}}

\newcommand{\comma}{\,\mathrm{,}\;\,}
\newcommand{\semicolon}{\,\mathrm{;}\;\,}
\newcommand{\fstop}{\,\mathrm{.}}
\DeclareMathOperator{\diam}{diam}
\DeclareMathOperator{\dist}{dist}
\DeclareMathOperator{\supp}{supp}

\DeclareMathOperator{\interior}{int}

\newcommand{\eqe}[2]{\tilde e_{#1,#2}}
\newcommand{\eqeCube}[3]{\tilde e_{#1,#2,#3}}
\newcommand{\eqc}[2]{\tilde q_{#1,#2}}

\newtheorem{theorem}{Theorem}[section]
\newtheorem{proposition}[theorem]{Proposition}
\newtheorem{lemma}[theorem]{Lemma}
\newtheorem{corollary}[theorem]{Corollary}
\newtheorem{conjecture}[theorem]{Conjecture}

\theoremstyle{definition}
\newtheorem{definition}[theorem]{Definition}
\newtheorem{example}[theorem]{Example}

\theoremstyle{remark}
\newtheorem{remark}[theorem]{Remark}

\numberwithin{equation}{section}

\usepackage{scrextend}

\begin{document}

\title[Asymptotics for Optimal Empirical Quantization of Measures]{Asymptotics for \\ Optimal Empirical Quantization \\ of Measures}

\author{Filippo Quattrocchi}
\address{Institute of Science and Technology Austria, Am Campus 1, 3400 Klosterneuburg, Austria}
\curraddr{}
\email{filippo.quattrocchi@ista.ac.at}
\thanks{}

\keywords{optimal empirical quantization, vector quantization, Wasserstein distance, semidiscrete optimal transport, Zador's Theorem, Pierce's Lemma}

\date{\today} 

\thanks{The author is thankful to Nicolas Clozeau, Lorenzo Dello Schiavo, Jan Maas, Dejan Slepčev, and Dario Trevisan for many fruitful discussions and comments.
	 The author gratefully acknowledges support from the Austrian Science Fund (FWF) project \href{https://doi.org/10.55776/F65}{10.55776/F65}.}

\dedicatory{}

\subjclass[2020]{41A25 (Primary), 62E17, 49Q22.}

\begin{abstract}
	We investigate the \emph{minimal} error in approximating a general probability measure~$\mu$ on~$\R^d$ by the \emph{uniform} measure on a finite set with prescribed cardinality~$n$. The error is measured in the~$p$-Wasserstein distance. In particular, when~$1\le p<d$, we establish \emph{asymptotic} upper and lower bounds as~$n \to \infty$ on the rescaled minimal error that have the \emph{same, explicit} dependency on~$\mu$.
	
	In some instances, we prove that the rescaled minimal error has a limit. These include general measures in dimension~$d = 2$ with~$1 \le p < 2$, and uniform measures in arbitrary dimension with~$1 \le p < d$. For some uniform measures, we prove the limit existence for~$p \ge d$ as well.

	For a class of compactly supported measures with H\"older densities, we determine the convergence speed of the minimal error for~\emph{every}~$p \ge 1$.
	
	Furthermore, we establish a new Pierce-type (i.e.,~nonasymptotic) upper estimate of the minimal error when~$1 \le p < d$.
	
	In the initial sections, we survey the state of the art and draw connections with similar problems, such as classical and random quantization.
\end{abstract}

\maketitle

\section{Introduction}

\emph{Quantization} is the problem of optimally approximating a probability measure~$\mu$ on~$\R^d$ by another one, say~$\mu_n$, supported on a finite number~$n$ of points. For instance, we can think of~$\mu$ as the description of a picture and of~$\mu_n$ as its digital \emph{compression}. Another typical example comes from \emph{urban planning}: if~$\mu$ represents the population distribution in a city, then the support of the approximating measure~$\mu_n$ determines good locations to build schools, supermarkets, parks, etc.

The mathematical formulation is as follows: for a given number of points~$n$ and a fixed parameter~$p \in [1,\infty)$, find a solution to the minimization problem
\begin{equation}
	e_{p,n}(\mu) \coloneqq \min_{\mu_n \in \ProRd{}} \set{ W_p(\mu,\mu_n) \, : \, \#\supp (\mu_n) \le n} \comma
\end{equation}
where~$W_p$ is the~$p$-Wasserstein--Kantorovich--Rubinstein  distance~\cite{Villani09,Santambrogio15}. We will call this minimal number the $n^\text{th}$ \emph{optimal quantization error} of order~$p$ for~$\mu$. Equivalently (see~\cite{GrafLuschgy00}),~$e_{p,n}(\mu)$ can be written as the following minimum over maps~$T \colon \R^d \to \R^d$ \emph{onto at most~$n$ points}:
\begin{equation*}
	 e_{p,n}(\mu) = \min_T \left( \int \norm{x-T(x)}^p \, \dif \mu \right)^{1/p} = \min_T \mathbb E_{X \sim \mu} \left[\norm{X-T(X)}^p\right]^{1/p} \comma
\end{equation*}
or, also,
\begin{equation} \label{eq:epn}
	e_{p,n} (\mu) = \min_{x_1,\dots,x_n \in \R^d} \left( \int \min_i \norm{x-x_i}^p \, \dif \mu \right)^{1/p} \fstop
\end{equation}

One of the questions that has attracted considerable attention over the years is the \emph{asymptotic} behavior of~$e_{p,n}(\mu)$ as~$n \to \infty$, see~\cite{GrafLuschgy00}. The most fundamental result in this direction is Zador's Theorem: given a probability measure~$\mu$ on~$\R^d$ enjoying a suitable moment condition, and denoting by~$\rho$ the density of its absolutely continuous part, we have
\begin{equation} \label{eq:BWZ}
	\lim_{n \to \infty} n^{1/d} e_{p,n}(\mu) = q_{p,d} \left( \int \rho^\frac{d}{d+p} \, \dif \LebRd \right)^\frac{d+p}{dp} \comma
\end{equation}
with
\begin{equation}
	q_{p,d} \coloneqq \inf_{n \in \N_1} n^{1/d} e_{p,n}\bigl(\LebRd|_{[0,1]^d}\bigr) > 0 \comma
\end{equation}
see~\cite{Zador64,Zador82,BucklewWise82,GrafLuschgy00} and \Cref{thm:BWZ} below. That is, when~$\mu$ is not purely singular (otherwise the limit~\eqref{eq:BWZ} equals zero), Zador's Theorem determines the speed of convergence~$n^{-1/d}$ of~$e_{p,n}(\mu)$ and the explicit dependency of the prefactor on~$\mu$. For a heuristic derivation, see~\cite{Dereich09} or \Cref{sec:heuristics} below.

We will focus on a variation of the classical quantization problem: \emph{optimal} (or \emph{deterministic}) \emph{empirical quantization} (also known as \emph{optimal/deterministic uniform quantization}). The subject of our study is the \emph{optimal empirical quantization error}~$\eqe{p}{n}(\mu)$, defined by
\begin{equation}
	\eqe{p}{n}(\mu) \coloneqq \min_{\mu_n \in \ProRd{}} \set{W_p(\mu,\mu_n) \, : \, \mu_n = \frac{1}{n} \sum_{i=1}^n \delta_{x_i} \text{ for some } x_1,\dots,x_n \in \R^d} \comma
\end{equation}
or, equivalently,
\[
	\eqe{p}{n}(\mu) = \min_{x_1,\dots,x_n \in \R^d} \min_{\mu^{1},\dots,\mu^{n}} \left( \sum_{i=1}^n \int \norm{x-x_i}^p \, \dif \mu^{i} \right)^{1/p} \comma
\]
where~$\mu^{1},\dots,\mu^{n}$ are subprobabilities, \emph{each having total mass equal to~$1/n$}, that sum up to~$\mu$. 
The two numbers~$e_{p,n}(\mu)$ and~$\eqe{p}{n}(\mu)$ are similarly defined, but the second one is a minimum over a \emph{smaller} set of measures, hence~$e_{p,n}(\mu) \le \eqe{p}{n}(\mu)$. Our aim is to find formulas analogous to~\eqref{eq:BWZ} for~$\eqe{p}{n}(\mu)$. %

Several results are available, both for the case~$d=1$ \cite{JourdainReygner16,XuBerger19,GilesHefterMayerRitter19,BencheikhJourdain22,BencheikhJourdain22bis,GilesSheridan-Methven23}, and in arbitrary dimension~\cite{GilesHefterMayerRitter19,MerigotMirebeau16,Chevallier18,GilesHefterMayerRitter19bis,BrownSteinerberger20}, but a general statement like Zador's Theorem is still missing. As we will see in greater detail in \Cref{sec:previous}, the works~\cite{MerigotMirebeau16,Chevallier18} contain the proof of the following. For ``sufficiently nice'' probability measures~$\mu$ (and assuming, for simplicity,~$p \neq d$), we have
\begin{equation*} \label{eq:MerigotMirebeauChevallier}
	0 < \liminf_{n \to \infty} n^{1/d} \eqe{p}{n}(\mu) \quad \text{and} \quad  \limsup_{n \to \infty} n^{\frac{1}{\max(p,d)}} \eqe{p}{n}(\mu) < \infty \fstop
\end{equation*}
In particular, the speed of convergence to~$0$ of~$\eqe{p}{n}(\mu)$ is~$n^{-1/d}$ in the regime~$p < d$.
However, the known bounds from below and above of the limits inferior and superior are rather loose, in that they do not depend in the same way on~$\mu$ (see~\eqref{eq:MerigotMirebeauChevallierbis} below), and the existence of this limit is unknown.

\subsection{Main theorem}

Our main theorem addresses the first of the two matters above by providing a \emph{high-resolution formula} for~$p<d$ (in the same spirit of~\cite{DereichScheutzowSchottstedt13} for random empirical quantization). %

\begin{theorem} \label{thm:main1}
	Assume that~$1 \le p < d$ and let~$p^*\coloneqq \frac{dp}{d-p}$ be the Sobolev conjugate of~$p$. Let~$\mu$ be a probability measure on~$\R^d$ and assume that, for some~$\theta > p^*$, the~$\theta^\text{th}$ moment of~$\mu$ is finite. Let~$\rho$ be the density of the absolutely continuous part of~$\mu$ and let~$\supp \mu^s$ be the support of the singular part of~$\mu$ (w.r.t. the Lebesgue measure~$\LebRd$). Then:
	\begin{align}
		\label{eq:mainbelow} \tag{L}
		q_{p,d} \left(\int_{\R^d \setminus \supp (\mu^s)} \rho^\frac{d-p}{d} \, \dif \LebRd \right)^{1/p} &\le \liminf_{n \to \infty} n^{1/d} \eqe{p}{n}(\mu) \comma \\
		\label{eq:mainabove} \tag{U} \limsup_{n \to \infty} n^{1/d} \eqe{p}{n}(\mu) &\le \eqc{p}{d} \left(\int_{\R^d} \rho^\frac{d-p}{d} \, \dif \LebRd \right)^{1/p} \comma
	\end{align}
	where
	\begin{equation} \label{eq:main10} 
		q_{p,d} \coloneqq \inf_{n \in \N_1} n^{1/d} e_{p,n}\bigl(\LebRd|_{[0,1]^d}\bigr) > 0 \quad \text{and} \quad \eqc{p}{d} \coloneqq \inf_{n \in \N_1} n^{1/d} \eqe{p}{n}\bigl(\LebRd|_{[0,1]^d}\bigr) > 0 \fstop
	\end{equation}
\end{theorem}

Note that the dependence on the measure in~\eqref{eq:BWZ} and in \Cref{thm:main1} is different; we will give a heuristic explanation of this phenomenon in \Cref{sec:heuristics}. It is also worth noting that the integral~$\int \rho^\frac{d-p}{d} \, \dif \LebRd$ has already appeared in the asymptotic study of other (related) problems in \emph{combinatorial optimization} \cite{BeardwoodHaltonHammersley59,Steele97,Yukich98,BartheBordenave13,GoldmanTrevisan24} and \emph{random (empirical) quantization} \cite{GrafLuschgy00,DereichScheutzowSchottstedt13}.

In general, the two constants~$q_{p,d}$ and~$\eqc{p}{d}$ in~\eqref{eq:main10} are not known explicitly, but it is possible to establish upper and lower bounds, see~\cite[Chapters~8~\&~9]{GrafLuschgy00}. We pose the following.%

\begin{conjecture} \label{conj:q}
	The identity~$q_{p,d} = \eqc{p}{d}$ holds (for every~$p \ge 1$ and~$d \in \N_1$).
\end{conjecture}

This is tightly linked with a famous conjecture by A.~Gersho~\cite{Gersho79}, which, in essence, states the following: if~$A \subseteq \R^d$ is convex and we denote by~$U_A$ its uniform measure, then the optimal quantizers~$\mu_n$ for~$e_{p,n}(U_A)$ are asymptotically \emph{uniform}, and ``most'' of the Voronoi regions generated by~$\supp (\mu_n)$ are congruent to one another. Weak versions of Gersho's Conjecture have been proven in~\cite{GrafLuschgyPages12,Zhu11,Zhu20}, but they seem to be insufficient to settle \Cref{conj:q}. As noted in~\cite[Remark~2]{DereichScheutzowSchottstedt13}, proving the equality of the constants appearing in the upper and lower bounds \guillemotleft{}\emph{seems to be a general open problem in transport problems}\guillemetright{} (see also~\cite{BartheBordenave13,GoldmanTrevisan24}).

Nonetheless, with \Cref{ex:1dim} and \Cref{thm:q2} (see below), we show that \Cref{conj:q} is true for~$d=1$ and~$d=2$.

\subsection{Existence of the limit} \label{sec:existenceLimit}

The second matter, namely the convergence of the renormalized error (i.e.,~$n^{1/d} \eqe{p}{n}(\mu)$ if the speed of convergence of~$\eqe{p}{n}(\mu)$ is~$n^{-1/d}$), remains, in general, an open question. For~$p < d$, however, \Cref{thm:main1} reduces this problem to \Cref{conj:q}. Indeed, assuming \Cref{conj:q}, that~$1\le p<d$, that the~$\theta^\text{th}$ moment of~$\mu$ is finite for some~$\theta > p^*$, and~$\int_{\supp (\mu^s)} \rho \, \dif \LebRd = 0$, the limit of~$n^{1/d} \eqe{p}{n}(\mu)$ exists by the combination of~\eqref{eq:mainbelow} and~\eqref{eq:mainabove}.

Moreover, with the results of this work, we are able to prove the limit existence in some cases:

\begin{enumerate}
	\item for every~$p \ge 1$ and~$d \in \N_1$, when~$\mu$ is the uniform measure on a cube, see \Cref{prop:uniformCube};
	\item when~$1 \le p < d$ and~$\mu$ is the uniform measure on a bounded nonnegligible Borel set, see \Cref{cor:uniform};
	\item \label{2dUnif} for every~$p \ge 1$, when~$d = 2$ and~$\mu$ is the uniform measure on a set which is bi-Lipschitz equivalent to a closed disk;
	\item \label{2dGen} when~$1 \le p < d = 2$, the~$\theta^\text{th}$ moment of~$\mu$ is finite for some~$\theta > p^*$, and~$\supp (\mu^s)$ is~$\mu^a$-negligible, where~$\mu^a$ and~$\mu^s$ are the absolutely continuous and singular parts of~$\mu$, respectively.
\end{enumerate}

In all these cases, the upper bound~\eqref{eq:mainabove} is attained in the limit:
\begin{equation} \label{eq:limit}
	\lim_{n \to \infty} n^{1/d} \eqe{p}{n}(\mu) = \eqc{p}{d} \left(\int_{\R^d} \rho^\frac{d-p}{d} \, \dif \LebRd \right)^{1/p} \fstop
\end{equation}

The points~\eqref{2dUnif},\eqref{2dGen} descend directly from \Cref{thm:main1} and the following.

\begin{theorem} \label{thm:q2}
	Let~$A \subseteq \R^2$ be bi-Lipschitz equivalent to a closed disk\footnote{Note that every convex body~$A$ is bi-Lipschitz equivalent to a closed disk: further assuming, without loss of generality, that~$0$ lies in the interior part of~$A$, the map
			\[
			A \ni x \longmapsto \frac{\inf\set{r>0 \, : \, x \in rA}}{\norm{x}} x
			\]
			(deformation by the Minkowski functional) is bi-Lipschitz onto the unit disk.
		} and let~$U_A$ be its uniform measure. Then, for every~$p \ge 1$, the limit of~$\sqrt{n} \eqe{p}{n}(U_A)$ exists and coincides with~$\lim_{n \to \infty} \sqrt{n} e_{p,n}(U_A)$, that is (by~\eqref{eq:BWZ}),
	\begin{equation}
		\lim_{n \to \infty} \sqrt{n} \eqe{p}{n}(U_A) = q_{p,2} \sqrt{\abs{A}} \fstop
	\end{equation}
	
	In particular, we can choose~$A \coloneqq [0,1]^2$ and obtain~$q_{p,2} = \eqc{p}{2}$.
\end{theorem}

By~\cite[Theorem~5.15]{XuBerger19} (restated as \Cref{thm:xuberger} below), the limit exists also when~$d=1$ and the upper quantile function of~$\mu$ is absolutely continuous.

\subsection{Asymptotic behavior for~$p \in [1,\infty)$}

As a first step towards the proof of \Cref{thm:main1}, we will prove~\eqref{eq:limit} for the uniform measure~$\LebRd|_{[0,1]^d}$ \emph{for every~$p \ge 1$} (\Cref{prop:uniformCube}). In particular, we have\footnote{\emph{For~$d \ge 3$}, the bound~\eqref{eq:boundUnif} can also be easily derived from the theory of random empirical quantization, see~\cite[Formula~(8)]{Ledoux23}.}
\begin{equation} \label{eq:boundUnif}
\limsup_{n \to \infty} n^{1/d} \eqe{p}{n} \bigl(\LebRd|_{[0,1]^d}\bigr) < \infty \comma
\end{equation}
from which we derive one corollary which may be of independent interest. Note that, while \Cref{thm:main1} assumes~$p < d$, this corollary applies when~$p \ge d$ as well.

\begin{corollary} \label{cor:pGd2}
	Let~$\tilde \Omega, \Omega$ be open bounded sets in~$\R^d$ and let~$\mu = \rho \LebRd$ be an absolutely continuous probability measure concentrated on~$\Omega$. Assume that:
	\begin{enumerate}
		\item the set $\tilde \Omega$ is convex with~$C^{1,1}$ boundary;
		\item there exists a diffeomorphism~$M \colon \tilde \Omega \to \Omega$ of class~$C^1$ with (globally) H\"older continuous and uniformly nonsingular Jacobian;
		\item the restriction~$\rho|_{\Omega}$ is uniformly positive and H\"older continuous (globally on~$\Omega$).
	\end{enumerate}
	Then, for every~$p \ge 1$,
	\begin{equation}
		0 < \liminf_{n \to \infty} n^{1/d} \eqe{p}{n}(\mu) \le \limsup_{n \to \infty} n^{1/d} \eqe{p}{n}(\mu) < \infty \fstop
	\end{equation}
\end{corollary}

	For general measures and~$p \ge d$, it is possible, and often expected, that
	\[
	\limsup_{n \to \infty} n^{1/d} \eqe{p}{n}(\mu) = \infty \comma
	\]
	see~\cite[Example~5.8 \& Remark 5.22]{XuBerger19} and \Cref{ex:distant} below. \Cref{cor:pGd2} states that the error convergence is still fast (of order~$n^{-1/d}$) if the measure is ``smooth and well-concentrated''.

\begin{remark}
	\Cref{cor:pGd2} applies in particular when~$\Omega$ itself is convex with~$C^{1,1}$ boundary and~$\rho|_\Omega$ is uniformly positive and H\"older: the identity is admissible as a diffeomorphism~$M$ onto~$\Omega$.
\end{remark}

\begin{remark}
	The proof of \Cref{cor:pGd2} relies on a theorem by~S.~Chen, J.~Liu, and X.-J.~Wang~\cite{ChenLiuWang21} on the regularity of optimal transport maps. Using other results from this field, e.g.~\cite{ChenLiuWang19,ChenLiuWang18}, it is possible to adapt \Cref{cor:pGd2} to other sets of assumptions.
\end{remark}

\subsection{Nonasymptotic upper bound}

Along the way, we also prove a nonasymptotic upper bound on the optimal empirical quantization error. This is analogous to what is known as \emph{Pierce's Lemma} \cite{Pierce70} in classical quantization.

\begin{theorem} \label{thm:nonasy}
	Under the assumptions of \Cref{thm:main1}, there exists a constant~$c_{p,d,\theta}$ (independent of~$\mu$ and~$n$) such that
	\begin{equation} \label{eq:thm:nonasy:0}
		n^{1/d} \eqe{p}{n}(\mu) \le c_{p,d,\theta} \left( \int \norm{x}^\theta \, \dif \mu \right)^{1/\theta} \comma \qquad n \in \N_1 \fstop
	\end{equation}
\end{theorem}

\subsection{Related literature} The theory of quantization has been studied since the 1940s by electrical engineers interested in the compression of analog signals. Early works include~\cite{Shannon48} by C.~E.~Shannon,~\cite{OliverPierceShannon48} by B.~M.~Oliver, J.~R.~Pierce and C.~E.~Shannon,~\cite{Bennett48} by W.~R.~Bennett, and~\cite{PanterDite51} by P.~F.~Panter and W.~Dite. We refer the reader to~\cite{GrayNeuhoff98} for a survey of the related literature in the fields of \emph{signal processing} and \emph{information theory} until the late 1990s.

\emph{Algorithms} to solve the quantization problem in~$\R^d$ are known since the works of H.~Steinhaus~\cite{Steinhaus56} and S.~P.~Lloyd~\cite{Lloyd82}. Arguably, the most popular ones are \emph{Lloyd's method} (also known as~\emph{$k$-means algorithm}) and the \emph{Competitive Learning Vector Quantization}, see~\cite[Section~3]{Pages15}.

Over the years, quantization theory has found applications to~\emph{data science} (\emph{clustering}, \emph{recommender systems}, etc.)~\cite{LiuPages20,LiuDongXiaoChenHuZhuZhuSakaiWu24}, mathematical models in \emph{economics}~\cite{BollobasStern72,Bollobas73,ButtazzoSantambrogio09}, \emph{computer graphics}~\cite{BonneelDigne23}, \emph{geometry} (approximation of convex bodies and \emph{Alexandrov's Problem})~\cite{Gruber04,MerigotOudet16}. The survey~\cite{Pages15} describes its applications to numerics, particularly to~\emph{numerical integration}~\cite{Pages98}, \emph{numerical probability}~\cite{PagesPham05}, and numerical solving of (stochastic) (partial) \emph{differential equations}, relevant, e.g.,~in \emph{mathematical finance}~\cite{PagesPrintems09}. Quantization has been studied also beyond the finite-dimensional Euclidean setting, particularly in \emph{Riemannian manifolds}~\cite{Gruber01,Gruber04,Kloeckner12,Iacobelli16,AydinIacobelli24,LeBrigantPuechmorel19,LeBrigantPuechmorel19bis} and \emph{infinite-dimensional Banach spaces} (\emph{functional quantization}), see~\cite{LuschgyPages23} and references therein.%

For a more comprehensive and detailed picture of this extensive mathematical subject, we refer to the following monographs. In chronological order:~\cite{GershoGray92} by~A.~Gersho and~R.~M.~Gray,~\cite{GrafLuschgy00} by~S.~Graf and~H.~Luschgy, and~\cite{LuschgyPages23} by~H.~Luschgy and~G.~Pag\`es.

As previously noted, asymptotics for~$\eqe{p}{n}(\mu)$ have been investigated in~\cite{JourdainReygner16,MerigotMirebeau16,Chevallier18,XuBerger19,GilesHefterMayerRitter19,BencheikhJourdain22,BencheikhJourdain22bis,GilesSheridan-Methven23,GilesHefterMayerRitter19bis,BrownSteinerberger20}, see also \Cref{sec:previous}. Algorithms and other theoretical properties of optimal empirical quantization (and of the slightly more general \emph{capacity-constrained quantization}\footnote{In this version, the competitors~$\mu_n$ are of the form~$\mu_n = \sum_{i=1}^n \lambda_i \delta_{x_i}$, where~$n$ and~$\lambda_1,\dots,\lambda_n \in [0,1]$ are prescribed, and~$x_1,\dots,x_n$ are free.}) have been proposed and studied, e.g.,~in~\cite{AurenhammerHoffmannAronov98,BergerHillMorrison08,deGoesBreedenOstromoukhovDesbrun12,Baker15,XinEtAl16,MerigotSantambrogioSarrazin21,BalzerSchlomerDeussen09,Cortes10,PatelFrascaBullo14}%
. Furthermore, this theory has been used as a tool, e.g., for the approximation of \emph{variational problems} and \emph{(stochastic) differential equations}~\cite{MerigotMirebeau16,Sarrazin22,BencheikhJourdain22bis,KuehnXu22,LiuPfeiffer23,GkogkasKuehnXu23}, to prove convergence rates for \emph{regularized optimal transport}~\cite{EcksteinNutz24}, to analyze \emph{restricted Monte Carlo methods} for quadrature~\cite{GilesHefterMayerRitter19}, to optimally \emph{place robots} in an environment~\cite{Cortes10,PatelFrascaBullo14,CortesEgerstedt17}, in \emph{computer graphics} (e.g.,~to generate \emph{blue-noise} distributions) \cite{BalzerSchlomerDeussen09,deGoesBreedenOstromoukhovDesbrun12,XinEtAl16}, in \emph{neuronal evolution} modeling~\cite{ChevallierDuarteLocherbachOst19}, and in \emph{material} modeling~\cite{DingEtAl23,ZhouEtAl24}.

Several versions of optimal empirical quantization with respect to different metrics/divergences/discrepancies (in place of~$W_p$) have also been studied. We mention~\cite{MakJoseph18,BergerXu20,XuKorbaSlepcev22}, as well as the series of works~\cite{Beck84,AistleitnerDick14,AistleitnerBilykNikolov18,FairchildGoeringWeiss21,Weiss23} on the generalized \emph{star-discrepancy}, which is used to bound numerical integration error by means of the generalized \emph{Koksma--Hlawka inequality}~\cite{AistleitnerDick15}.

Closely related to optimal empirical quantization is~\emph{random} empirical quantization, i.e.,~the problem of approximating a measure~$\mu$ using \emph{random} empirical measures~$\mu_n = \frac{1}{n} \sum_{i=1}^n \delta_{X_i}$, where~$(X_i)_{i \in \N}$ is a sequence of random variables (typically independent and identically distributed), see~\cite{WeedBach19,Ledoux23} and references therein. In recent times, some asymptotic results for this problem have been proven using the theory of partial differential equations~\cite{AmbrosioStraTrevisan19} %
and Fourier analysis~\cite{BobkovLedoux21}.

\subsection{Open questions} It may be interesting to further investigate the following problems.

\begin{enumerate}
	\item In \eqref{eq:mainbelow}, the domain of the integral is~$\R^d \setminus \supp (\mu^s)$. Is this just an artifact of our proof? That is: Can we replace this domain with the whole space~$\R^d$?
	
	\item We already stated \Cref{conj:q} on the equality~$q_{p,d}=\eqc{p}{d}$. Unclear is also the relation between~$q_{p,d},\eqc{p}{d}$ and the constants that appear in~\cite[Theorem~2]{DereichScheutzowSchottstedt13} and~\cite[Theorem~1.6]{CagliotiGoldmanPieroniTrevisan24} in the context of random empirical quantization. Numerical estimates of the constants may also help understand this relation.
	
	\item Depending on~$\mu$, several asymptotic behaviors are possible for the error~$\eqe{p}{n}(\mu)$ when~$p \ge d$, see~\cite[Table~1]{BencheikhJourdain22} as well as \Cref{cor:pGd2}, \Cref{ex:distant}, and \Cref{prop:uniformCube}. It may be worth determining precise characterizations of the measures that exhibit a certain error decay. For example, given~$p \ge d$, for which absolutely continuous and compactly supported measures~$\mu$ is the limit superior of~$ n^{1/d} \eqe{p}{n}(\mu)$ finite?
	
	\item What can we say about the error asymptotics for \emph{singular} measures?
	
	\item It would be natural to also study the problem on manifolds (as in~\cite{Kloeckner12,Iacobelli16,AydinIacobelli24} for classical quantization) and infinite-dimensional spaces.
\end{enumerate}

\subsection{Plan of the work}
The first four sections are preparatory. In \Cref{sec:heuristics}, we give a simple heuristic argument that justifies the integral~$\int \rho^\frac{d-p}{d} \, \dif \LebRd$ in \Cref{thm:main1}. In \Cref{sec:preliminaries}, we fix the notation and give all necessary definitions. In \Cref{sec:previous}, we present some of the existing results in the literature, both to provide context and because we will use some of them.

The subsequent sections contain proofs. The major ones will be preceded by comments on the core ideas and techniques. In \Cref{sec:nonasy}, we prove the nonasymptotic upper bound of \Cref{thm:nonasy}. In \Cref{sec:cube}, we begin the proof of \Cref{thm:main1} by proving the limit~\eqref{eq:limit} for the uniform measure on the unit cube. In \Cref{sec:corollaries}, we prove \Cref{cor:pGd2}. In \Cref{sec:main}, we complete the proof of \Cref{thm:main1}. In \Cref{sec:limitUnif}, we prove the limit~\eqref{eq:limit} for uniform measures in the regime~$p < d$. In \Cref{sec:limit}, we prove \Cref{thm:q2}.

Not all sections are necessary for the later arguments in this manuscript. The following scheme outlines the logical dependencies among the sections 5 to 10.

\begin{figure}[h]
\begin{tikzpicture}[node distance=2.1cm]
	\node (s5) [rec] {Sec.~\ref{sec:nonasy}};
	\node (s6) [rec, right of=s5] {Sec.~\ref{sec:cube}};
	\node (s7) [rec, right of=s6] {Sec.~\ref{sec:corollaries}};
	\node (s8) [rec, right of=s7] {Sec.~\ref{sec:main}};
	\node (s9) [rec, right of=s8] {Sec.~\ref{sec:limitUnif}};
	\node (s10) [rec, right of=s9] {Sec.~\ref{sec:limit}};
	\draw [arrow] (s5) to [out=25,in=150] (s8);
	\draw [arrow] (s6) to [out=30,in=160] (s8);
	\draw [arrow] (s6) to  (s7);
	\draw [arrow] (s8) to  (s9);
\end{tikzpicture}
\end{figure}

\subsection*{Note} Some of the results in this work have already appeared in the author's master's thesis \cite{Quattrocchi21}, written under the supervision of Dario~Trevisan at the University of Pisa. The lower bound \eqref{eq:mainbelow} in \Cref{thm:main1}, \Cref{cor:pGd2}, and \Cref{thm:nonasy} are new.

\section{Heuristics} \label{sec:heuristics}

Firstly, let us formally derive Zador's Theorem. A similar heuristic argument is given in~\cite{Dereich09}. Fix a ``nice'' probability measure~$\mu$, say absolutely continuous, compactly supported, and with continuous density~$\rho$. Let~$S_n = \set{x_1,\dots,x_n} \subseteq \R^d$ be the support of an optimal classical quantizer (i.e.,~a minimizer in~\eqref{eq:epn}) and let~$\sigma_n \LebRd$ be a ``nice'' approximation of the measure~$\frac{1}{n} \sum_{i=1}^n \delta_{x_i}$. For~$n$ large, the number of points in~$S_n$ that fall within a small ball~$B_\epsilon(\bar x)$ of radius~$\epsilon$ centered at~$\bar x \in \R^d$ is, approximately and up to a dimensional constant,~$ \epsilon^d n \sigma_n(\bar x)$. Since~$\rho$ is continuous and~$\epsilon$ is small, we can expect the points of~$S_n \cap B_\epsilon(\bar x)$ to be evenly spread on~$B_\epsilon(\bar x)$; therefore, the distance~$r$ of a generic point in such a ball from~$S_n$ is roughly equal to the $d^\text{th}$ root of the ratio between the volume of the ball and the cardinality~$\# (S_n \cap B_\epsilon(\bar x))$, i.e.,\footnote{In this work, the symbols~$\approx,\lessapprox,\gtrapprox$ do not have a rigorous meaning. They are used in heuristic arguments as shorthands for `is approximately equal to' and `is approximately smaller/greater than'.}~$r \approx \sqrt[d]{\frac{\epsilon^d}{\epsilon^d n \sigma_n(\bar x)}} = \bigl(n \sigma_n(\bar x )\bigr)^{-1/d}$. Hence,
\[
	\int \min_i \norm{x-x_i}^p \, \dif \mu \approx n^{-p/d} \int \sigma_n^{-p/d} \rho \, \dif \LebRd \fstop
\]
Thus, we can rephrase the problem in~\eqref{eq:epn} as a minimization over functions:
\[
	e_{p,n}^p(\mu) \approx n^{-p/d} \inf_\sigma \int \sigma^{-p/d} \rho \, \dif \LebRd \comma
\]
under the constraint~$\int \sigma \, \dif \LebRd = 1$.
By H\"older's inequality,
\begin{align*}
	\int \rho^\frac{d}{d+p} \, \dif \LebRd &\le \left( \int \left(\rho^\frac{d}{d+p} \sigma^{-\frac{p}{d+p}} \right)^\frac{d+p}{d} \, \dif \LebRd \right)^\frac{d}{d+p} \left( \int \left(\sigma^\frac{p}{d+p} \right)^\frac{d+p}{p} \, \dif \LebRd \right)^\frac{p}{d+p} \\
	&= \left( \int \sigma^{-p/d} \rho \, \dif \LebRd \right)^\frac{d}{d+p} \underbrace{\left(\int \sigma \, \dif \LebRd \right)^\frac{p}{d+p}}_{=1} \comma
\end{align*}
and the inequality is an equality for~$\sigma \coloneqq c \, \rho^\frac{d}{d+p}$, where~$c$ is a normalizing constant.

\medskip

In the case of optimal empirical quantization, we expect that the optimal locations~$S_n = \set{x_1,\dots,x_n}$ are, instead, approximately distributed according to~$\rho$: to keep the Wasserstein distance minimal, we should approximately match the mass in every small ball~$B_\epsilon(\bar x)$ to the points in (or closest to) such a ball, which means, in particular,
\[
	\epsilon^d \rho(\bar x) \approx \mu(B_\epsilon(\bar x)) \approx n^{-1} \# (S_n \cap B_\epsilon(\bar x)) \approx \epsilon^d \sigma_n(\bar x)\comma
\]
where, as before,~$\sigma_n$ is an approximation of the uniform measure on~$S_n$.\footnote{This explains why the same formula (up to constant) appears in random empirical quantization, see~\cite[Theorem~2]{DereichScheutzowSchottstedt13}.} Since, once again, the points~$S_n \cap B_\epsilon(\bar x)$ are evenly spread on~$B_\epsilon(\bar x)$, a generic point~$x \in B_\epsilon(\bar x)$ should be matched by an optimal transport plan to the \emph{closest}~$x_i \in S_n$. Recall that the typical distance from~$S_n$ is of order~$\bigl(n \sigma_n(\bar x))^{-1/d}$, which, combined with the considerations above, yields
\[
	\eqe{p}{n}^p(\mu) \approx n^{-p/d} \int \rho^{-p/d} \rho \, \dif \LebRd \fstop
\]

\medskip

We conclude this section with another simple observation. \emph{Postulate} that
\[
	\eqe{p}{n}(\mu) \approx n^{-a} \left( \int \rho^b \, \dif \LebRd \right)^c
\]
for some~$a,b,c \in \R$ and for every (sufficiently ``nice'') probability measure~$\mu = \rho \LebRd$. It is easy to check that~$\eqe{p}{n}(\lambda^{-d} \rho\bigl(\lambda^{-1} \cdot\bigr) \LebRd) = \lambda \eqe{p}{n}(\mu)$ for every~$\lambda > 0$. Therefore,
\[
	\lambda \left( \int \rho^b \, \dif \LebRd \right)^c = \left( \int_{\R^d} \bigl(\lambda^{-d}\rho(\lambda^{-1} \cdot ) \bigr)^b \, \dif \LebRd \right)^c = \lambda^{dc(1-b)} \left( \int \rho^b \, \dif \LebRd \right)^c \comma
\]
from which we obtain the identity~$1 = dc (1-b)$. Note that this is coherent with the statement of \Cref{thm:main1}.

\section{Preliminaries} \label{sec:preliminaries}

\subsection{Notation}

We regard~$\R^d$ as a measure space endowed with the~$\sigma$-algebra of \emph{Borel sets}~$\BorRd$, on which the \emph{Lebesgue measure}~$\LebRd$ is defined, and as a normed space with the \emph{Euclidean norm}~$\norm{\cdot}=\norm{\cdot}_2$. For every~$x \in \R^d$, we let~$\delta_x$ be the Dirac delta measure at~$x$. Given a Borel set~$A \in \BorRd$, we sometimes write~$\abs{A}$ in place of~$\LebRd(A)$. If~$|A| \neq 0,\infty$, it is well defined the \emph{uniform measure} \[ U_A \coloneqq \frac{\LebRd|_A}{\abs{A}}\] of~$A$. For convenience, we further define
\[
	U_d \coloneqq U_{[0,1]^d} \fstop
\]

For every set~$A \subseteq \R^d$, we denote by~$\diam(A)$ its diameter, i.e.,
\[
	\diam(A) \coloneqq \begin{cases}
		0 &\text{if } A = \emptyset \comma \\
		\sup_{x,y \in A}  \norm{x-y} &\text{otherwise,}
	\end{cases}
\]
and by~$\#A \in \N_0 \cup \set{\infty}$ its cardinality. We write~$\interior(A)$ and~$\overline{A}$ for its interior part and topological closure, respectively.

For every pair of sets~$A,B \subseteq \R^d$, we denote by~$\dist(A,B)$ their minimal distance
\[
	\dist(A,B) \coloneqq \begin{cases}
		\inf \set{\norm{x-y} \, : \, x \in A \comma y \in B} &\text{if } A,B \neq \emptyset \comma \\
		\infty &\text{otherwise,}
	\end{cases}
\]
and, similarly, we write~$\dist(x,A) \coloneqq \dist(\set{x},A)$.

We denote by~$\ProRd{}$ the space of Borel \emph{probability measures} on~$\R^d$ and by~$\MeasRd{}$ the space of Borel \emph{nonnegative finite measures} on $\R^d$, i.e.,~$\MeasRd{} \coloneqq \R_{\ge 0} \cdot \ProRd{}$. For every~$p \ge 1$, it is also convenient to introduce the space~$\ProRd{p}$ of probability measures with \emph{finite~$p^\mathrm{th}$ moment}
\[ \ProRd{p} \coloneqq \set{ \mu \in \ProRd{} \, \colon \, \int \norm{x}^p \, \dif \mu(x) < \infty }  \]
and the space~$\ProRd{{\mathrm{c}}}$ of \emph{compactly supported probability measures}
\[ \ProRd{{\mathrm{c}}} \coloneqq \set{ \mu \in \ProRd{} \, \colon \, \exists K \subseteq \R^d \text{ compact such that } \mu(K)=1 } \fstop \]
For~$n \in \N_1$, we further define the set
\[
	\ProRd{(n)} \coloneqq \set{\mu_n \in \ProRd{} \, : \, \exists x_1,x_2,\dotsc, x_n \in \R^d \comma  \mu_n = \frac{1}{n} \sum_{i=1}^n \delta_{x_i} } \fstop
\]
Analogously, we set
\[
	\MeasRd{p} \coloneqq \R_{\ge 0} \cdot \ProRd{p}\comma \, \MeasRd{{\mathrm{c}}} \coloneqq \R_{\ge 0} \cdot \ProRd{{\mathrm{c}}}\comma \, \MeasRd{(n)} \coloneqq \R_{\ge 0} \cdot \ProRd{(n)} \comma
\]
and~$\MeasRd{(0)} \coloneqq \set{0}$.

The (total variation) \emph{norm of a measure}~$\mu \in \MeasRd{}$ is~$\normTV{\mu} \coloneqq \mu(\R^d)$.
For every measurable function~$T \colon \R^{d_1} \to \R^{d_2}$, we denote by~$T_\# \colon \mathcal M(\R^{d_1}) \to \mathcal M(\R^{d_2})$ the \emph{pushforward} operator, defined by
\[
	T_\# \mu(A) \coloneqq \mu\bigl( T^{-1}(A) \bigr) \comma \qquad A \in \mathscr B(\R^{d_2}) \fstop
\]
Note that the norm is invariant under pushforward, i.e.,~$\normTV{T_\# \mu} = \normTV{\mu}$.
For~$\mu \in \MeasRd{}$, we write~$\supp(\mu)$ for the support of~$\mu$, i.e., the smallest closed set on which~$\mu$ is concentrated. For ease of notation, when~$\rho \in L^1_{\ge 0}(\R^d)$, we sometimes write~$\rho$ to denote the measure~$\rho \LebRd \in \MeasRd{}$.  %

We use the notation~$a \lesssim b$ when there exists a constant~$c>0$ for which~$a \le cb$. Given two sequences~$(a_n)_n$ and~$(b_n)_n$ of positive real numbers (defined for an unbounded set of natural indices), we write~$a_n \asymp b_n$ if
\[
	\frac{1}{c} \le  \liminf_{n \to \infty} \frac{a_n}{b_n} \le \limsup_{n \to \infty} \frac{a_n}{b_n} \le c
\]
for some constant~$c > 0$. Possible dependencies of the constant~$c$ are explicitly displayed as subscripts of the symbols~$\lesssim$ and~$\asymp$.

\subsection{Wasserstein distance}
Let~$p \ge 1$, and take two measures~$\mu, \nu \in \MeasRd{p}$ such that~$\normTV{\mu} = \normTV{\nu}$. We denote by~$\Gamma(\mu,\nu)$ the set of couplings between~$\mu$ and~$\nu$, i.e., the nonnegative Borel measures~$\gamma$ on~$\R^d \times \R^d$ that have~$\mu$ and~$\nu$ as marginals. The \emph{Wasserstein distance} of order~$p$ between~$\mu$ and~$\nu$ is given by the formula
\begin{equation} \label{eq:wasserstein}
	W_p(\mu,\nu) \coloneqq \inf_{\gamma \in \Gamma(\mu,\nu)} \left( \int \norm{x-y}^p \, \dif \gamma(x,y) \right)^{1/{p}} \fstop
\end{equation}
The function~$W_p$ is really a distance on~$\lambda \ProRd{p}$ for every~$\lambda \ge 0$ (the case~$\lambda = 0$ is trivial), as shown, for instance, in~\cite[Proposition 5.1]{Santambrogio15}, and we have~$W_p(\lambda \mu, \lambda \nu) = \lambda^{1/p} W_p(\mu,\nu)$ for every admissible choice of~$\mu,\nu,\lambda$. %
Moreover, by a simple compactness argument (see~\cite[Theorem 1.7]{Santambrogio15}), the infimum in~\eqref{eq:wasserstein} is actually a minimum.

The following two nice features of~$W_p$ will be used in this work. The first one is a subadditivity property.

\begin{lemma}[Subadditivity] \label{lemma:subadditivityWasserstein}
	Let~$\mu^1,\mu^2,\nu^1,\nu^2 \in \MeasRd{p}$ be such that~$\normTV{\mu^1} = \normTV{\nu^1}$ and~$\normTV{\mu^2} = \normTV{\nu^2}$. Then we have
	\begin{equation} \label{eq:subadditivityWasserstein}
		W_p^{p}\bigl(\mu^1 + \mu^2, \nu^1 + \nu^2\bigr) \le W_p^{p}(\mu^1,\nu^1) + W_p^{p}(\mu^2,\nu^2) \fstop
	\end{equation}
\end{lemma}

\begin{proof}
	This result follows from the implication
	\[
		\gamma^i \in \Gamma(\mu^i,\nu^i) \comma i \in \set{1,2} \quad \Longrightarrow \quad \gamma^1 + \gamma^2 \in \Gamma\bigl(\mu^1 + \mu^2, \nu^1 +  \nu^2 \bigr)
	\]
	and the linearity of
	\[
		\gamma \mapsto \int \norm{x-y}^p \, \dif \gamma(x,y) \fstop \qedhere
	\]
\end{proof}

The second one is: on a fixed compact set, the Wasserstein distance of two a.c.~measures can be controlled by the~$L^1$-distance of their densities.

\begin{lemma}[Comparison with~$\norm{\cdot}_{L^1}$] \label{lemma:comparL1}
	Let~$\mu = \rho \LebRd,\nu = \sigma \LebRd$ be compactly supported and absolutely continuous measures, with~$\normTV{\mu} = \normTV{\nu}$. Let~$A \subseteq \R^d$ be a bounded set on which both~$\mu$ and~$\nu$ are concentrated. Then:
	\begin{equation}
		W_p(\mu,\nu) \le \diam(A) \norm{\rho-\sigma}_{L^1}^{1/p} \fstop
	\end{equation}
\end{lemma}

\begin{proof}
	We can and will assume that~$\mu \neq \nu$. Set
	\[
		\mu^1 = \nu^1 \coloneqq \min(\rho,\sigma) \LebRd \comma \quad \mu^2 = \max(\rho-\sigma,0)\LebRd \comma \qquad \nu^2 = \max(\sigma-\rho,0)\LebRd \comma
	\]
	and notice that neither~$\mu^2$ nor~$\nu^2$ is equal to the zero measure. The hypotheses of \Cref{lemma:subadditivityWasserstein} are satisfied. Hence,
	\[
		W_p^{p}(\mu,\nu) \le \underbrace{W_p^{p}(\mu^1,\nu^1)}_{=0} + W_p^{p}(\mu^2,\nu^2) \fstop
	\]
	Therefore, it suffices to find a suitable coupling between~$\mu^2$ and~$\nu^2$. We choose
	\[
		\gamma \coloneqq \frac{\mu^2 \otimes \nu^2}{\normTV{\mu^2}} = \frac{\mu^2 \otimes \nu^2}{\normTV{\nu^2}} \comma
	\]
	which yields
	\begin{align*}
		\int \norm{x-y}^p \dif \gamma(x,y) &\le \int \norm{x-y}^p \frac{\mu^2 \otimes \nu^2}{\normTV{\nu^2}}(x,y) \, \dif x \, \dif y \\
		&\le \diam(A)^p \frac{\normTV{\mu^2} \normTV{\nu^2}}{\normTV{\nu^2}} = \diam(A)^p \normTV{\mu^2} \fstop
	\end{align*}
	We conclude by the inequality~$\normTV{\mu^2} \le \norm{\rho-\sigma}_{L^1}$.
\end{proof}

\subsection{Boundary Wasserstein pseudodistance}
A.~Figalli and N.~Gigli introduced in~\cite{FigalliGigli10} a modified Wasserstein distance~$Wb$ for measures defined on a bounded Euclidean domain, giving a special role to the boundary of such a domain: it can be interpreted as an infinite reservoir, where mass can be deposited and taken freely. We give here a slightly modified definition of a \emph{pseudodistance} between measures \emph{defined on the whole~$\R^d$}.

Let~$p \ge 1$ and fix an open bounded nonempty set~$\Omega \subseteq \R^d$. Take two measures~$\mu, \nu \in \MeasRd{}$, possibly having different total mass. Let~$\Gamma b_\Omega(\mu,\nu)$ be the set of the nonnegative Borel measures~$\gamma$ on the closure~$\overline \Omega \times \overline \Omega$ such that~$\gamma|_{\Omega \times \overline \Omega}$ has~$\mu|_\Omega$ as first marginal, and~$\gamma|_{\overline \Omega \times \Omega}$ has~$\nu|_\Omega$ as second marginal.

\begin{definition}
	The \emph{boundary Wasserstein pseudodistance} of order~$p$ for~$\Omega$ between~$\mu$ and~$\nu$ is given by the formula
	\begin{equation}
		Wb_{\Omega, p}(\mu, \nu) \coloneqq \inf_{\gamma \in \Gamma b_\Omega(\mu,\nu)} \left( \int \norm{x-y}^p \, \dif \gamma(x,y) \right)^{1/p} \fstop
	\end{equation}
\end{definition}

It is easy to check that~$Wb_{\Omega,p}(\mu,\nu)$ is nonnegative and finite for every~$\mu,\nu$, that the symmetry property~$Wb_{\Omega,p}(\mu,\nu) = Wb_{\Omega,p}(\nu,\mu)$ holds, and that~$Wb_{\Omega,p}(\mu,\mu) = 0$. The triangle inequality can be proven as in \cite[Theorem 2.2]{FigalliGigli10} (or directly deduced from this theorem). Clearly, with our definition,~$Wb_{\Omega,p}$ cannot be a true distance, as it does not distinguish measures that differ out of~$\Omega$: for every~$\mu,\nu \in \MeasRd{}$, we have~$Wb_{\Omega,p}(\mu,\nu) = Wb_{\Omega,p}(\mu|_\Omega,\nu|_\Omega)$. As with~$W_p$, we have the identity~$Wb_{\Omega,p}(\lambda \mu, \lambda \nu) = \lambda^{1/p} W_{\Omega,p}(\mu,\nu)$ for~$\lambda \ge 0$. Further notice that~$Wb_{\Omega,p}(\mu,\nu) \le W_p(\mu,\nu)$ when~$\mu, \nu \in \MeasRd{p}$ and~$\normTV{\mu} = \normTV{\nu}$, for any~$\Omega$.

A crucial property of~$Wb_{\Omega,p}$ is its geometric superadditivity, which will be used in the proof of the lower bound~\eqref{eq:mainbelow}.

\begin{lemma}[Superadditivity] \label{lemma:superadditivityWb}
	If~$\set{\Omega_i}_i$ is a (finite or countably infinite) family of open, bounded, nonempty, and pairwise \emph{disjoint} subsets of~$\Omega$, then
	\begin{equation}
		Wb_{\Omega,p}^p(\mu, \nu) \ge \sum_{i} Wb_{\Omega_i,p}^p(\mu, \nu) \comma \qquad \mu, \nu \in \MeasRd{} \fstop
	\end{equation}
\end{lemma}

\begin{proof}
The proof of this lemma can be found in \cite[Section 2.2]{AmbrosioGoldmanTrevisan22}.
\end{proof} 

\subsection{Quantization errors and coefficients}

\begin{definition} \label{def:eqe}
	The $n^\text{th}$ \emph{optimal quantization error} of order~$p$ is
	\begin{multline} \label{eq:defe}
		e_{p,n}(\mu) \coloneqq \inf \set{ W_p(\mu,\mu_n) \, : \, \#\supp(\mu_n) \le n \text{  and } \normTV{\mu} = \normTV{\mu_n} } \comma \\ \mu \in \MeasRd{p}\comma
	\end{multline}
	and the \emph{optimal quantization coefficient} of order~$p$ is
	\begin{equation}
		q_{p,d} \coloneqq \inf_{n \in \N_1} n^{1/d} e_{p,n}(U_d) \fstop
	\end{equation}
\end{definition}
\begin{definition}
	The $n^\text{th}$ \emph{optimal empirical quantization error} of order~$p$ is
	\begin{multline}\label{eq:defetilde}
		\eqe{p}{n}(\mu) \coloneqq \inf \set{ W_p(\mu, \mu_n) \, : \, \mu_n \in \MeasRd{(n)} \comma \normTV{\mu_n} = \normTV{\mu} } \comma \\ \mu \in \MeasRd{p} \comma
	\end{multline}
	and the \emph{optimal empirical quantization coefficient} of order~$p$ is
	\begin{equation}
		\eqc{p}{d} \coloneqq \inf_{n \in \N_1} n^{1/d} \eqe{p}{n}(U_d) \fstop
	\end{equation}
\end{definition}
We leave~$e_{p,0}(\mu)$ and~$\eqe{p}{0}(\mu)$ undefined when~$\mu \neq 0$.

In words, the optimal quantization error measures the minimal distance to atomic measures supported on at most~$n$ points (with the same total mass); the optimal empirical quantization error measures the minimal distance to (appropriately rescaled) sums of $n$ Dirac deltas.

\begin{remark} \label{rmk:scaling}
	For every~$\mu \in \MeasRd{p}$, the following inequality hold:
	\begin{equation} \label{eq:disErrors}
		e_{p,n}(\mu) \le \eqe{p}{n}(\mu) \fstop
	\end{equation}
	Both errors are~$\frac{1}{p}$-homogeneous and~$e_{p,n}(0) = \eqe{p}{n}(0) = 0$ for every~$n$, including~$n=0$. Moreover, if~$T \colon \R^d \to \R^d$ is an affine transformation of the form~$T(x) = v + \lambda x$, with~$v \in \R^d$ and~$\lambda \in \R$ then
	\begin{equation} \label{eq:scalingDilation}
		e_{p,n}(T_\# \mu) = \abs{\lambda} e_{p,n}(\mu) \comma \quad \eqe{p}{n}(T_\# \mu) = \abs{\lambda} \eqe{p}{n}(\mu) \fstop
	\end{equation}
\end{remark}

\begin{remark}
	From~\eqref{eq:disErrors}, we deduce also~$q_{p,d} \le \eqc{p}{d}$. Moreover, the quantization coefficients are strictly positive, see \Cref{thm:BWZ}.
\end{remark}

\begin{remark} \label{rmk:subaddErr}
	Let~$\mu^1,\mu^2 \in \MeasRd{p}$ and~$n_1,n_2 \in \N_0$ be such that
	\[
		\bigl[ n_i = 0 \Rightarrow \mu^i = 0 \bigr] \comma \qquad i \in \set{ 1,2} \fstop
	\]
	Then it follows from \Cref{lemma:subadditivityWasserstein} that
	\begin{equation}
	e^p_{p,n_1+n_2}(\mu^1+\mu^2) \le e^p_{p,n_1}(\mu^1) + e^p_{p,n_2}(\mu^2) \fstop
	\end{equation}
	If, moreover,~$
		\normTV{\mu^1} n_2 = \normTV{\mu^2} n_1
	$,
	then
	\begin{equation}
		\eqe{p}{n_1+n_2}^p(\mu^1+\mu^2) \le \eqe{p}{n_1}^p(\mu^1) + \eqe{p}{n_2}^p(\mu^2) \fstop
	\end{equation}
\end{remark}

The infima in~\eqref{eq:defe} and~\eqref{eq:defetilde} are, in fact, minima. For~$e_{p,n}(\mu)$, the proof can be found in~\cite[Theorem~4.12]{GrafLuschgy00} (which, in turn, follows the lines of~\cite[Lemma~8]{Pollard82}) or in~\cite[Appendix~A.4]{Quattrocchi21}. Let us prove the existence of the minimum in~\eqref{eq:defetilde}.

\begin{lemma}
	\label{lemma:existence}
	Let~$\mu \in \MeasRd{p}$. For every~$n \in \N_1$ there exists a measure~$\mu_n \in \MeasRd{(n)}$ with~$\normTV{\mu_n} = \normTV{\mu}$ and such that~$\eqe{p}{n}(\mu) = W_p(\mu,\mu_n)$.
\end{lemma}

\begin{proof}
	If~$\mu = 0$, then~$\mu_n \coloneqq 0 \in \MeasRd{(n)}$ is the sought measure. Otherwise, we may renormalize and assume that~$\normTV{\mu} = 1$, we have to prove that the function
	\[
		\psi \colon \R^{nd} \ni (x_1,\dots,x_n) \longmapsto W_p\left( \mu, \frac{1}{n} \sum_{i=1}^n \delta_{x_i} \right)
	\]
	admits a minimizer.
	This function is continuous: by the triangle inequality and \Cref{lemma:subadditivityWasserstein},
	\begin{align*}
		\abs{\psi(x_1,\dots,x_n)-\psi(y_1,\dots,y_n)}^p %
		&\le W_p^p \left(\frac{1}{n} \sum_{i=1}^n \delta_{x_i},\frac{1}{n} \sum_{i=1}^n \delta_{y_i}\right) \\ &\le \frac{1}{n} \sum_{i=1}^n W_p^p(\delta_{x_i},\delta_{y_i}) = \frac{1}{n} \sum_{i=1}^n \norm{x_i-y_i}^p 
	\end{align*}
	for every~$x_1,\dots,x_n,y_1,\dots,y_n \in \R^d$. Again by the triangle inequality,
	\begin{align*}
		\psi(x_1,\dots,x_n) %
		&\ge W_p\left( \delta_0, \frac{1}{n} \sum_{i=1}^n \delta_{x_i} \right) - W_p( \mu, \delta_0 ) = \frac{1}{n} \sum_{i=1}^n \norm{x_i}^p - \underbrace{\int \norm{x}^p \, \dif \mu}_{< \infty} \comma
	\end{align*}
	which implies that the sublevels of~$\psi$ are bounded.
	We conclude by applying the extreme value theorem on a sufficiently large compact set.
\end{proof}

Let us show that the sequence~$\bigl(\eqe{p}{n}(\mu)\bigr)_n$ is infinitesimal as~$n \to \infty$ \emph{for every}~$\mu$ with finite~$p^\text{th}$ moment. In dimension~$d=1$, this was established in~\cite[Corollary~5.12]{XuBerger19}. The analogous result for~$e_{p,n}(\mu)$ follows as a corollary, but was also proven, e.g.,~in~\cite[Lemma~6.1]{GrafLuschgy00}.

\begin{proposition} \label{prop:infinitesimal}
	For every~$\mu \in \MeasRd{p}$, we have
	\begin{equation}
		\lim_{n \to \infty} \eqe{p}{n}(\mu) = 0 \fstop
	\end{equation}
\end{proposition}

\begin{proof}
	We may and will assume that~$\mu$ is a probability measure, not concentrated on a single point. Set
	\[
		M \coloneqq \int \norm{x}^p \, \dif \mu \comma
	\]
	and fix~$r > 0$ large enough for the ball~$B_r \coloneqq \set{x \, : \, \norm{x} < r}$ to have nonzero~$\mu$-measure.
	
	If~$\mu$ is concentrated on~$B_r$, then the conclusion follows from \Cref{thm:MerigotMirebeauChevallier}. Otherwise, let us define
	\[
		R \coloneqq \left( \frac{2M}{1-\mu(B_r
			)} \right)^{1/p}
	\]
	and notice that, by Markov's inequality,
	\begin{equation} \label{eq:infinitesimal1}
		\mu \set{x \, : \, \norm{x} \ge R} \le \frac{M}{R^p} = \frac{1-\mu(B_r)}{2} \fstop
	\end{equation}
	For every natural number~$n > \frac{2}{1-\mu(B_r)}$, define~$n_1 \coloneqq \bigl\lceil n \mu(B_r) \bigr\rceil$ and $n_2 \coloneqq n-n_1$.
	Since the~$\mu$-measure of the ball~$B_R \coloneqq \set{x \, : \, \norm{x} < R}$ can be estimated with
	\[
		\mu(B_R) \stackrel{\eqref{eq:infinitesimal1}}{\ge} 1-\frac{1-\mu(B_r)}{2} \ge \mu(B_r) + \frac{1-\mu(B_r)}{2} > \mu(B_r) + \frac{1}{n} \comma
	\]
	there exists a measure~$\mu^1$ (dependent on~$n$) such that~$\mu|_{B_r} \le \mu^1 \le \mu|_{B_R}$ and~$\normTV{\mu^1} = n_1/n$. Let~$\mu^2 \coloneqq \mu - \mu^1$. By \Cref{rmk:subaddErr},
	\[
		\eqe{p}{n}^p(\mu) \le \eqe{p}{n_1}^p(\mu^1) + \eqe{p}{n_2}^p(\mu^2) \fstop
	\]
	By \Cref{thm:MerigotMirebeauChevallier}, there exists an infinitesimal function~$f_{p,d}$ such that
	\[
		\eqe{p}{n_1}^p(\mu^1) \le \frac{n_1}{n} R^p f_{p,d}(n_1) \comma
	\]
	and, since~$\mu^2$ is concentrated on~$\R^d \setminus B_r$,
	\[
		\eqe{p}{n_2}^p(\mu^2) \le W_p^p\left( \mu^2, \normTV{\mu^2} \delta_0\right) = \int \norm{x}^p \, \dif \mu^2 \le \int_{\R^d \setminus B_r} \norm{x}^p \, \dif \mu \fstop
	\]
	
	Note that
	\[
		\limsup_{n \to \infty} \frac{n_1}{n} R^p f_{p,d}(n_1) = \mu(B_r) R^p \limsup_{n_1 \to \infty} f_{p,d}(n_1) = 0 \semicolon
	\]
	therefore, we infer that
	\[
		\limsup_{n \to \infty} \eqe{p}{n}^p(\mu) \le \int_{\R^d \setminus B_r} \norm{x}^p \, \dif \mu \comma
	\]
	and we conclude by arbitrariness of~$r$.
\end{proof}

\begin{remark}
	The minimizers of~\eqref{eq:defe} and~\eqref{eq:defetilde} are \emph{not}, in general, unique. For example, let~$\mu$ be invariant under orthogonal transformations and not concentrated at the origin. If~$n$ is large enough, by \Cref{prop:infinitesimal}, no minimizer can be concentrated at the origin; hence, infinitely many orthogonal transformations map any minimizer to \emph{other} minimizers (via pushforward).
\end{remark}

Let us conclude this section with a lemma that relates the classical quantization error and the boundary Wasserstein pseudodistance. This result will be used in the proof of the lower bound~\eqref{eq:mainbelow}.

\begin{lemma} \label{lemma:lowerBoundWb}
	Let~$\Omega$ be an open bounded nonempty subset of~$\R^d$. Choose~$\epsilon > 0$ and define the ``tightened'' open set
	\begin{equation} \label{eq:omegaeps}
		\Omega_\epsilon^- \coloneqq \set{x \in \Omega \, : \, \dist(x,\R^d \setminus \Omega) > \epsilon } \fstop
	\end{equation}
	Fix~$\mu \in \MeasRd{}$. Then, for every~$n \in \N_0$ and~$\mu_n \in \MeasRd{}$ with~$\# \supp (\mu_n|_\Omega)  \le n$, we have
	\begin{equation}
		Wb_{\Omega,p}(\mu, \mu_n) \ge e_{p,n+N}(\mu|_{\Omega_\epsilon^-})  \comma \quad \text{where} \quad 
	N \coloneqq \left \lceil\frac{\sqrt{d}\diam(\Omega)}{\epsilon} \right\rceil^d \fstop
	\end{equation}
\end{lemma}

\begin{proof}
	By considering the vertices of a suitable regular grid, it is easy to check that there exist a set~$\mathcal Y \subseteq \R^d$ with~$\# \mathcal Y \le N$ and a Borel function~$T \colon \Omega_\epsilon^- \to \mathcal Y$ such that~$\norm{x-T(x)} \le \epsilon$ for every~$x \in \Omega_\epsilon^-$. Given~$\mu, \mu_n$ as in the statement, let~$\gamma$ be a nonnegative Borel measure on~$\overline \Omega \times \overline \Omega$ such that~$\gamma|_{\Omega \times \overline \Omega}$ has~$\mu|_\Omega$ as first marginal, and~$\gamma|_{\overline \Omega \times \Omega}$ has~$\mu_n|_\Omega$ as second marginal. Let~$\pi^1 \colon \R^d \times \R^d \to \R^d$ be the projection onto the first~$d$ coordinates, and define
	\[
	\gamma' \coloneqq \gamma|_{\Omega_\epsilon^- \times \Omega} + (\pi^1, T \circ \pi^1)_\# (\gamma|_{\Omega_\epsilon^- \times \partial \Omega}) \fstop
	\]
	Let~$\nu$ be the second marginal of~$\gamma'$. Notice that~$\supp(\nu) \subseteq \supp(\mu_n|_\Omega) \cup \mathcal Y$, which implies~$\# \supp(\nu) \le n + N$. Moreover, since the norm is invariant under pushforward,
	\begin{equation*}
		\normTV{\nu}  = \normTV{\gamma|_{\Omega_\epsilon^- \times \Omega}} + \normTV{\gamma|_{\Omega_\epsilon^- \times \partial \Omega}} = \normTV{\gamma|_{\Omega_\epsilon^- \times \overline \Omega}} = \normTV{\mu|_{\Omega_\epsilon^-}} \fstop
	\end{equation*}
	Consequently,~$e_{p,n+N}(\mu|_{\Omega_\epsilon^-}) \le W_p(\mu|_{\Omega_\epsilon^-}, \nu)$. By noticing that~$\gamma' \in \Gamma(\mu|_{\Omega_\epsilon^-}, \nu)$, we deduce that
	\[
	e^p_{p,n+N}(\mu|_{\Omega_\epsilon^-}) \le \int \norm{x-y}^p \, \dif \gamma|_{\Omega_\epsilon^- \times \Omega} + \int \norm{x-T(x)}^p \, \dif \gamma|_{\Omega_\epsilon^- \times \partial \Omega} \fstop
	\]
	Moreover, by definition of~$T$ and~$\Omega_\epsilon^-$,
	\[
	\int \norm{x-T(x)}^p \, \dif \gamma|_{\Omega_\epsilon^- \times \partial \Omega} \le \int \epsilon^p \, \dif \gamma|_{\Omega_\epsilon^- \times \partial \Omega} \le \int \norm{x-y}^p \, \dif \gamma|_{\Omega_{\epsilon} \times \partial \Omega} \semicolon
	\]
	therefore,
	\[
	e^p_{p,n+N}(\mu|_{\Omega_\epsilon^-}) \le \int \norm{x-y}^p \, \dif \gamma|_{\Omega_\epsilon^- \times \overline \Omega} \le \int \norm{x-y}^p \, \dif \gamma \fstop
	\]
	We conclude by arbitrariness of~$\gamma$.
\end{proof}

\section{Previous results} \label{sec:previous}

There is a rich literature studying asymptotics for classical quantization, see, e.g.,~\cite{GrafLuschgy00,LuschgyPages23}. The following is a fundamental result by P.~Zador~\cite{Zador64,Zador82}.%

\begin{theorem}[{Zador's Theorem,~\cite[Theorem 6.2]{GrafLuschgy00}}] \label{thm:BWZ}
	Let~$\mu \in \ProRd{\theta}$ for some~$\theta > p \ge 1$ and let~$\rho$ be the density of the absolutely continuous part of~$\mu$. Then:
		\begin{equation}
			\lim_{n \to \infty} n^{1/d} e_{p,n}(\mu) = q_{p,d} \left( \int \rho^\frac{d}{d+p} \, \dif \LebRd \right)^\frac{d+p}{dp} \comma
		\end{equation}
	and the optimal quantization coefficient~$q_{p,d}$ is strictly positive.
\end{theorem}
This theorem establishes the exact asymptotic of~$e_{p,n}(\mu)$ as~$n \to 0$ for every~$\mu$ which is not purely singular, under a moment condition (which is not dispensable, see~\cite[Example 6.4]{GrafLuschgy00}).

Less is known about the rate of convergence of~$\eqe{p}{n}(\mu)$. The case~$d=1$ has been studied in \cite{JourdainReygner16,XuBerger19,GilesHefterMayerRitter19,BencheikhJourdain22,BencheikhJourdain22bis,GilesSheridan-Methven23}. %
In particular, the following theorem determines the exact convergence rate under a suitable assumption.

\begin{theorem}[C.~Xu and A.~Berger, {\cite[Theorem~5.15]{XuBerger19}}] \label{thm:xuberger}
	Let~$\mu \in \mathcal P_p(\R)$. If the (upper) quantile function
		\begin{equation}
			F_\mu^{-1} (t) \coloneqq \sup\set{x \in \R \, : \, \mu\bigl((-\infty,x] \bigr) \le t} \comma \qquad t \in (0,1)
		\end{equation}
		is absolutely continuous, then
		\begin{equation} \label{eq:limXuBerger}
			\lim_{n \to \infty} n \eqe{p}{n}(\mu) = q_{p,1} \norm{\frac{\dif F^{-1}_\mu}{\dif t}}_{L^p} \fstop
		\end{equation}
\end{theorem}

For general measures in arbitrary dimension, we have the following theorem, independently proven in~\cite{MerigotMirebeau16} (only for~$p=2$) and~\cite{Chevallier18}.

	\begin{theorem}[Q.~Mérigot and J.-M.~Mirebeau,~{\cite[Proposition~12]{MerigotMirebeau16}}; J.~Chevallier,~{\cite[Theorem~3]{Chevallier18}}] \label{thm:MerigotMirebeauChevallier}
		If~$\mu \in \ProRd{{\mathrm{c}}}$ is supported in~$[-r,r]^d$ for some~$r > 0$, then:
		\begin{equation} \label{eq:thm:MerigotMirebeauChevallier}
			\eqe{p}{n}(\mu) \lesssim_{p,d} r \cdot \begin{cases}
				n^{-1/d} &\text{if } p<d \comma \\
				(1+\log n)^{1/d} n^{-1/d} &\text{if } p = d \comma \\
				n^{-1/p} &\text{if } p > d \comma
			\end{cases}
			\qquad n \in \N_1 \fstop
		\end{equation}
	\end{theorem}
	
	Combined with \Cref{thm:BWZ} and \Cref{rmk:scaling}, this theorem determines the speed of convergence~$\eqe{p}{n}(\mu) \asymp_{p,d,\mu} n^{-1/d}$ in the regime~$p<d$ for every~$\mu$ which is compactly supported and not purely singular:
	\begin{equation} \label{eq:MerigotMirebeauChevallierbis}
		q_{p,d} \left( \int \rho^\frac{d}{d+p} \, \dif \LebRd \right)^\frac{d+p}{dp} \le \liminf_{n \to \infty} n^{1/d} \eqe{p}{n}(\mu) \le \limsup_{n \to \infty} n^{1/d} \eqe{p}{n}(\mu) \lesssim_{p,d} r \comma
	\end{equation}
	where~$\rho$ is as in \Cref{thm:BWZ} and~$r$ is as in \Cref{thm:MerigotMirebeauChevallier}.
	
	We note that also for~$p > d$ the upper bound~\eqref{eq:thm:MerigotMirebeauChevallier} is tight, in the sense that there exist compactly supported measures---even absolutely continuous and with smooth densities---for which
	\begin{equation} \label{eq:disconnected}
		\limsup_{n \to \infty} n^{1/p} \eqe{p}{n}(\mu) > 0 \comma
	\end{equation}
	as demonstrated by the following example; see also the~$1$-dimensional case in~\cite[Remark~5.22]{XuBerger19}.
	
	\begin{example} \label{ex:distant}
		Assume that~$\mu \in \ProRd{p}$ is concentrated on the union of two distant sets~$A,B \subseteq \R^d$ with~$\mu(A)>0$ and~$\mu(B) > 0$. Then~\eqref{eq:disconnected} holds.
	\end{example}
	\begin{proof} Le us write~$r \coloneqq \dist(A,B)$. Define
		\[
		\tilde A \coloneqq \set{x \in \R^d \, : \, \dist(x,A) \le \dist(x,B)} \comma \quad \tilde B \coloneqq \R^d \setminus \tilde A \fstop
		\]
		Given~$n \in \N_1$, take any~$\mu_n \in \ProRd{(n)}$ and~$\gamma \in \Gamma(\mu, \mu_n)$. We have
		\begin{align*}
			\int \norm{x-y}^p \, \dif \gamma &\ge \left(\frac{r}{2}\right)^p \left( \gamma(A \times \tilde B) + \gamma(B \times \tilde A) \right) \ge \left(\frac{r}{2}\right)^p \abs{ \gamma(A \times \tilde B) - \gamma(B \times \tilde A) } \\
			&= \left(\frac{r}{2}\right)^p \abs{ \mu_n(\tilde B)-\gamma(B \times \tilde B) - \gamma(B \times \tilde A) } = \left(\frac{r}{2}\right)^p \abs{ \mu_n(\tilde B)-\mu(B) } \fstop
		\end{align*}
		Let us denote by~$\tau(n)$ the fractional part of~$n \mu(B)$. Since~$n \mu_n(\tilde B) \in \N_0$, we get
		\[
		\int \norm{x-y}^p \, \dif \gamma \ge \frac{1}{n} \left(\frac{r}{2}\right)^p \min\bigl(\tau(n), 1-\tau(n)\bigr) \comma
		\]
		and, by arbitrariness of~$\gamma$ and~$\mu_n$, we find
		\[
		n^{1/p}\eqe{p}{n}(\mu) \ge \frac{r}{2} \min\bigl(\tau(n), 1-\tau(n) \bigr)^{1/p} \comma \qquad n \in \N_1 \fstop
		\]
		To conclude~\eqref{eq:disconnected}, it suffices to prove that, for infinitely many numbers~$n \in \N_1$, we have~$\tau(n) \in [1/3,2/3]$. Firstly, we note that
		\[
		\tau(n) = 0 \quad \Rightarrow \quad \tau(n+1) = \mu(B) \in (0,1) \semicolon
		\]
		hence~$\tau(n) \in (0,1)$ frequently. Finally, it is easy to check that
		\[
		\tau(n) \in (0,1) \setminus [1/3,2/3] \quad \Rightarrow \quad \tau\left( \left \lceil\frac{1}{ 3 \min\bigl(\tau(n), 1-\tau(n)\bigr) } \right \rceil n\right) \in [1/3,2/3] \fstop \qedhere
		\]
	\end{proof}

	For~$p = d$, it is still unknown whether the logarithmic term in~\eqref{eq:thm:MerigotMirebeauChevallier} is necessary (for compactly supported measures), see \cite[Remark~1]{Chevallier18}.

In addition to \Cref{thm:MerigotMirebeauChevallier}, we mention the results in~\cite{GilesHefterMayerRitter19,GilesHefterMayerRitter19bis}, applicable to certain measures in infinite-dimensional Banach spaces,~\cite{BrownSteinerberger20} for the volume measure on a compact manifold, and the upper bounds that can be deduced from the theory of random empirical quantization~\cite{Ledoux23} using the trivial inequality
\[
	\eqe{p}{n}(\mu) \le \mathbb E\left[ W_p\left(\mu, \frac{1}{n} \sum_{i=1}^n \delta_{X_i} \right)\right]\comma
\]
valid for every family of random variables~$\set{X_1,\dots,X_n}$. In particular, the following theorems already provide an upper estimate of the form~\eqref{eq:mainabove} and a nonasymptotic upper bound like~\Cref{thm:nonasy} \emph{in the regime~$p < d/2$}.

\begin{theorem}[{S.~Dereich, M.~Scheutzow, and R.~Schottstedt,~\cite[Theorem~2]{DereichScheutzowSchottstedt13}}] \label{thm:DereichScheutzowSchottstedt}
	Under the assumptions of \Cref{thm:main1}, further suppose that~$p < d/2$, and that~$\rho$ is Riemann integrable or~$p=1$. If~$X_1,X_2,\dots$ is a sequence of~$\mu$-distributed i.i.d.~random variables, then
	\begin{equation} \label{eq:thm:DereichScheutzowSchottstedt}
		\mathbb E\left[ W_p^p\left(\mu, \frac{1}{n} \sum_{i=1}^n \delta_{X_i} \right)\right]^{1/p} \asymp_{p,d}n^{-1/d} \left(\int_{\R^d} \rho^\frac{d-p}{d} \, \dif \LebRd \right)^{1/p} \quad \text{as } n \to \infty \fstop
	\end{equation}
\end{theorem}

\begin{theorem}[{S.~Dereich, M.~Scheutzow, and R.~Schottstedt,~\cite[Theorem~1]{DereichScheutzowSchottstedt13}}] \label{thm:DereichScheutzowSchottstedtPierce}
	Under the assumptions of \Cref{thm:main1}, further suppose that~$p < d/2$. If~$X_1,X_2,\dots$ is a sequence of~$\mu$-distributed i.i.d.~random variables, then
	\begin{equation}
		\mathbb E\left[ W_p^p\left(\mu, \frac{1}{n} \sum_{i=1}^n \delta_{X_i} \right)\right]^{1/p} \lesssim_{p,d,\theta} n^{-1/d} \left( \int \norm{x}^\theta \, \dif \mu \right)^{1/\theta} \comma \qquad n \in \N_1 \fstop
	\end{equation}
\end{theorem}

In a recent preprint, E.~Caglioti, M.~Goldman, F.~Pieroni, and D.~Trevisan~\cite{CagliotiGoldmanPieroniTrevisan24} extended~\cite[Theorem~2]{DereichScheutzowSchottstedt13} to~$p \le d$ (with some modifications when~$p = d$) for~$\mu$ in a certain class of radially symmetric and rapidly decaying probability laws, including, e.g.,~the normal distribution.
	
As noted in the introduction, it remains unknown whether any of the hidden constants appearing in~\eqref{eq:thm:DereichScheutzowSchottstedt} and in~\cite[Theorem~1.6]{CagliotiGoldmanPieroniTrevisan24} coincides with~$q_{p,d}$ or~$\eqc{p}{d}$.

	With our \Cref{thm:main1} and \Cref{thm:nonasy}, we obtain several improvements over what was previously known:
	\begin{itemize}
		\item We establish the speed of convergence~$\eqe{p}{n}(\mu) \asymp_{p,d,\mu} n^{-1/d}$ for general (not purely singular) measures in the whole range~$p \in [1,d)$, under a moment condition, but without assuming compactness of the support or Riemann integrability of the density.
		\item We prove a nonasymptotic upper bound also for~$d/2 \le p < d$ without assuming compactness of the support.
		\item We find the \emph{same explicit} dependence on the measure in the asymptotic upper and lower bounds~\eqref{eq:mainbelow} and~\eqref{eq:mainabove} (assuming~$\mu^a(\supp (\mu^s))=0$, where~$\mu^a$ and~$\mu^s$ are the absolutely continuous and singular parts of~$\mu$, respectively).
		\item We establish the asymptotic upper bound with the constant~$\eqc{p}{d}$, which is optimal, since~\eqref{eq:mainabove} is an equality for~$\mu = U_d$.
	\end{itemize}
	Furthermore, we determine the existence of the limit in some instances (\Cref{sec:existenceLimit}) and we find the speed of convergence~$\eqe{p}{n}(\mu) \asymp_{p,d,\mu} n^{-1/d}$ for every~$p \ge 1$ for a certain class of measures (\Cref{cor:pGd2}).

\section{Nonasymptotic upper bound (\Cref{thm:nonasy})} \label{sec:nonasy}

The proof of \Cref{thm:nonasy} is similar to those of its counterpart for compactly supported measures in \cite{MerigotMirebeau16,Chevallier18} (\Cref{thm:MerigotMirebeauChevallier}). Iteratively $n$ times, we extract from the given measure~$\mu$ a sufficiently concentrated subprobability with mass equal to~$1/n$. In~\cite{MerigotMirebeau16,Chevallier18}, where~$\mu$ is compactly supported, the subprobabilities are found by splitting the support into a finite number of pieces (of small, comparable size) and applying a pigeonhole-like principle: the measure of at least one of these pieces is sufficiently large. Since our measure is not compactly supported, we use the moment condition to first identify, at each iteration, a compact region (small relative to the moment of~$\mu$) where enough mass is concentrated; we then proceed as before.

\begin{proof}[Proof of~\Cref{thm:nonasy}]
	Fix~$n \in \N_1$, let~$M \coloneqq \int \norm{x}^{\theta} \, \dif \mu$,	and define
	\[
	r_k \coloneqq \left( \frac{2nM}{n-k} \right)^\frac{1}{\theta} \comma \qquad k \in \set{1,2,\dots,n-1} \fstop
	\]
	With this choice, for every~$k$ we have
	\begin{align} \label{eq:thm:nonasy:1}
		\begin{split}
			\mu\bigl( [-r_k,r_k]^d \bigr) &= 1 - \int_{\R^d \setminus [-r_k,r_k]^d} \, \dif \mu \ge 1-r_k^{-\theta} \int_{\R^d \setminus [-r_k,r_k]^d} \norm{x}^{\theta} \dif \mu \\
			&\ge 1- \frac{n-k}{2nM} M = \frac{1}{2} + \frac{k}{2n} \fstop
		\end{split}
	\end{align}
	Using \cite[Lemma~1]{Chevallier18}, we now argue as in the proof of \cite[Theorem~2]{Chevallier18}. By applying this lemma to the measure~$\nu_1 \coloneqq \mu|_{[-r_1,r_1]^d}$, we find a measure~$\eta_1 \le \nu_1$ with total mass~$\normTV{\eta_1} = 1/n$ and supported on a set with diameter bounded by~$4\sqrt{d}\,r_1(n\normTV{\nu_1})^{-1/d}$. We repeat this procedure with~$\nu_2 \coloneqq \mu|_{[-r_2,r_2]^d} - \eta_1$ to find~$\eta_2$, then with~$\nu_3 \coloneqq \mu|_{[-r_3,r_3]^d} - \eta_1 - \eta_2$, and so on. At the end, we have a family of measures~$\eta_1,\eta_2,\dots,\eta_{n-1}$ with
	\[
	\eta_1+\eta_2+\cdots+\eta_{n-1} \le \mu \quad \text{and} \quad \normTV{\eta_1}=\normTV{\eta_2} = \cdots = \normTV{\eta_{n-1}} = \frac{1}{n} \comma
	\]
	and
	\begin{align} \label{eq:thm:nonasy:2}
		\begin{split}
			\diam (\supp (\eta_k)) &\lesssim_d r_k(n \normTV{\nu_k})^{-1/d} = r_k\left(n \mu \bigl( [-r_k,r_k]^d \bigr) - (k-1) \right)^{-1/d} \\
			&\stackrel{\eqref{eq:thm:nonasy:1}}{\le} r_k \left( \frac{n-k}{2} \right)^{-1/d}
		\end{split}
	\end{align}
	for every~$k \in \set{1,2,\dots,n-1}$.
	
	Let us pick a point~$x_k$ from each~$\supp (\eta_k)$. After defining~$\eta_n \coloneqq \mu - \eta_1 - \cdots \eta_{n-1}$, \Cref{lemma:subadditivityWasserstein} gives
	\begin{align} \label{eq:thm:nonasy:3}
		\begin{split}
			\eqe{p}{n}^p(\mu) &\le W_p^p(\eta_1+\cdots+\eta_{n-1}+\eta_n, n^{-1} \bigl(\delta_{x_1}+\cdots+\delta_{x_{n-1}} + \delta_0)\bigr) \\
			&\le W_p^p(\eta_n, n^{-1} \delta_0) + \sum_{k=1}^{n-1} W_p^p(\eta_k, n^{-1} \delta_{x_k}) \\
			&\le \int \norm{x}^p \, \dif \eta_n + \sum_{k=1}^{n-1} \frac{\bigl(\diam (\supp (\eta_k))\bigr)^p}{n} \fstop
		\end{split}
	\end{align}
	H\"older's inequality and the fact that~$\theta > p^*$ yield
	\begin{equation} \label{eq:thm:nonasy:4}
		\int \norm{x}^p \, \dif \eta_n \le M^{p/\theta} \normTV{\eta_n}^{1-\frac{p}{\theta}} \le M^{p/\theta} n^{\frac{p}{\theta}-1} \le M^{p/\theta} n^{-p/d} \fstop
	\end{equation}
	Moreover,
	\begin{align} \label{eq:thm:nonasy:5}
		\begin{split}
			\sum_{k=1}^{n-1} \frac{\bigl(\diam (\supp (\eta_k))\bigr)^p}{n} &\stackrel{\eqref{eq:thm:nonasy:2}}{\lesssim_{p,d,\theta}} M^{p/\theta} n^{\frac{p}{\theta}-1} \sum_{k=1}^{n-1} (n-k)^{-\frac{p}{\theta} - \frac{p}{d}} \lesssim_{p,d,\theta} M^{p/\theta} n^{-p/d} \comma
		\end{split}
	\end{align}
	where, in the last inequality, we used that~$\frac{p}{\theta}+\frac{p}{d} < 1$.
	We conclude by combining~\eqref{eq:thm:nonasy:3},~\eqref{eq:thm:nonasy:4}, and~\eqref{eq:thm:nonasy:5}.
\end{proof}

\section{Uniform measure on a cube} \label{sec:cube}

In this section, we establish the limiting behavior of the optimal empirical quantization error for the uniform measure~$U_d$ on~$[0,1]^d$.

\begin{proposition} \label{prop:uniformCube}
	For every~$p \ge 1$, we have the identity
	\begin{equation} \label{eq:uniformCube}
		\lim_{n \to \infty} n^{1/d} \eqe{p}{n}(U_d) = \inf_{n \in  \N_1} n^{1/d} \eqe{p}{n}(U_d) \eqqcolon \eqc{p}{d} \fstop
	\end{equation}
\end{proposition}

Note that this proposition applies also when~$p \ge d$.

\Cref{prop:uniformCube} is easy to prove in dimension~$d=1$, see \Cref{ex:1dim} below. Moreover, exploiting the self-similarity of the cube, we can build a simple ``scale-and-copy'' argument (\Cref{lemma:scaling}, inspired by~\cite[Step 1 in Theorem~6.2]{GrafLuschgy00}, see also \Cref{fig:scaleCopy}) to write
\[
	\eqc{p}{d} = \inf_{m \in \N_1} \limsup_{k \to \infty} k m^{1/d} \eqe{p}{k^dm}(U_d) \fstop
\]
In order to prove that
\[
	\limsup_{n \to \infty} n^{1/d} \eqe{p}{n}(U_d) \le \inf_{m \in \N_1} \limsup_{k \to \infty} k m^{1/d} \eqe{p}{k^dm}(U_d) \comma
\]
we estimate the increase rate of the function~$n \mapsto \eqe{p}{n}(U_d)$: given~$n$ and~$m$, we want to know how far~$n^{1/d}\eqe{p}{n}(U_d)$ is from the sequence~$k \mapsto k m^{1/d} \eqe{p}{k^dm}(U_d)$. The bound on the increase rate is proven in \Cref{lemma:errorInductiveStep} by constructing a suitable competitor for the minimization problem that defines~$\eqe{p}{n}(U_d)$. This competitor is built by combining two optimal empirical quantizers: one for~$U_d$ and one for the uniform measure~$U_{d-1}$ on the~$(d-1)$-dimensional (!) cube. In the end, this procedure shifts the problem to estimating the optimal empirical quantization error for the uniform measure on a \emph{lower dimensional} cube. In fact,~\eqref{eq:uniformCube} is proven \emph{by induction} on the dimension.

\begin{remark}
	While it is obvious that~$n \mapsto e_{p,n}(\mu)$ is nonincreasing, the same cannot in general be said for the optimal empirical quantization error. For example, if~$\mu = \frac{\delta_x + \delta_y}{2}$ for two distinct points~$x,y \in \R^d$, then~$\eqe{p}{2}(\mu) = 0$ but~$\eqe{p}{3}(\mu) > 0$.
\end{remark}

\begin{remark}[$1$-dimensional case] \label{ex:1dim}
	The values of~$e_{p,n}(U_1)$ and~$\eqe{p}{n}(U_1)$ are easy to compute and coincide. For both the problems, the optimal measure~$\mu_n \in \ProRd{(n)}$ and the optimal transport plan~$\gamma \in \Gamma(U_1, \mu_n)$ are simply:
	\[
	\mu_n \coloneqq \frac{1}{n} \sum_{i=1}^n \delta_{\frac{2i-1}{2n}} \comma \quad \gamma \coloneqq \sum_{i=1}^n \mathscr{L}^1|_{\left(\frac{i-1}{n},\frac{i}{n}\right)} \otimes \delta_{\frac{2i-1}{2n}} \in \Gamma(U_1,\mu_n) \comma
	\]
	see~\cite[Theorem~4.16]{GrafLuschgy00},~\cite[Theorem~5.5]{XuBerger19}, and \Cref{fig:1dim}.
Hence,
	\[
	e_{p,n}^p(U_1) = \eqe{p}{n}^p(U_1) = \sum_{i=1}^n \int_{\frac{i-1}{n}}^\frac{i}{n} \abs{x-\frac{2i-1}{n}}^p \, \dif x = \frac{1}{(p+1)(2n)^p} \fstop
	\]
	\begin{figure}
		\centering
		\includegraphics[width = .5\textwidth]{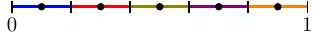}
		\caption{The points on which the optimal measure~$\mu_n$ for~$U_1$ is concentrated are evenly separated on~$[0,1]$.}
		\label{fig:1dim}
	\end{figure}
\end{remark}

\begin{proof}[Proof of~\Cref{prop:uniformCube}]
For simplicity, we write~$\eqeCube{p}{n}{d}$ in place of~$\eqe{p}{n}(U_d)$. The proof is by induction on the dimension~$d$. Base step: By \Cref{ex:1dim},~$n \eqeCube{p}{n}{1}$ is constantly equal to~$\frac{1}{2\sqrt[p]{p+1}}$.

For the inductive step, we make use of two lemmas. %

\begin{lemma} \label{lemma:scaling}
	For every~$m,k \in \N_1$, we have the inequality
	\begin{equation}
		\eqeCube{p}{k^d n}{d} \le \frac{1}{k} \eqeCube{p}{n}{d} \fstop
	\end{equation}
\end{lemma}

\begin{proof}
	Let~$\mu_m \in \ProRd{(m)}$ and~$\gamma \in \Gamma(U_d, \mu_m)$. For every~$i \in \set{0,1,\dots,k-1}^d$, we define
	\[
		T_i(x) \coloneqq \frac{1}{k}(i+x) \comma \qquad x \in [0,1]^d \fstop
	\]
	Notice that~$T_i$ maps~$[0,1]^d$ to~$i/k+[0,1/k]^d$. The idea is to use these transformations to make smaller copies of~$\mu_n$, which, together, constitute an appropriate competitor for the infimum that defines~$e_{p,k^dm,d}$. Precisely, we set
	\begin{align*}
		\mu_{k^d m}' &\coloneqq \frac{1}{k^d} \sum_{i} (T_i)_\# \mu_m \in \ProRd{(k^dm)} \comma \\
		\gamma' &\coloneqq \frac{1}{k^d} \sum_{i} (T_i,T_i)_\# \gamma \in \Gamma(U_d,\mu_{k^dm}') \fstop
	\end{align*}
	With these choices, we obtain
	\begin{align*}
		\eqeCube{p}{k^dm}{d}^p &\le \int \norm{x-y}^p \, \dif \gamma' = \frac{1}{k^d} \sum_{i} \int \abs{T_i(x)-T_i(y)}^p \, \dif \gamma \\ &=  \frac{1}{k^{p+d}} \sum_{i} \int \norm{x-y}^p \, \dif \gamma = \frac{1}{k^p} \int \norm{x-y}^p \, \dif \gamma \fstop
	\end{align*}
	We conclude by arbitrariness of~$\gamma$ and~$\mu_m$.
\end{proof}

\begin{figure}
	\centering
	\includegraphics[width=0.9\textwidth]{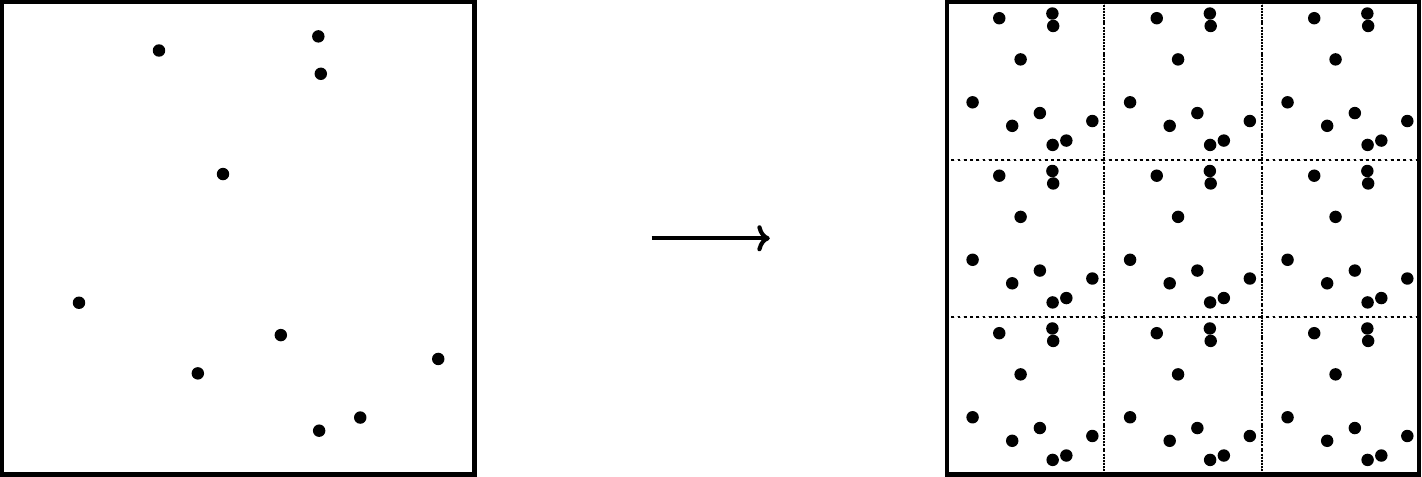}
	\caption{Idea for \Cref{lemma:scaling}. Given the measure~$\mu_n$, concentrated on the black dots in the left square, the competitor~$\mu_{k^dn}'$ is built by making~$k^d$ rescaled copies of~$\mu_n$.}
	\label{fig:scaleCopy}
\end{figure}

\begin{lemma} \label{lemma:errorInductiveStep}
	There exists a constant~$c_{p} > 0$ such that, for every~$n,l \in \N_1$, we have
	\begin{equation}
		\eqeCube{p}{n+l}{d+1}^p \le \frac{n}{n+l} \eqeCube{p}{n}{d+1}^p + c_{p} \frac{l}{n+l} \eqeCube{p}{l}{d}^p + c_{p} \left(\frac{l}{n+l}\right)^{p+1} \fstop
	\end{equation}
\end{lemma}

\begin{proof}
	Let~$\mu_n, \nu_l$ be probability measures of the form
	\[
		\mu_n = \frac{1}{n} \sum_{i=1}^n \delta_{(x_i, t_i)} \in \mathcal{P}_{(n)}(\R^{d+1}) \comma \qquad \nu_l = \frac{1}{l} \sum_{i=n+1}^{n+l} \delta_{x_i} \in \ProRd{(l)},
	\]
	for some~$x_1,\dotsc,x_n,x_{n+1},\dotsc,x_{n+l} \in \R^{d}$ and~$t_1,\dotsc,t_n \in \R$. Pick~$\gamma \in \Gamma(U_{d+1}, \mu_n)$ and~$\eta \in \Gamma(U_d,\nu_l)$. Consider the linear $1$-Lipschitz function~$T \colon \R^{d+1} \to \R^{d+1}$ given by the formula
	\[
		T(x,t) \coloneqq \left(x, \frac{n}{n+l} t \right) \comma \qquad x \in \R^d \comma \quad t \in \R \comma
	\]
	and define~$\gamma' \in \mathcal{P}(\R^{d+1} \times \R^{d+1})$ via
	\begin{equation*}
		\int \varphi \, \dif \gamma' \coloneqq \frac{n}{n+l} \int \varphi\bigl(T(x,t),T(x',t')\bigr) \, \dif \gamma + \int_{\frac{n}{n+l}}^1 \int \varphi(x,t,x',1) \, \dif \eta(x,x') \, \dif t \comma
	\end{equation*}
	for every continuous and bounded test function~$\varphi \colon \R^{d+1} \times \R^{d+1} \to \R$.
	\begin{figure}
		\centering
		\includegraphics[width=0.9\textwidth]{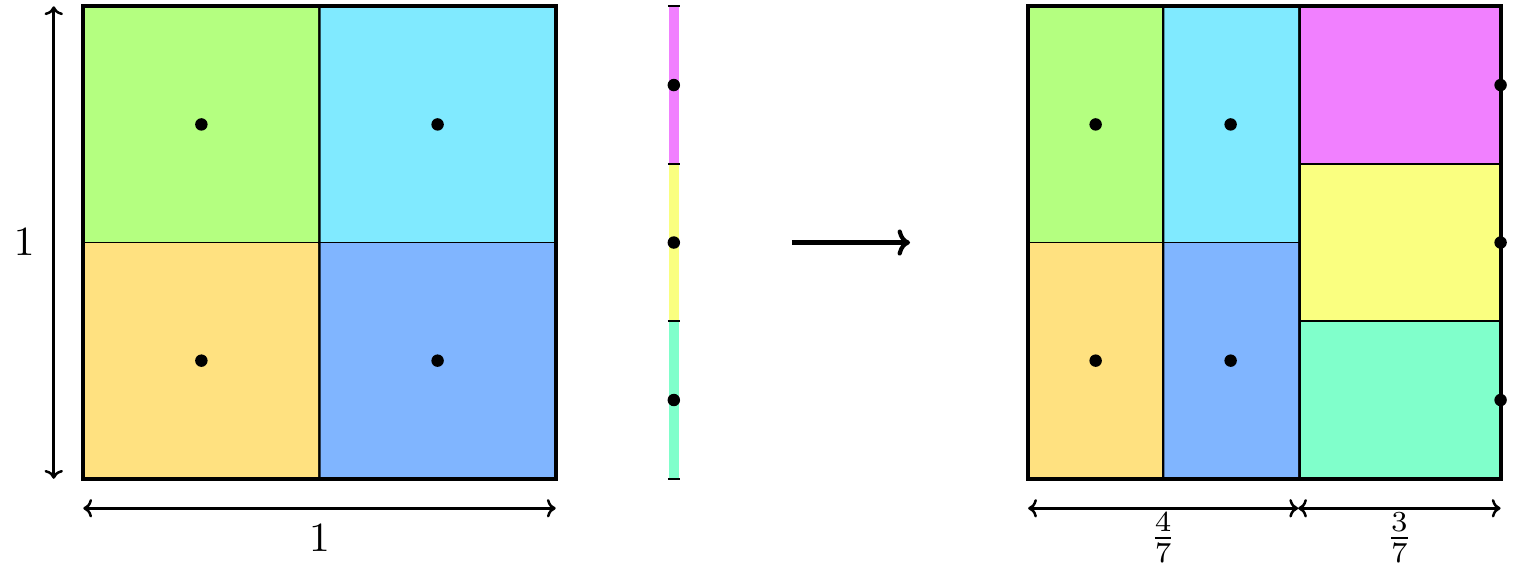}
		\caption{Idea for \Cref{lemma:errorInductiveStep} with~$d=1, n=4, l=3$. From a transport plan for~$U_{d+1}$ with~$n$ points and one for~$U_d$ with~$l$ points, we construct a new plan for~$U_{d+1}$ with~$n+l$ points by ``shrinking'' the first one and ``expanding'' the second one.}
	\end{figure}
	It is not difficult to check that~the first marginal of $\gamma'$ is~$U_{d+1}$. Indeed, given a test function~$\psi \colon \R^{d+1} \to \R$ and denoting by~$\pi^1 \colon \R^{d+1} \times \R^{d+1} \to \R^{d+1}$ the projection onto the first~$d+1$ coordinates, we have
	\begin{align*}
		\int \psi \, \dif \pi^1_\# \gamma' &= \frac{n}{n+l} \int \psi \bigl( T(x,t) \bigr) \, \dif \pi^1_\# \gamma + \int_{\frac{n}{n+l}}^1 \int \psi(x,t) \, \dif \pi^1_\# \eta(x) \, \dif t \\
		&= \frac{n}{n+l} \int_0^1 \int \psi\left(x,\frac{n t}{n+l}  \right) \, \dif U_d(x) \, \dif t + \int_{\frac{n}{n+l}}^1 \! \int \psi(x,t) \, \dif U_d(x) \, \dif t \\
		&= \int \psi \, \dif U_{d+1} \fstop
	\end{align*}
	The second marginal is
	\[
		\pi^2_\# \gamma' = \frac{1}{n+l} \sum_{i=1}^n \delta_{T(x_i,t_i)} + \frac{1}{n+l} \sum_{i = n+1}^{n+l} \delta_{(x_i,1)} \in \ProRd{(n+l)} \comma
	\]
	because
	\begin{align*}
		\int \psi \, \dif \pi^2_\# \gamma' &= \frac{n}{n+l} \int \psi\bigl( T(x,t) \bigr) \, \dif \pi^2_\# \gamma + \int_{\frac{n}{n+l}}^1 \int \psi(x',1) \, \dif \pi^2_\# \eta(x') \, \dif t \\
		&= \frac{\cancel{n}}{n+l} \frac{1}{\cancel{n}} \sum_{i=1}^n \psi\bigl( T(x_i,t_i) \bigr) + \frac{\cancel{l}}{n+l} \frac{1}{\cancel{l}} \sum_{i = n+1}^{n+l} \psi(x_i,1) \fstop
	\end{align*}
	We infer the inequality
	\begin{multline*}
		\eqeCube{p}{n+l}{d+1}^p \le  \frac{n}{n+l} \int \norm{T(x,t)-T(x',t')}^p \, \dif \gamma
		\\+ \int_{\frac{n}{n+l}}^1 \int \norm{(x,t)-(x',1)}^p \, \dif \eta(x,x') \, \dif t \fstop
	\end{multline*}
	Now we make the following two observations:
	\begin{enumerate}
		\item since~$T$ is~$1$-Lipschitz,~$\norm{T(x,t)-T(x',t')} \le \norm{(x,t)-(x',t')}$,
		\item there exists a constant~$c_{p}$ such that \[ \norm{(x,t)-(x',t')}^p \le c_{p} \left( \norm{x-x'}^p + \abs{t - t'}^p \right)\] for every~$x,x' \in \R^d$ and~$t, t' \in \R$. Precisely,~$c_p = \max\bigl(1,2^{\frac{p}{2}-1}\bigr)$.
	\end{enumerate}
	Therefore, we obtain
	\begin{multline*}
		\eqeCube{p}{n+l}{d+1}^p \le \frac{n}{n+l} \int \norm{(x,t) - (x', t')}^p \, \dif \gamma + c_p \frac{l}{n+l} \int \norm{x-x'}^p \, \dif \eta \\ + c_p \int_{\frac{n}{n+l}}^1 (1-t)^p \, \dif t \fstop
	\end{multline*}
	By arbitrariness of~$\mu_n,\nu_l,\gamma,\eta$,
	\[
		\eqeCube{p}{n+l}{d+1}^p \le \frac{n}{n+l} \eqeCube{p}{n}{d+1}^p + c_p \frac{l}{n+l} \eqeCube{p}{l}{d}^p + c_p \int_{\frac{n}{n+l}}^1 (1-t)^p \, \dif t \comma
	\]
	and the conclusion follows.
\end{proof}

Assume that~\eqref{eq:uniformCube} is true for a certain dimension~$d$, and fix~$m \in \N_1$. For every~$n \ge 2^{d+1} m$, set
\[
	k_n \coloneqq \left \lfloor \left ( \frac{n}{m} \right)^{\frac{1}{d+1}} \right \rfloor - 1 \comma \quad l_n \coloneqq n - k_n^{d+1} m \fstop
\]
Observe that~$k_n, l_n \ge 1$ for every~$n$ (they are integer and strictly positive), and that~$l_n \asymp_{d,m} n^{\frac{d}{d+1}}$. Indeed, on the one hand,
\[
	\frac{n}{m} \le \left( \left\lfloor \left ( \frac{n}{m} \right)^{\frac{1}{d+1}} \right \rfloor +1 \right)^{d+1} = (k_n+2)^{d+1} \comma
\]
from which we get
\[
	l_n \le \bigl( (k_n+2)^{d+1} - k_n^{d+1} \bigr) m \lesssim_d k_n^{d} m \le n^{\frac{d}{d+1}} m^\frac{1}{d+1} \fstop
\]
On the other hand,
\[
	l_n \ge \bigl((k_n+1)^{d+1}-k_n^{d+1} \bigr)m \gtrsim_d (k_n+2)^{d}m \ge n^{\frac{d}{d+1}} m^\frac{1}{d+1} \fstop
\]
\Cref{lemma:errorInductiveStep} gives the estimate
\[
	\eqeCube{p}{n}{d+1}^p - \eqeCube{p}{k_n^{d+1}m}{d+1}^p \lesssim_p \frac{l_n}{n} \eqeCube{p}{l_n}{d}^p + \left( \frac{l_n}{n} \right)^{p+1} \comma
\]
and, by inductive hypothesis,~$\eqeCube{p}{l_n}{d}^p \lesssim_{p,d} l_n^{-p/d}$. Thus,
\[
	\eqeCube{p}{n}{d+1}^p - \eqeCube{p}{k_n^{d+1} m}{d+1}^p \lesssim_{p,d} \frac{l_n^{1-\frac{p}{d}}}{n}+\left( \frac{l_n}{n} \right)^{p+1} \fstop
\]
The combination of the latter with~$l_n \asymp_{d,m} n^{\frac{d}{d+1}}$ gives
\[
	\limsup_{n \to \infty} n^{\frac{p}{d+1}} \eqeCube{p}{n}{d+1}^p \le \limsup_{n \to \infty} n^{\frac{p}{d+1}} \eqeCube{p}{k_n^{d+1} m}{d+1}^p \fstop
\]
Now we use \Cref{lemma:scaling} to write
\[
	\limsup_{n \to \infty} n^{\frac{1}{d+1}} \eqeCube{p}{n}{d+1} \le \eqeCube{p}{m}{d+1} \limsup_{n \to \infty} \frac{n^\frac{1}{d+1}}{k_n} = m^\frac{1}{d+1} \eqeCube{p}{m}{d+1} \fstop
\]
We conclude the inductive step (and therefore the proof) by arbitrariness of~$m$.
\end{proof}

\section{Asymptotic behavior for~$p \in [1,\infty)$ (\Cref{cor:pGd2})} \label{sec:corollaries}

This section is devoted to \Cref{cor:pGd2}. Note that this result will \emph{not} be used later in this work. The following simple observation is at the core of the proof.

\begin{remark} \label{cor:lipPF}
	The property
	\begin{equation} \label{eq:uppermu}
		\limsup_{n \to \infty}  n^{1/d} \eqe{p}{n}(\mu) < \infty 
	\end{equation}
	is invariant under pushforward via Lipschitz maps. In particular, if~$T \colon [0,1]^d \to \R^d$ is Lipschitz, then~\eqref{eq:uppermu} holds with~$\mu \coloneqq T_\# U_d$ for every~$p \ge 1$.
\end{remark}

\begin{proof}[Proof of \Cref{cor:pGd2}] \emph{Step 1 ($\Omega = \tilde \Omega$).} Assume at first that~$\Omega = \tilde \Omega$, i.e.,~$\Omega$ itself is convex and with~$C^{1,1}$ boundary. The idea is to use the regularity theory for optimal transport to find a Lipschitz map~$T$ such that~$\mu = T_\# U_d$ in order to apply \Cref{cor:lipPF}. Precisely, we use~\cite[Theorem~1.1]{ChenLiuWang21} (see also~\cite[Theorem~1.1~(i)]{ChenLiuWang19}): given a measure~$\mu_0 = \rho_0 \LebRd$ concentrated on an open set~$\Omega_0$, with the same assumptions as~$\rho$ and~$\Omega$, there exists a Lipschitz transport map\footnote{In fact, the map~$T_1$ is of class~$C^1$ with H\"older Jacobian. Since~$\Omega_0$ is convex,~$T_1$ is Lipschitz.}~$T_1$ pushing~$\mu_0$ to~$\mu$. If we manage to find one such~$\mu_0$ of the form~$\mu_0 = (T_0)_\# U_d$ for some~Lipschitz map $T_0$, then we can set~$T \coloneqq T_1 \circ T_0$ and conclude. The obstruction to simply taking~$\mu_0=U_d$ is that the boundary of~$[0,1]^d$ is not of class~$C^{1,1}$. Let us also note that it makes no difference if we find~$\mu_0$ as Lipschitz pushforward of the uniform measure on \emph{another} $d$-dimensional cube, such as the unit ball w.r.t.~$1$-norm~$\norm{\cdot}_1$.
	
	In light of the previous discussion, proving the following lemma suffices to complete this Step.
	\begin{lemma}
	The map%
		\[
			T_0(x) \coloneqq \left( 1 - \norm{x}_1 + \frac{\norm{x}_1^2}{\norm{x}_2} \right)x \comma \qquad \norm{x}_1 < 1 \comma
		\]
		is Lipschitz continuous. Moreover, the measure~$\mu_0 \coloneqq (T_0)_\# U\bigl( \set{\norm{\cdot}_1 < 1} \bigr)$ is concentrated on the Euclidean ball~$\Omega_0 \coloneqq \set{\norm{\cdot}_2 < 1}$ and, therein, it has Lipschitz continuous and uniformly positive density.
	\end{lemma}
	\begin{proof}
	We omit the simple proofs that~$ T_0$ is Lipschitz and that~$\mu_0$ is concentrated on~$\Omega_0$, and focus on the computation of the density of~$\mu_0$. Let~$\varphi \colon \Omega_0 \to \R$ be a Borel measurable and bounded test function. We have
	\begin{align*}
		\int \varphi \, \dif \mu_0 = \frac{1}{c_d} \int_{\Sph^{d-1}} \int_0^{\norm{v}_1^{-1}} \! \! \! \varphi\bigl(  T_0(rv) \bigr) r^{d-1} \, \dif r \dif \mathscr{H}^{d-1}(v) \comma
	\end{align*}
	where~$c_d \coloneqq \abs{\set{\norm{\cdot}_1 < 1}}= \frac{2^d}{d!}$, the set~$\Sph^{d-1}$ is the~$(d-1)$-dimensional sphere (w.r.t.~the~$2$-norm), and~$\mathscr{H}^{d-1}$ is the~$(d-1)$-dimensional Hausdorff measure on it. Let us write
	\[
		T_0(rv) = \underbrace{r \left( 1 - r\norm{v}_1 + r \norm{v}_1^2  \right)}_{\eqqcolon \xi_v(r)} v \comma \qquad r \in \left(0, \norm{v}_1^{-1} \right) \comma \quad v \in \Sph^{d-1} \comma
	\]
	and notice that
	\[
		(\partial_r \xi_v)(r)  = 1 + 2r \norm{v}_1 \bigl( \norm{v}_1-1 \bigr) \ge 1 \comma
	\]
	where, in the last inequality, we used that~$\norm{\cdot}_2 \le \norm{\cdot}_1$. In particular,~$\xi_v$ is invertible. Thus, by changing variables, we find
	\[
		\int \varphi \, \dif \mu_0 = \frac{1}{c_d} \int_{\Sph^{d-1}} \int_0^1 \varphi(\tilde r v) \frac{\xi_v^{-1}(\tilde r)^{d-1}}{(\partial_r \xi_v)\bigl(  \xi_v^{-1}(\tilde  r) \bigr)} \dif \tilde r \, \dif \mathscr{H}^{d-1}(v) \comma
	\]
	and, therefore, the density of~$\mu_0$ on~$\Omega_0$ is
	\[
		\rho_0(x) \coloneqq \frac{\xi_{v_x}^{-1}\bigl(\norm{x}_2\bigr)^{d-1}}{c_d\norm{x}_2^{d-1} (\partial_r \xi_{v_x})\Bigl( \xi_{v_x}^{-1}\bigl(\norm{x}_2\bigr) \Bigr)} \comma \quad \text{where } v_x \coloneqq \frac{x}{\norm{x}_2} \comma \qquad \norm{x}_2<1 \fstop
	\]
	If we set
	\[
		\alpha(t) \coloneqq \begin{cases}
			\frac{1}{c_d} &\text{if } t=0 \comma \\
			\frac{1}{c_d \sqrt{1+4t}} \left(\frac{\sqrt{1+4 t}-1}{2t}\right)^{d-1} &\text{if } t > 0 \comma
		\end{cases} \quad \beta(x) \coloneqq \norm{x}_1 \left( \frac{\norm{x}_1}{\norm{x}_2} - 1\right) \comma
	\]
	tedious but simple computations (passing through the explicit formula for~$\xi_v^{-1}$) reveal that~$\rho_0|_{\Omega_0} = \alpha \circ \beta$.
	When~$\norm{x}_2 < 1$, the values of~$\beta(x)$ range between~$0$ and~$d-\sqrt{d}$. On this interval, the function~$\alpha$ is Lipschitz continuous and positive. Since~$\beta|_{\Omega_0}$ is Lipschitz too, the proof is complete.
	\end{proof}

\emph{Step 2 ($\Omega \ne \tilde \Omega$).} Let us now generalize to the case where, possibly,~$\Omega \ne \tilde \Omega$, but there exists~$M \colon \tilde \Omega \to \Omega$ as in the assumptions.
	Consider the probability measure~$\tilde \mu$ defined by
	\[
	\tilde \rho \coloneqq \begin{cases}
		( \rho \circ M) \abs{\det \nabla M} &\text{on } \tilde \Omega \comma \\
		0 &\text{on } \R^d \setminus \tilde \Omega \comma
	\end{cases}
	\quad
	\tilde \mu \coloneqq \tilde \rho \LebRd \fstop
	\]
	Thanks to the assumptions on~$M$ and~$\rho$, to this new measure we can apply Step~1; thus~\eqref{eq:uppermu} holds for~$\tilde \mu$. Moreover, by the change of variables formula,~$\mu = M_\# \tilde \mu$, and the map~$M$ is Lipschitz because its Jacobian is bounded and~$\tilde \Omega$ is convex. Hence, we conclude by \Cref{cor:lipPF}.
\end{proof}

\section{Main theorem (\Cref{thm:main1})} \label{sec:main}

This section is subdivided into four parts: we first establish two preliminary lemmas, then we prove \Cref{thm:main1} for \emph{singular} measures, the upper bound~\eqref{eq:mainabove} (in general), and eventually the lower bound~\eqref{eq:mainbelow}.

\subsection{Preliminary lemmas}

\begin{lemma} \label{lemma:seqlemma}
	Let~$(b_k)_{k \in \N_1}$ be a sequence of nonnegative numbers, infinitesimal as~$k \to \infty$. Then there exists a sequence~$(k_n)_{n \in \N_1} \subseteq \N_1$ such that~$k_n \to \infty$ as~$n \to \infty$ and
	\begin{equation} \label{eq:seqlemma}
		\lim_{n \to \infty} n^{-1/d} 2^{k_n} = \lim_{n \to \infty} n^{1/d} 2^{-k_n} b_{k_n} = 0 \fstop
	\end{equation}
\end{lemma}

\begin{proof}
	The existence of such a sequence is established in the proof of~\cite[Theorem~5]{BartheBordenave13} by F.~Barthe and C.~Bordenave.
\end{proof}

\begin{lemma} \label{lemma:partition}
	Let~$C \subseteq \R^d$ be a closed set and let~$\rho \in L^1_{\ge 0}(\R^d)$. For every~$k \in \N_1$ and~$s \ge 0$ define the open sets
	\begin{equation} \label{eq:lemma:partition:1}
		\Omega_i \coloneqq (0, 2^{-k})^d+2^{-k}i \comma \qquad i \in \Z^d \comma
	\end{equation}
	and
	\begin{equation}
		\Omega^{(k)} \coloneqq \interior\left(\bigcup_{i \in \Z^d \, : \, \overline{\Omega_i} \cap C = \emptyset} \overline{\Omega_i}\right) \comma
	\end{equation}
	(see \Cref{fig:geomSetup}), the set of indices
	\begin{equation} \label{eq:lemma:partition:2}
		I_{k,s} \coloneqq 
			\set{ i \in \Z^d \, : \, \displaystyle \norm{x-y} > s \, \, \forall x \in \Omega_i \, \, \forall y \in \R^d \setminus {\Omega^{(k)}}} \comma
	\end{equation}
	and the function
	\begin{equation} \label{eq:lemma:partition:4}
		\rho_{k,s} \coloneqq \sum_{i \in I_{k,s}} \left(\fint_{\Omega_i} \rho \, \dif \LebRd \right) \1_{\Omega_i} \fstop
	\end{equation}
	Then~$\rho_{k,s} \to \rho|_{\R^d \setminus C}$ almost everywhere and in~$L^1(\R^d)$ as~$k \to \infty$ and~$s \to 0$.
\end{lemma}

\begin{proof}
	Almost every point~$x \in \R^d \setminus C$ (for example, the points out of~$C$ for which all coordinates are irrational) is contained in some~$\Omega_i$ with~$i \in I_{k,s}$ as soon as its distance from~$C$ is larger than~$\sqrt{d}2^{1-k}+s$. Therefore, by~\cite[Theorem~6.2.3]{Cohn13}, we have~$\rho_{k,s} \to \rho|_{\R^d \setminus C}$ almost everywhere and, by Scheffé's Lemma~\cite[Theorem~5.10]{Williams91}, in~$L^1(\R^d)$.
\end{proof}

\begin{remark} \label{rmk:omegak}
	With the notation of \Cref{lemma:partition}, note the following:
	\[
		\bigcup_{i \in I_{k,s}} \Omega_i \subseteq \bigcup_{i \in I_{k,0}} \Omega_i \subseteq \Omega^{(k)} \subseteq \bigcup_{i \in I_{k,0}} \overline{\Omega_i} \subseteq \R^d \setminus C \comma \qquad k \in \N_1 \comma \quad s \ge 0 \fstop
	\]
\end{remark}

\subsection{Singular measures}

The proof of \Cref{thm:main1} for singular measures is inspired by~\cite[Proposition~3]{DereichScheutzowSchottstedt13}. We will combine the following three observations:
\begin{itemize}
	\item we can split~$\mu$ into measures~$\mu^i$ supported on small \emph{cubes} (plus a remainder that we control with \Cref{thm:nonasy});%
	\item the error~$\eqe{p}{n}^p(\mu)$ is subadditive in the sense of \Cref{rmk:subaddErr} and, by \Cref{thm:MerigotMirebeauChevallier}, for every~$\mu^i \in \ProRd{{\mathrm{c}}}$, we can bound~$n^{1/d}\eqe{p}{n}(\mu^i)$ in terms of the diameter of the support of~$\mu^i$;
	\item since~$\mu$ is singular, it is concentrated on open sets with arbitrarily small Lebesgue measure.
\end{itemize}

\begin{proof}[Proof of \Cref{thm:main1} for~$\mu \perp \LebRd$]
	Choose any open set~$\Omega \subseteq \R^d$ such that~$\mu(\Omega) = 1$, and write it as a countable \emph{disjoint} union of (half-open) \emph{cubes}~$\set{Q_i}_{i \in \N_1}$, see, e.g.,~\cite[Theorem~1.11]{WheedenZygmund15}. Note, in particular, that~$\diam(Q_i) \lesssim_d \abs{Q_i}^{1/d}$.
	
	Pick two numbers~$n,i_\mathrm{max} \in \N_1$ and define
	\[
		n_i \coloneqq \lfloor n \mu(Q_i) \rfloor \comma \quad \mu^i \coloneqq \begin{cases}
			\frac{n_i}{n\mu(Q_i)} \mu|_{Q_i} &\text{if } n_i \ge 1 \comma\\
			0 &\text{otherwise,}
		\end{cases} \qquad i \in \set{1,\dots,i_\mathrm{max}} \fstop
	\]
	Notice that~$\mu^i \le \mu$ for every~$i$ and, since the cubes are all disjoint, also the sum~$\sum_{i=1}^{i_\mathrm{max}} \mu^i$ is not larger than~$\mu$. Define
	\[
		n_0 \coloneqq n - \sum_{i=1}^{i_\mathrm{max}} n_i \comma \quad  \mu^0 \coloneqq \mu - \sum_{i = 1}^{i_\mathrm{max}} \mu^i \comma
	\]
	and notice that~$\normTV{\mu^0} = n_0/n$.
	
	Owing to \Cref{rmk:subaddErr}, we have
	\[
		\eqe{p}{n}^p(\mu) \le \sum_{i=0}^{i_\mathrm{max}} \eqe{p}{n_i}^p(\mu^i) \fstop
	\]
	\Cref{thm:nonasy} (or \Cref{thm:MerigotMirebeauChevallier} for~$i \ge 1$) yields
	\begin{align} \label{eq:lemma:singular:1}
		\begin{split}
		\eqe{p}{n}^p(\mu) &\lesssim_{p,d,\theta} n^{\frac{p}{\theta}-1}n_0^{1-\frac{p}{d}-\frac{p}{\theta}} \left(\int \norm{x}^{\theta} \, \dif \mu_0 \right)^\frac{p}{\theta}  + \sum_{i=1}^{i_\mathrm{max}} n^{-1} n_i^{1-\frac{p}{d}} \diam(Q_i)^p  \\
		&\lesssim_{p,d,\theta,\mu} n^{\frac{p}{\theta}-1}n_0^{1-\frac{p}{d}-\frac{p}{\theta}} + n^{-1} \sum_{i=1}^{i_\mathrm{max}} n_i^{1-\frac{p}{d}} \abs{Q_i}^{p/d} \comma
		\end{split}
	\end{align}
	where~$\theta > p^*$ is such that~$\mu \in \ProRd{\theta}$. Note that~$a \coloneqq 1-\frac{p}{d}-\frac{p}{\theta} > 0$. Since~$p < d$, we can apply H\"older's inequality to the last sum and obtain
	\begin{equation} \label{eq:lemma:singular:2}
		\sum_{i=1}^{i_\mathrm{max}} n_i^{1-\frac{p}{d}} \abs{Q_i}^{p/d} \le \left( \sum_{i=1}^{i_\mathrm{max}} n_i \right)^{1-\frac{p}{d}} \left(\sum_{i=1}^{i_\mathrm{max}} \abs{Q_i} \right)^{p/d} \le n^{1-\frac{p}{d}} \abs{\Omega}^{p/d} \fstop
	\end{equation}
	Furthermore, we notice that
	\begin{equation} \label{eq:lemma:singular:3}
		n_0 \le n - \sum_{i=1}^{i_\mathrm{max}} \bigl(n\mu(Q_i)-1\bigr) = n\left( 1-\sum_{i=1}^{i_\mathrm{max}} \mu(Q_i) \right) + {i_\mathrm{max}} \fstop
	\end{equation}
	We now combine~\eqref{eq:lemma:singular:1},~\eqref{eq:lemma:singular:2}, and~\eqref{eq:lemma:singular:3} to infer
	\[
		n^{p/d} \eqe{p}{n}^p(\mu) \lesssim_{p,d,\theta,\mu} \left( 1-\sum_{i=1}^{i_\mathrm{max}} \mu(Q_i) + \frac{{i_\mathrm{max}}}{n} \right)^{a} + \abs{\Omega}^{p/d} \semicolon
	\]
	hence,
	\begin{equation} \label{eq:lemma:singular:4}
		\limsup_{n \to \infty} n^{p/d} \eqe{p}{n}^p(\mu) \lesssim_{p,d,\theta,\mu} \left( 1-\sum_{i=1}^{i_\mathrm{max}} \mu(Q_i) \right)^{a} + \abs{\Omega}^{p/d} \fstop
	\end{equation}
	Since~$\mu$ is concentrated on~$\bigcup_{i \in \N_1} Q_i$, the first term at the right-hand side of~\eqref{eq:lemma:singular:4} tends to~$0$ as~${i_\mathrm{max}} \to \infty$. Moreover, when~$\mu$ is singular,~$\abs{\Omega}$ can be made arbitrarily small.
\end{proof}

\subsection{Upper bound}

To prove the upper bound, we first assume that the measure~$\mu$ is compactly supported and absolutely continuous. We split the domain into cubes~$\set{\Omega_i}_i$ with edge length $2^{-k}$ and consider an approximating density~$\rho_k$ that is constant on each of these cubes. We then construct a further approximation $\rho_{k}^{(n)}$ having mass on each cube equal to an integer multiple of~$1/n$, i.e., of the form
\[
	\rho_{k}^{(n)} \coloneqq \sum_{i} \frac{n_i}{n} \frac{\1_{\Omega_i}}{\abs{\Omega_i}}
\]
with~$n_i \approx n \mu(\Omega_i)$.
Using \Cref{rmk:scaling}, \Cref{rmk:subaddErr}, and \Cref{prop:uniformCube}, it is possible to show that
\[
	\limsup_{k \to \infty} \limsup_{n \to \infty} n^{1/d} \eqe{p}{n}\bigl(\rho_{k}^{(n)} \bigr) \le \eqc{p}{d} \left(\int_{\R^d} \rho^\frac{d-p}{d} \, \dif \LebRd \right)^{1/p} \fstop
\]
Indeed, heuristically:
\begin{alignat*}{2}
	\eqe{p}{n}^p\bigl(\rho_k^{(n)}\bigr) &\le \sum_i \frac{n_i}{n} \abs{\Omega_i}^{p/d} \eqe{p}{n_i}^p(U_d) &&\qquad \text{(Rmk.~\ref{rmk:scaling}, Rmk.~\ref{rmk:subaddErr})} \\
	&\approx \eqc{p}{d}^p \sum_i \frac{n_i}{n} \abs{\Omega_i}^{p/d} n_i^{-p/d} &&\qquad \text{(Prop.~\ref{prop:uniformCube})} \\
	&\approx n^{-p/d}\eqc{p}{d}^p \sum_i \mu(\Omega_i)^\frac{d-p}{d} \abs{\Omega_i}^{p/d} &&\qquad \text{(} n_i \approx n\mu(\Omega_i) \text{)} \\
	&\approx n^{-p/d}\eqc{p}{d}^p \int_{\R^d} \rho^\frac{d-p}{d} \, \dif \LebRd \fstop
\end{alignat*}

Our argument is similar to the proof of Zador's Theorem (see \cite[Steps~2 \& 3 in Theorem~6.2]{GrafLuschgy00}), but we have an additional obstacle: for fixed~$k$, the approximating error explodes as~$n \to \infty$. Even worse: the two errors made by replacing~$\mu$ with~$\rho_k $, and~$\rho_k $ with~$\rho_{k}^{(n)} $ compete with each other, in the sense that (up to constant), each one is \emph{almost} equal to a negative power of the other. However, in the upper bound for the error~$W_p^p(\mu,\rho_k )$, thanks to \Cref{lemma:comparL1}, there is also the additional term~$\norm{\rho-\rho_k}_{L^1}$. This term is infinitesimal as~$k \to \infty$ (here we use that~$\mu$ is absolutely continuous and \Cref{lemma:partition}). Taking advantage of \Cref{lemma:seqlemma}, we can let~$k$ tend to infinity with~$n$ in such a way that both approximating errors become negligible. This solution is partly inspired by the proof of~\cite[Theorem~5]{BartheBordenave13}.

To deal with a general measure, we split it into its singular part, a compactly supported and absolutely continuous part, and a remainder. To the latter, we apply \Cref{thm:nonasy}.

\begin{proof}[Proof of the upper bound in~\Cref{thm:main1}]
	\emph{Step 1 ($\mu \in \ProRd{{\mathrm{c}}}$ and $\mu \ll \LebRd$).} We start by proving the upper bound~\eqref{eq:mainabove} under the additional assumption that~$\mu$ is absolutely continuous, i.e.,~$\mu = \rho \LebRd$, and compactly supported. It is easy to check that, if~$T \colon \R^d \to \R^d$ is a homothety, then~\eqref{eq:mainabove} for~$\mu$ and for~$T_\# \mu$  are equivalent. Thus, without loss of generality, we assume that~$\mu$ is concentrated on~$(0,1)^d$.

	Fix~$k,n \in \N_1$. Let us define~$\set{\Omega_i}_{i \in \Z^d}$,~$I_k = I_{k,0}$, and~$\rho_k = \rho_{k,0}$ as in \Cref{lemma:partition} with~$C \coloneqq \R^d \setminus (-1,2)^d$ and~$s = 0$. Notice that~$\norm{\rho_k}_{L^1} = \normTV{\mu}=1$. For every~$i \in I_k$, we define~$n_i = n_i(n,k) \coloneqq \lfloor n \mu(\Omega_i) \rfloor$, and we let~$n_0 \coloneqq n-\sum_{i \in I_k} n_i$. We then set
	\[
	 \rho_{k}^{(n)} \coloneqq \sum_{i \in I_k} \frac{n_i}{n} 2^{kd} \1_{\Omega_i} \le \rho_k \fstop
	\]
	
	By using the triangle inequality, it is immediate to check that
	\begin{equation} \label{eq:thm:main1:1}
		\eqe{p}{n}(\mu) \le \eqe{p}{n}(\rho_k ) + W_p(\mu, \rho_k ) \fstop
	\end{equation}
	\Cref{rmk:subaddErr} yields
	\begin{equation} \label{eq:thm:main1:2}
		\eqe{p}{n}^p(\rho_k ) \le \eqe{p}{n_0}^p\bigl(\rho_k-\rho_{k}^{(n)}\bigr) + \sum_{i \in I_k} \eqe{p}{n_i}^p\left( \frac{n_i}{n} U(\Omega_i) \right)  \comma
	\end{equation}
	and can use \Cref{rmk:scaling} to write
	\begin{equation} \label{eq:thm:main1:3}
		\eqe{p}{n_i}^p\left( \frac{n_i}{n} U(\Omega_i) \right) = \frac{n_i}{n} \eqe{p}{n_i}^p\bigl(U(\Omega_i)\bigr)= \frac{n_i}{n 2^{kp}} \eqe{p}{n_i}^p(U_d) \comma \qquad i \in I_k \text{ s.t.~} n_i \ge 1 \fstop
	\end{equation}
	The $\frac{1}{p}$-homogeneity of~$\eqe{p}{n_0}$, combined with \Cref{thm:MerigotMirebeauChevallier} (recall that, currently, all measures are concentrated on~$(0,1)^d$) gives
	\begin{equation*}
		\eqe{p}{n_0}^p\bigl(\rho_k-\rho_{k}^{(n)}\bigr) \lesssim_{p,d} \norm{\rho_k-\rho_{k}^{(n)}}_{L^1} n_0^{-p/d} = \frac{n_0^\frac{d-p}{d}}{n} \fstop
	\end{equation*}
	Thus, since
	\[
	n_0 = n - \sum_{i \in I_k} \lfloor n \mu(\Omega_i) \rfloor \le n - \sum_{i \in I_k} n \mu(\Omega_i) + \# I_k = \# I_k \lesssim_d 2^{kd} \comma
	\]
	we have
	\begin{equation} \label{eq:thm:main1:4}
		\eqe{p}{n_0}^p\bigl(\rho_k-\rho_{k}^{(n)}\bigr) \lesssim_{p,d} \frac{2^{k(d-p)}}{n}
	\end{equation}
	(here we use~$p < d$).
	Moreover, by applying \Cref{lemma:subadditivityWasserstein} and \Cref{lemma:comparL1}, we get
	\begin{align} \label{eq:thm:main1:5}
		\begin{split}
			W_p^p(\mu, \rho_k ) &\le \sum_{i \in I_k} W_p^p\left( \mu|_{\Omega_i}, \mu(\Omega_i) U(\Omega_i) \right) \\
			&\le \sum_{i \in I_k} \diam(\Omega_i)^p \norm{(\rho-\rho_k)|_{\Omega_i}}_{L^1} 
			\lesssim_{p,d} 2^{-kp} \norm{\rho-\rho_k}_{L^1} \fstop
		\end{split}
	\end{align}
	
	By \Cref{lemma:seqlemma} and since~$\rho_k \stackrel{L^1}{\to} \rho$ as~$k \to \infty$ (\Cref{lemma:partition}), we can choose~$k=k_n$ as a function of~$n$ in such a way that %
	\begin{equation} \label{eq:thm:main1:7}
	\lim_{n \to \infty}  n^{-1/d} 2^{k_n} = 	\lim_{n \to \infty} n^{1/d}2^{-k_n}\norm{\rho-\rho_{k_n}}_{L^1}^{1/p} = 0 \comma
	\end{equation}
	By~\eqref{eq:thm:main1:1},~\eqref{eq:thm:main1:2},~\eqref{eq:thm:main1:3},~\eqref{eq:thm:main1:4},~\eqref{eq:thm:main1:5}, and~\eqref{eq:thm:main1:7} we thus have
	\begin{align*}
		\begin{split}
		\limsup_{n \to \infty} n^{p/d} \eqe{p}{n}^p(\mu) &\le \limsup_{n \to \infty} 2^{-k_np} \sum_{i \in I_{k_n} \, : \, n_i \ge 1}  \left(\frac{n_i}{n } \right)^\frac{d-p}{d} n_i^{p/d} \eqe{p}{n_i}^p(U_d) \\
		&\le \limsup_{n \to \infty} \int \rho_{k_n}^\frac{d-p}{d} h( \lfloor n2^{-k_n d} \rho_{k_n} \rfloor ) \, \dif \LebRd \comma
	\end{split}
\end{align*}
	where
	\[
		h(m) \coloneqq \begin{cases} m^{p/d} \eqe{p}{m}^p(U_d) &\text{if } m \in \N_1 \comma \\
			0 &\text{if } m = 0 \fstop
		\end{cases}
	\]
	Note that~$h$ is nonnegative, bounded, and converges to~$\eqc{p}{d}^p$ as~$m \to \infty$ by \Cref{prop:uniformCube}. In particular, since~$n2^{-k_nd} \to \infty$ and~$\rho_{k_n} \to \rho$ a.e., we have
	\begin{equation} \label{eq:thm:main1:8}
		\lim_{n \to \infty} h( \lfloor n2^{-k_n d} \rho_{k_n} \rfloor ) \, \dif \LebRd = \eqc{p}{d}^p \comma \qquad \text{a.e.~on } \set{\rho > 0} \fstop
	\end{equation}
	Since~$\frac{d-p}{d} \in (0,1)$, the function~$t \mapsto t^\frac{d-p}{d}$ is subadditive. Therefore,
	\begin{multline*}
		\int \rho_{k_n}^\frac{d-p}{d} h( \lfloor n2^{-k_n d} \rho_{k_n} \rfloor ) \, \dif \LebRd \le \int \rho^\frac{d-p}{d} h( \lfloor n2^{-k_n d} \rho_{k_n} \rfloor ) \, \dif \LebRd \\
		+ \sup_{m \in \N_0} h(m) \int \abs{\rho - \rho_{k_n}}^\frac{d-p}{d} \, \dif \LebRd \fstop
	\end{multline*}
	Using that~$\rho$ and~$\rho_{k_n}$ are supported on~$[0,1]^d$, Jensen's inequality gives
	\[
		\int \abs{\rho - \rho_{k_n}}^\frac{d-p}{d} \, \dif \LebRd \le \norm{\rho-\rho_{k_n}}_{L^1}^\frac{d-p}{d} \to 0 \comma \quad \text{as } n \to \infty \fstop
	\]
	In the end, we obtain
	\[
		\limsup_{n \to \infty} n^{p/d} \eqe{p}{n}^p(\mu) \le \limsup_{n \to \infty} \int \rho^\frac{d-p}{d} h(\lfloor n {2^{-k_n}d} \rho_{k_n} \rfloor) \, \dif \LebRd = \eqc{p}{d}^p \int \rho^\frac{d-p}{d} \, \dif \LebRd \comma
	\]
	where the last identity follows from~\eqref{eq:thm:main1:8} and the dominated convergence theorem.

	\emph{Step 2 (conclusion).} Let~$\mu \in \ProRd{\theta}$ and fix~$r > 0$. We let:
	\begin{itemize}
		\item $\mu^1$ be the \emph{absolutely continuous part} of~$ \mu|_{[-r,r]^d}$, 
		\item $\mu^2$ be the \emph{singular part} of~$\mu|_{[-r,r]^d}$,
		\item $\mu^3 \coloneqq \mu|_{\R^d \setminus [-r,r]^d}$.
	\end{itemize}
	Furthermore, for~$n \in \N_1$, we define
	\[
	n_i \coloneqq \left \lfloor n \normTV{\mu^i} \right \rfloor \comma \quad \mu^{i,n} \coloneqq \begin{cases}
		\frac{n_i}{n \, \normTV{ \mu^i} } \mu^i &\text{if } n_i \ge 1 \comma\\
		0 &\text{otherwise,}
	\end{cases} \qquad i \in \set{1,2,3} \comma
	\]
	as well as~$n_0 \coloneqq n-n_1-n_2-n_3$ and~$\mu^{0,n} \coloneqq \mu-\mu^{1,n}-\mu^{2,n}-\mu^{3,n}$. Note that~$n_0 \le 3$.
	
	By \Cref{rmk:subaddErr}, we can make the estimate
	\begin{equation} \label{eq:thm:main1:1bis}
		\limsup_{n \to \infty} n^{p/d} \eqe{p}{n}^p(\mu) \le \sum_{i=0}^3 \limsup_{n \to \infty} n^{p/d} \eqe{p}{n_i}^p(\mu^{i,n}) \fstop
	\end{equation}
	We shall bound the four terms in the sum separately.

	When~$n_0 \ge 1$, \Cref{rmk:scaling} and \Cref{thm:nonasy} yield
	\begin{align*}
		n^{p/d} \eqe{p}{n_0}^p(\mu^{0,n}) &= \left(\frac{n_0}{n}\right)^{1-\frac{p}{d}} n_0^{p/d} \eqe{p}{n_0}^p\left( \frac{\mu^{0,n}}{\normTV{\mu^{0,n}}} \right) \\
		&\stackrel{\eqref{eq:thm:nonasy:0}}{\lesssim_{p,d,\theta}} \left(\frac{n_0}{n}\right)^{1-\frac{p}{d}} \left(\int \norm{x}^{\theta} \, \dif \frac{\mu^{0,n}}{\normTV{\mu^{0,n}}} \right)^{p/\theta} \\
		&\le \left(\frac{n_0}{n}\right)^{1-\frac{p}{d}-\frac{p}{\theta}} \left(\int \norm{x}^{\theta} \, \dif \mu \right)^{p/\theta} \fstop
	\end{align*}
	The exponent~$a \coloneqq 1-\frac{p}{d}-\frac{p}{\theta}$ is positive; hence (the case~$n_0 = 0$ is trivial),
	\begin{equation*}
		n^{p/d} \eqe{p}{n_0}^p(\mu^{0,n}) \lesssim_{p,d,\theta,\mu} n^{-a} \comma
	\end{equation*}
	which means that the $0^\text{th}$ term of the sum in \eqref{eq:thm:main1:1bis} is zero.
	
	When~$n_3 \ge 1$, similar computations give
	\[
		n^{p/d} \eqe{p}{n_3}^p(\mu^{3,n}) \lesssim_{p,d,\theta} \left( \frac{n_3}{n} \right)^{a} \left( \int \norm{x}^{\theta} \, \dif \mu^3 \right)^{p/\theta} \fstop
	\]
	Using that~$n_3 \le n \normTV{\mu^3} \le n \normTV{\mu}$, we thus obtain (trivially if~$n_3 = 0$)
	\begin{equation*}
		n^{p/d} \eqe{p}{n_3}^p(\mu^{3,n}) \lesssim_{p,d,\theta,\mu} \left( \int \norm{x}^{\theta} \, \dif \mu|_{\R^d \setminus [-r,r]^d} \right)^{p/\theta} \fstop 
	\end{equation*}
	
	If~$\normTV{\mu^2} > 0$, then~$n_2 \to \infty$ as~$n \to \infty$; therefore \Cref{thm:main1} for singular measures yields
	\[
		\limsup_{n \to \infty} n^{p/d} \eqe{p}{n_2}^p(\mu^{2,n}) = \limsup_{n \to \infty} \left(\frac{n_2}{n} \right)^{\frac{d-p}{d}} n_2^{p/d} \eqe{p}{n_2}^p\left( \frac{\mu^2}{\normTV{\mu^2}} \right) = 0 \comma
	\]
	and the same conclusion holds trivially if~$\normTV{\mu^2}=0$.
	
	If~$\norm{\mu^1} > 0$, then~$n_1 \to \infty$ as~$n \to \infty$; therefore the previous Step gives
	\begin{align*}
		\limsup_{n \to \infty} n^{p/d} \eqe{p}{n_1}^p(\mu^{1,n}) &= \limsup_{n \to \infty} \left(\frac{n_1}{n}\right)^{\frac{d-p}{d}} n_1^{p/d} \eqe{p}{n_1}^p \left( \frac{\mu^1}{\normTV{\mu^1}} \right) \\
		&\le \eqc{p}{d}^p \normTV{\mu^1}^{\frac{d-p}{d}} \int_{[-r,r]^d} \left(\frac{\rho}{\normTV{\mu^1}}\right)^\frac{d-p}{d} \, \dif \LebRd \\
		&= \eqc{p}{d}^p \int_{[-r,r]^d} \rho^\frac{d-p}{d} \, \dif \LebRd \comma
	\end{align*}
	and the same conclusion holds trivially if~$\normTV{\mu^1} = 0$.
	
	In the end,~\eqref{eq:thm:main1:1bis} and the subsequent estimates prove that
	\[
		\limsup_{n \to \infty} n^{p/d} \eqe{p}{n}^p(\mu) \le \eqc{p}{d}^p \int_{[-r,r]^d} \rho^\frac{d-p}{d} \, \dif \LebRd + c_{p,d,\theta,\mu} \left( \int \norm{x}^{\theta} \, \dif \mu|_{\R^d \setminus [-r,r]^d} \right)^{p/\theta}
	\]
	for some constant~$c_{p,d,\theta,\mu}$ independent of~$r$. We conclude by letting~$r \to \infty$.
\end{proof}

\subsection{Lower bound}

To prove the lower bound, we again split the domain into cubes~$\set{\Omega_i}_i$ with edge length~$2^{-k}$ and approximate the density~$\rho$ of the given measure~$\mu$ by the piecewise constant function
\[
	\rho_{k,0} = \sum_{i} \mu(\Omega_i) \frac{\1_{\Omega_i}}{\abs{\Omega_i}} \fstop
\]
Given an optimal empirical quantizer~$\mu_n$ for~$\mu$, we aim to bound from below the boundary Wasserstein pseudodistance between~$\rho_{k,0}$ and~$\mu_n$. We make use of this pseudodistance---smaller than the Wasserstein distance---because its geometric superadditivity (\Cref{lemma:superadditivityWb}) is well-suited to reduce the lower bound problem to the single cubes. On each cube, we use \Cref{lemma:lowerBoundWb} and the definition of~$q_{p,d}$ to obtain the integral of $\rho^\frac{d-p}{d}$. The argument can be sketched as follows:
\begin{alignat*}{2}
	W_p^p(\rho_{k,0},\mu_n) &\ge \sum_i Wb_{\Omega_i,p}^p\bigl(\mu(\Omega_i) U(\Omega_i), \mu_n|_{\Omega_i} \bigr) &&\qquad \text{(Lem.~\ref{lemma:superadditivityWb})} \\
	&\gtrapprox \sum_i e_{p,n_i}^p\bigl(\mu(\Omega_i) U(\Omega_i) \bigr) \text{ with } n_i \approx n \mu_n(\Omega_i) &&\qquad \text{(Lem.~\ref{lemma:lowerBoundWb})} \\
	&\ge q_{p,d}^p \sum_i n_i^{-p/d} \mu(\Omega_i) \abs{\Omega_i}^{p/d} &&\qquad \text{(Rmk.~\ref{rmk:scaling}, Def.~\ref{def:eqe})}
\end{alignat*}
and, since it is reasonable to expect that~$\mu_n(\Omega_i) \approx \mu(\Omega_i)$, we get
\[
	n^{p/d} W_p^p(\rho_{k,0},\mu_n) \gtrapprox q_{p,d}^p \int_{\R^d} \rho^\frac{d-p}{d} \, \dif \LebRd \fstop
\]

The idea of using the boundary Wasserstein pseudodistance (or a similar object) to exploit its geometric superadditivity is not new. It has been used to prove lower bounds in similar problems, see, e.g.,~\cite{BartheBordenave13,DereichScheutzowSchottstedt13,AmbrosioGoldmanTrevisan22}. There is, however, a technical difference between these works, which estimate the expected value of a functional of i.i.d.~random variables, and the current one. Given a set of $\mu$-distributed i.i.d.~random variables, the random number of those that fall within a certain region (cube) is a binomial r.v.~whose law can be explicitly determined in terms of~$\mu$. Instead, given an optimal empirical quantizer~$\mu_n$, it does not seem immediate to rigorously justify the heuristic~$\mu_n(\Omega_i) \approx \mu(\Omega_i)$. A considerable part of the proof is indeed devoted to this problem.

\begin{proof}[Proof of the lower bound in~\Cref{thm:main1}] %
	
	Fix~$k, n \in \N_1$ and~$s \in (0,2^{-k})$, choose two numbers~$\epsilon_1, \epsilon_2 \in (0, 1)$, and define~$\set{\Omega_i}_{i \in \Z^d},\Omega^{(k)},I_{k,0},I_{k,s},\rho_{k,0},\rho_{k,s}$ as in \Cref{lemma:partition} with~$C \coloneqq \supp (\mu^s)$. %
	Set
	\[
		\Omega_i^- \coloneqq \set{x \in \Omega_i \, : \, \dist(x,\R^d \setminus \Omega_i) > \epsilon_12^{-k-1} } \comma \qquad i \in \Z^d \fstop
	\]
	Note that each~$\Omega_{i}^-$ is an open cube with edge length equal to~$(1-\epsilon_1)2^{-k}$. It is also convenient to define the ``enlarged'' sets
	\[
	\Omega_i^+ \coloneqq \set{x \in \R^d \, : \, \dist(x,\Omega_i) < s } \comma \qquad i \in \Z^d \fstop
	\]
	
	\begin{figure}
		\centering
		\includegraphics[trim={10.9cm 10cm 3.6cm 10cm},clip,width = 0.6\textwidth,angle=90]{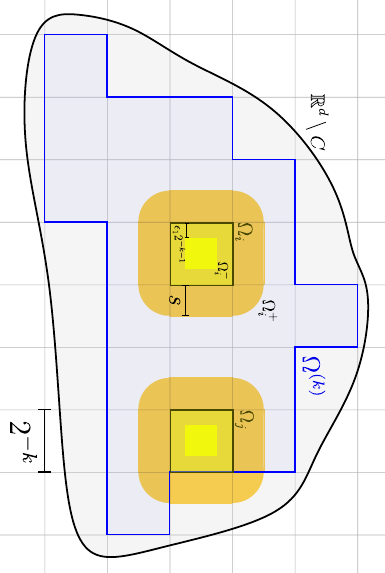}
		\caption{Geometric setup in the proof of the lower bound. In this example,~$i \in I_{k,s}$ and~$j \in I_{k,0} \setminus I_{k,s}$.}
		\label{fig:geomSetup}
	\end{figure}
	
	An important observation that we are going to use later is:
	\begin{equation} \label{eq:thm:main1:1ter}
		\abs{\Omega_i^+ \cap \Omega_j} \le s 2^{-k(d-1)} \quad \text{if } i \neq j \fstop
	\end{equation}

	We say that two cubes~$\Omega_i$ and~$\Omega_j$ are \emph{adjacent}, and we write~$i \sim j$, if~$\overline {\Omega_i} \cap \overline {\Omega_j} \neq \emptyset$ (it suffices that their closures share a single vertex). 
	Notice that each cube has~$3^d$ adjacent cubes, including itself, and that, since~$s < 2^{-k}$, the intersection~$\Omega_i^+ \cap \Omega_j$ is nonempty iff~$i \sim j$.

	Using~\Cref{lemma:existence}, pick~$\mu_n \in \ProRd{(n)}$ such that~$\eqe{p}{n}(\mu) = W_p(\mu,\mu_n)$. We have
	\begin{align} \label{eq:thm:main1:2ter} \begin{split}
		\eqe{p}{n}(\mu) &\ge Wb_{\Omega^{(k)},p}(\mu,\mu_n) \ge Wb_{\Omega^{(k)},p}(\rho_{k,0} , \mu_n) - Wb_{\Omega^{(k)},p}(\rho_{k,0} , \mu) \\
		&\ge Wb_{\Omega^{(k)},p}(\rho_{k,0} , \mu_n) - W_p(\rho_{k,0} , \mu|_{\Omega^{(k)}}) \comma
		\end{split}
	\end{align}
	where, in the last inequality, we used that~$\mu\bigl(\Omega^{(k)}\bigr) = \norm{\rho_{k,0}}_{L^1}$ (recall \Cref{rmk:omegak}). Note that, in the same way we derived~\eqref{eq:thm:main1:5}, we can deduce
	\begin{equation} \label{eq:thm:main1:3ter}
		W_p(\rho_{k,0},\mu|_{\Omega^{(k)}}) \lesssim_{p,d} 2^{-k} \norm{\rho|_{\Omega^{(k)}}-\rho_{k,0}}_{L^1}^{1/p} \fstop
	\end{equation}
	
	Let us focus on~$Wb_{\Omega^{(k)},p}(\rho_{k,0},\mu_n)$. Set~$ n_i \coloneqq n \mu_n(\Omega_i) \in \N_0$ for every~$i \in I_{k,s}$. %
	By the superadditivity property of \Cref{lemma:superadditivityWb},
	\[
		Wb_{\Omega^{(k)},p}^p\bigl(\rho_{k,0} , \mu_n\bigr) \ge \sum_{i \in I_{k,s} \, : \, \mu(\Omega_i) \ge \epsilon_2 2^{-kd}} Wb_{\Omega_i,p}^p\bigl(\rho_{k,0} , \mu_n\bigr) \comma
	\]
	and, by \Cref{lemma:lowerBoundWb},
	\begin{equation*}
		Wb_{\Omega_i,p}^p\bigl(\rho_{k,0} , \mu_n\bigr) \ge e_{p,  n_i+N}^p\bigl(\rho_{k,0} |_{\Omega_{i}^-}\bigr) 
		= \frac{\left(1 - \epsilon_1 \right)^{p+d}}{2^{kp}} \mu(\Omega_i) e_{p, n_i+N}^p (U_d) \comma \qquad
	\end{equation*}
	where~$N \coloneqq \left\lceil {2d} / \epsilon_1 \right\rceil^d$. %
	Hence, by the definition of~$q_{p,d}$, we have
	\begin{equation} \label{eq:thm:main1:5ter}
		Wb_{\Omega^{(k)},p}^p(\rho_{k,0}, \mu_n) \ge  \frac{q_{p,d}^p  \left(1 - \epsilon_1 \right)^{p+d}}{2^{kp}} \sum_{i \in I_{k,s} \, : \, \mu(\Omega_i) \ge \epsilon_2 2^{-kd}} \mu(\Omega_i) \bigl( n_i+N\bigr)^{-p/d} \fstop
	\end{equation}

	At this point, we need to estimate~$ n_i$ from above. To this aim, pick~$\gamma \in \Gamma b_{\Omega^{(k)}}(\rho_{k,0} , \mu_n)$, which gives:
	\begin{align*}
		\frac{ n_i}{n} &= \mu_n(\Omega_i) = \gamma\bigl(\overline{\Omega^{(k)}} \times \Omega_i\bigr) \\
		&= \gamma\Bigl(\bigl(\overline{\Omega^{(k)}} \setminus \Omega_i^+ \bigr) \times \Omega_i \Bigr)  +  \gamma\Bigl(\bigl( {\Omega^{(k)}} \cap \Omega_i^+\bigr) \times \Omega_i\Bigr) + \gamma\Bigl(\bigl( \partial{\Omega^{(k)}} \cap \Omega_i^+ \bigr) \times \Omega_i\Bigr) \\
		&\le \frac{1}{s^p} \underbrace{\int \norm{x-y}^p \, \dif \gamma|_{\overline{\Omega^{(k)}} \times \Omega_i}}_{\eqqcolon \alpha_i} +  \int_{\Omega_i^+} \rho_{k,0} \dif \LebRd + \gamma\Bigl(\bigl( \partial{\Omega^{(k)}} \cap \Omega_i^+ \bigr) \times \Omega_i\Bigr) \fstop
	\end{align*}
	If~$i \in I_{k,s}$, then the last term is zero since~$\partial \Omega^{(k)} \cap \Omega_i^+ = \emptyset$ by the definitions of~$I_{k,s}$ and~$\Omega_i^+$. Moreover, since~$s < 2^{-k}$, we have
	\[
		\int_{\Omega_i^+} \rho_{k,0} \dif \LebRd = \sum_{j \in I_{k,0} \, : \, j \sim i} 2^{kd} \mu(\Omega_j) \abs{\Omega_i^+ \cap \Omega_j} \stackrel{\eqref{eq:thm:main1:1ter}}{\le} \mu(\Omega_i) + \sum_{j \in I_{k,0} \, : \, j \sim i} 2^ks \mu(\Omega_j) \fstop
	\]
	We thus obtain
	\[
		 n_i + N \le n \mu(\Omega_i) + N + n \left( \frac{\alpha_i}{s^p} + \sum_{j \in I_{k,0} \, : \, j \sim i} 2^ks \mu(\Omega_j) \right) \fstop
	\]
	The elementary inequality~$(a+b)^{-\zeta} \ge a^{-\zeta} -  \zeta \frac{b}{a^{\zeta+1}}$, which holds for every~$a,\zeta > 0$ and~$b \ge 0$, yields
	\begin{multline} \label{eq:thm:main1:6ter}
		\frac{\mu(\Omega_i)}{\bigl( n_i+N\bigr)^{p/d}} \ge  \frac{\mu(\Omega_i)}{ \bigl( n \mu(\Omega_i) + N \bigr)^{p/d}}  - \underbrace{\frac{p}{d}}_{<1} \cdot \underbrace{\frac{n \mu(\Omega_i) \displaystyle \left( \displaystyle \frac{\alpha_i}{s^p} + \displaystyle \sum_{j \in I_{k,0} \, : \, j \sim i} 2^ks \mu(\Omega_j) \right)}{\bigl( n \mu(\Omega_i) + N \bigr)^\frac{p+d}{d}}}_{\eqqcolon \beta_i} \comma \\ i \in I_{k,s} \fstop
	\end{multline}
	Note that
	\[
	\frac{\mu(\Omega_i)}{ \bigl( n \mu(\Omega_i) + N \bigr)^{p/d}} \ge \frac{n^{-p/d} \mu(\Omega_i)^\frac{d-p}{d}}{\left(1+\frac{N2^{kd}}{\epsilon_2n}\right)^{p/d}} \quad \text{if } \mu(\Omega_i) \ge \epsilon_2 2^{-kd} \fstop
		\]
	Let us focus on the sum of the last terms in~\eqref{eq:thm:main1:6ter}:
	\begin{align} \label{eq:thm:main1:7ter}\begin{split}
		\sum_{i \in I_{k,s} \, : \, \mu(\Omega_i) \ge \epsilon_2 2^{-kd}} \beta_i &\le \epsilon_2^{-p/d} n^{-p/d} 2^{kp} \sum_{i \in I_{k,s}} \left( \displaystyle \frac{\alpha_i}{s^p} + \sum_{j \in I_{k,0} \, : \, j \sim i} 2^ks \mu(\Omega_j) \right) \\
		&\le \epsilon_2^{-p/d} n^{-p/d} 2^{kp} \left( \frac{\sum_{i \in I_{k,s}} \alpha_i}{s^p} + \sum_{i,j \in I_{k,0} \, : \, i \sim j} 2^ks \mu(\Omega_j) \right) \\
		&\le \epsilon_2^{-p/d} n^{-p/d} 2^{kp} \left( \frac{1}{s^p} \int \norm{x-y}^p \, \dif \gamma + 3^d 2^k s \right) \fstop
		\end{split}
	\end{align}
	We plug these estimates into~\eqref{eq:thm:main1:5ter}, take the infimum over~$\gamma \in \Gamma b_{\Omega^{(k)}}(\rho_{k,0} , \mu_n)$, and find
	\begin{multline*}
		n^{p/d} Wb_{\Omega^{(k)},p}^p(\rho_{k,0}, \mu_n) \ge \frac{q_{p,d}^p  \left(1 - \epsilon_1\right)^{p+d}}{2^{kp}} \sum_{i \in I_{k,s} \, : \, \mu(\Omega_i) \ge \epsilon_2 2^{-kd}} \frac{\mu(\Omega_i)^\frac{d-p}{d}}{ \bigl( 1+\frac{N 2^{kd}}{\epsilon_2 n} \bigr)^{p/d}} \\
		\quad - q_{p,d}^p \underbrace{(1-\epsilon_1)^{p+d}}_{<1} \epsilon_2^{-p/d} \Biggl( \frac{1}{s^p} Wb_{\Omega^{(k)},p}^p(\rho_{k,0} , \mu_n) + 3^d 2^k s \Biggr) \fstop
	\end{multline*}
	Now we make a choice for the values of~$s$ and $k$. Thanks to \Cref{lemma:seqlemma}, \Cref{lemma:partition}, and the observation that~$\Omega^{(k)} \nearrow \R^d \setminus C$ as~$k \to \infty$, we can find~$k = k_n$ such that
	\begin{equation} \label{eq:thm:main1:8ter}
		\lim_{n \to \infty} n^{-1/d}2^{k_n} = \lim_{n \to \infty} n^{1/d} 2^{-k_n} \norm{\rho|_{\Omega^{(k_n)}} - \rho_{k_n,0}}_{L^1}^{1/p} = 0 \fstop
	\end{equation}
	We set $s_n \coloneqq \sqrt{2^{-k_n} n^{-1/d}}$, which is smaller than~$2^{-k_n}$, at least for large values of~$n$, and obtain
	\begin{multline*}
		\left( n^{p/d} + \frac{q_{p,d}^p \,  2^\frac{k_n p}{2} n^\frac{p}{2d}}{\epsilon_2^{p/d} } \right) Wb_{\Omega^{(k_n)},p}^p(\rho_{k_n,0} , \mu_n) \\
		\ge q_{p,d}^p \, \frac{ \left(1 - \epsilon_1\right)^{p+d}}{ \bigl( 1+\frac{N 2^{k_n d}}{\epsilon_2 n} \bigr)^{p/d}} \int_{\set{\rho_{k_n,s_n} \ge \epsilon_2}} \rho_{k_n,s_n}^{\frac{d-p}{d}} \, \dif \LebRd 
		-3^d q_{p,d}^p \, \epsilon_2^{-p/d} \sqrt{2^{k_n}n^{-1/d} } \fstop
	\end{multline*}
	If we pass to the limit, keeping \eqref{eq:thm:main1:8ter} in mind, we get
	\begin{align} \label{eq:thm:main1:9ter}
		\begin{split}
		\liminf_{n \to \infty} n^{p/d} Wb_{\Omega^{(k_n)},p}^p(\rho_{k_n,0} , \mu_n) &\ge q_{p,d}^p (1-\epsilon_1)^{p+d} \liminf_{n \to \infty} \int_{\set{\rho_{k_n,s_n} \ge \epsilon_2}}  \rho_{k_n,s_n}^\frac{d-p}{d} \, \dif \LebRd \\
		&\ge q_{p,d}^p (1-\epsilon_1)^{p+d} \int_{\set{\rho > \epsilon_2} \setminus C} \rho^\frac{d-p}{d} \, \dif \LebRd \comma
		\end{split}
	\end{align}
	where the last inequality follows from \Cref{lemma:partition} and Fatou's Lemma. By combining the formulas~\eqref{eq:thm:main1:2ter},~\eqref{eq:thm:main1:3ter}, and~\eqref{eq:thm:main1:9ter}, and by arbitrariness of~$\epsilon_2,\epsilon_1$, we conclude:
	\[
		\liminf_{n \to \infty} n^{1/d} \eqe{p}{n}(\mu) \ge q_{p,d} \left( \int_{\R^d \setminus C} \rho^\frac{d-p}{d} \, \dif x \right)^{1/p} - c_{p,d} \underbrace{\limsup_{n \to \infty} \frac{\norm{\rho|_{\Omega^{(k_n)}}-\rho_{k_n,0}}_{L^1}^{1/p}}{2^{k_n} n^{-1/d}}}_{=0} \fstop \qedhere
	\]
\end{proof}

\section{Limit existence for uniform measures} \label{sec:limitUnif}

Combining the upper bound~\eqref{eq:mainabove} and the existence of the limit for the uniform measure on a cube, it is possible to prove (for~$p < d$) the existence of the limit for \emph{any} uniform measure on a bounded set. The proof is inspired by \cite[Theorem~24]{BartheBordenave13}.%

\begin{corollary} \label{cor:uniform}
	If~$p < d$ and~$A \subseteq \R^d$ is a bounded Borel set with~$\abs{A} \neq 0$, then
	\begin{equation}
		\lim_{n \to \infty} n^{1/d} \eqe{p}{n}(U_A) = \eqc{p}{d} \abs{A}^{1/d} \fstop
	\end{equation}
\end{corollary}

\begin{proof}
	Note that this result easily follows from \Cref{prop:uniformCube} if~$A$ is a cube. Moreover, in general, one inequality is already given by \eqref{eq:mainabove} in \Cref{thm:main1}.
	
	We may and will assume that~$A$ is contained in (and not essentially equal to)~$[0,1]^d$. Consider the measures:
	\begin{itemize}
	\item $\mu^1 \coloneqq U_d|_{A} = \abs{A} U_{A}$,
	\item $\mu^2 \coloneqq U_d-\mu^1 = \left( 1-\abs{A} \right) U_{[0,1]^d \setminus A}$.
	\end{itemize}
	For~$n \in \N_1$, define
	\[
		\hat n \coloneqq \left \lfloor \frac{n}{\abs{A}} \right \rfloor +1 \comma \quad n_1 \coloneqq n \comma \quad n_2 \coloneqq \hat n - n - 1 \comma \quad n_0 \coloneqq 1 \fstop
	\]
	Observe that that~$0\le n_i \le \hat n \normTV{\mu^i}$ for~$i \in \set{1,2}$ and define
	\[
		\mu^{i,n} \coloneqq
		 \frac{n_i}{\hat n\, \normTV{\mu^i}} \mu^i \quad \text{for } i \in \set{1,2}\comma \qquad \mu^{0,n} \coloneqq U_d - \mu^{1,n} - \mu^{2,n} \fstop
	\]
	By definition of~$\eqc{p}{d}$ and \Cref{rmk:subaddErr}, we have
	\[
		\eqc{p}{d}^p\hat n^{-p/d} \le \eqe{p}{\hat n}^p(U_d) \le \sum_{i=0}^2 \eqe{p}{n_i}^p(\mu^{i,n}) \fstop
	\]
	The $0^\text{th}$ term at the right-hand side can be easily bounded:
	\[
		\eqe{p}{1}^p(\mu^{0,n}) \le W_p^p\left(\mu^{0,n}, \normTV{\mu^{0,n}} \delta_0 \right) = \int \norm{x}^p \, \dif \mu^{0,n} \le \frac{d^{p/2}}{\hat n} \fstop
	\]
	Hence,
	\begin{equation*}
		\eqc{p}{d}^p \le d^{p/2} \left(\frac{1}{\hat n}\right)^\frac{d-p}{d} + \left(\frac{n}{\hat n}\right)^{\frac{d-p}{d}} n^{p/d} \eqe{p}{n}^p\left( U_{A} \right) + \left(\frac{n_2}{\hat n}\right)^\frac{d-p}{d} n_2^{p/d} \eqe{p}{n_2}\bigl(U_{[0,1]^d \setminus A}\bigr) \comma
	\end{equation*}
	which yields
	\begin{align*}
		\eqc{p}{d}^p &\le \abs{A}^\frac{d-p}{d}\liminf_{n \to \infty} n^{p/d} \eqe{p}{n}^p(U_A) + (1-\abs{A})^\frac{d-p}{d} \limsup_{n \to \infty} n^{p/d} \eqe{p}{n}\bigl(U_{[0,1]^d \setminus A} \bigr) \\
		&\stackrel{\eqref{eq:mainabove}}{\le} \abs{A}^\frac{d-p}{d} \liminf_{n \to \infty} n^{p/d} \eqe{p}{n}^p(U_A) + \eqc{p}{d}^p (1-\abs{A}) \fstop
	\end{align*}
	We conclude by rearranging the terms.
\end{proof}

\section{Proof of \Cref{thm:q2}} \label{sec:limit}

The proof of \Cref{thm:q2} is based on a fundamental result by L.~Fejes T\'oth~\cite[p.~81]{FejesToth53}, see also~\cite{Gruber99}.

\begin{theorem}[L.~Fejes T\'oth~\cite{FejesToth53}] \label{thm:fejestoth}
	Let~$f \colon [0,\infty) \to \R$ be a nondecreasing function, let~$H \subseteq \R^2$ be a convex hexagon centered at the origin, let~$n \in \N_1$, and let~$x_1,\dots,x_n \in \R^2$. Then
	\begin{equation}
		\int_{H} f\bigl(\norm{x}\bigr) \, \dif x \le \frac{1}{n} \int_{\sqrt{n}H} \min\set{f\bigl(\norm{x-x_i}\bigr) \, : \, i \in \set{1,\dots,n} } \, \dif x  \fstop
	\end{equation}
\end{theorem}

Let us fix~$n \in \N_1$. To prove \Cref{thm:q2}, we consider a hexagonal tiling~$\set{H_{i,n}}_i$ of the plane, with the area of each~$H_{i,n}$ being equal to~$\abs{A}/n$. The idea is to define an empirical quantizer by taking the centers of the hexagons contained in~$A$. \Cref{thm:fejestoth} (together with \Cref{thm:BWZ}) is used to show that this quantizer is asymptotically optimal for the \emph{classical} quantization problem; hence,~$e_{p,n}(U_A) \le \eqe{p}{n}(U_A) \lessapprox e_{p,n}(U_A)$. The issue is that, in general, we cannot tile~$A$ perfectly with hexagons. Therefore, we carry out this construction only for the hexagons that are ``well-contained'' in~$A$ and leave out a strip of approximate thickness~$n^{-1/2}$. We complete the quantizer by splitting the strip into approximately (and up to constant)~$\sqrt{n}$ square-looking pieces~$\set{B_j}_j$ of approximate size~$n^{-1/2} \times n^{-1/2}$ and taking one point~$x_j$ from each piece. In this way, the contribution of the strip to the $p^\text{th}$ power of the quantization error is bounded by
\begin{align*}
	W_p^p\left(\frac{1}{n}\sum_j U(B_j), \frac{1}{n} \sum_{j} \delta_{x_j}  \right) &\stackrel{\eqref{eq:subadditivityWasserstein}}{\le} \frac{1}{n} \sum_j W_p^p\bigl(U(B_j), \delta_{x_j}\bigr) \le \frac{1}{n} \sum_j \diam(B_j)^p \\
	&\lessapprox \frac{1}{n} \sqrt{n} n^{-p/2} \comma
\end{align*}
which is negligible, i.e.,~much smaller than~$n^{-p/2}$.

We use the bi-Lipschitz map to make the argument rigorous, by transforming the strip into a more explicit approximate annulus.

\begin{lemma} \label{lemma:jordan}
	Let~$ D \subseteq \R^2$ be the open unit disk and~$\overline D$ be its closure. Let~$T \colon  \overline D \to \R^2$ be a homeomorphism onto its image. Then~$\partial \bigl(T(D)\bigr) = T(\partial D)$.
\end{lemma}

\begin{proof}
	By the Jordan--Sch\"onflies Theorem (cf.~\cite[Theorem~3.1]{Thomassen92}), there exists a homeomorphism~$\Phi \colon \R^2 \to \R^2$ such that~$\Phi|_{\partial D} = T|_{\partial D}$. In particular, the connected set~$\Phi^{-1}\bigl(T(D)\bigr)$ is contained in~$\R^2 \setminus \partial D$, implying that, in fact, it is entirely contained in either~$D$ or~$\R^2 \setminus \overline D$. The latter case is impossible: any retraction~$r_1 \colon \R^2 \setminus D \to \partial D$ would induce a retraction~$r_1 \circ \Phi^{-1} \circ T \colon \overline D \to \partial D$, which is absurd by the No-Retraction Theorem. Hence,~$\Phi^{-1}\bigl(T(D)\bigr) \subseteq D$. We claim that equality holds. Indeed, if there exists~$z \in D \setminus \Phi^{-1}\bigl(T(D)\bigr)$, then we can find a retraction~$r_2 \colon \overline D \setminus \set{z} \to \partial D$, which induces a retraction~$r_2 \circ \Phi^{-1} \circ T \colon \overline D \to \partial D$. Once again, this is absurd. Thus,~$T(D) = \Phi(D)$ and, using that~$\Phi$ is a homeomorphism and that it coincides with~$T$ on~$\partial D$, we conclude:
	\[
		\partial \bigl(T(D)\bigr) = \partial \bigl( \Phi(D) \bigr) = \Phi(\partial D) = T(\partial D) \fstop \qedhere
	\]
\end{proof}

\begin{proof}[Proof of \Cref{thm:q2}]
	Fix a regular hexagon~$H \subseteq \R^2$ with unit area, centered at the origin. For~$n \in \N_1$, choose~$x_1,\dots,x_n \in \R^2$. By \Cref{thm:fejestoth} with~$f(t) \coloneqq t^p$, we have
	\begin{equation*}
		\int_H \norm{x}^p \, \dif x \le \frac{1}{n} \int_{\sqrt{n}H} \min_{i} \norm{x-\sqrt{n} x_i}^p \, \dif x
		= n^{p/2} \int_{H} \min_{i} \norm{x-x_i}^p \, \dif x \fstop
	\end{equation*}
	Hence, by arbitrariness of the points~$x_1,\dots,x_n$ and by \Cref{thm:BWZ},
	\begin{equation*}
		\int_H \norm{x}^p \, \dif x \le \lim_{n \to \infty} n^{p/2} e_{p,n}^p(U_H) = q_{p,2}^p\fstop
	\end{equation*}
	Therefore, again thanks to \Cref{thm:BWZ}, it will suffice to prove that
	\begin{equation} \label{eq:thm:q2:1}
		\limsup_{n \to \infty} n^{p/2} \eqe{p}{n}^p(U_A) \le \abs{A}^{p/2} \int_H \norm{x}^p \, \dif x \fstop
	\end{equation}

	We can and will assume that~$\abs{A} = 1$. Let~$T \colon  \overline D \to A$ be a (bijective) bi-Lipschitz map and let~$c_T > 1$ be a Lipschitz constant for both~$T$ and~$T^{-1}$. %
	Let~$\set{H_i}_{i \in \N_0}$ be a family of regular, unit area, pairwise disjoint hexagons that cover~$\R^2$. For~$n \in \N_1$, define~$H_{i,n} \coloneqq H_i / \sqrt{n}$ and
	\begin{equation} \label{eq:thm:q2:2}
		I_n \coloneqq \set{i \in \N_0 \, : \, H_{i,n} \subseteq T(D) \text{ and } \dist \left(H_{i,n}, T(\partial D) \right) > \frac{c_T}{\sqrt{n}} } \comma \quad  A_n \coloneqq \bigcup_{i \in I_n} H_{i,n} \fstop
	\end{equation}
	Note that we have~$A_n \subseteq T(D)$ and
	\begin{equation} \label{eq:thm:q2:2bis}
		D \setminus T^{-1}(A_n) \subseteq \biggl\{x \in D \, : \dist(x,\partial D) \le n^{-1/2} \underbrace{c_T \, \bigl(c_T + \diam(H)\bigr)}_{\eqqcolon \bar c} \biggr\} \fstop
	\end{equation}
	Indeed, for every~$x \in D \setminus T^{-1}(A_n)$ there exists~$i \not\in I_n$ such that~$T(x) \in H_{i,n}$. There are two cases. If
		$\dist\left(H_{i,n}, T(\partial D) \right) \le \frac{c_T}{\sqrt{n}}$,
		then
		\begin{equation*} \inf_{y \in \partial D} \norm{x-y} \le c_T \inf_{y \in \partial D} \norm{T(x)-T(y)} \le c_T \underbrace{\diam\left(H_{i,n}\right)}_{=n^{-1/2}\diam(H)} + c_T \underbrace{\dist\left(H_{i,n}, T(\partial D) \right)}_{\le n^{-1/2}c_T}  \fstop \end{equation*}
	If~$H_{i,n} \not \subseteq T(D)$, then, since~$H_{i,n}$ is connected, we have~$H_{i,n} \cap \partial \bigl(T(D)\bigr) \neq \emptyset$. By \Cref{lemma:jordan}, we know that~$\partial \bigl(T(D)\bigr) = T(\partial D)$; hence~$
		\dist(x,\partial D) \le c_T\diam\left(H_{i,n}\right)$.

	By the measure-theoretic properties of Lipschitz maps (cf.~\cite[Theorem~7.5]{Mattila95}), we have~$\abs{T(\partial D)} = 0$ and
	\begin{equation} \label{eq:thm:q2:3}
		\abs{A \setminus A_n} = \abs{T(D) \setminus A_n} \le c_T^2 \abs{D \setminus T^{-1}(A_n)} \to 0 \quad \text{as } n \to \infty \text{ by~\eqref{eq:thm:q2:2bis}.} 
	\end{equation}
	
	Let~$k_n \coloneqq \#I_n$. By \Cref{rmk:subaddErr}, we have the inequality
	\[
		\eqe{p}{n}^p(U_A) \le \sum_{i \in I_n} \eqe{p}{1}^p\bigl(U_A|_{H_{i,n}}\bigr) + \eqe{p}{n-k_n}^p\bigl(U_A|_{A \setminus A_n}\bigr) \comma
	\]
	and the sum over~$i \in I_n$ is easy to bound:
	\begin{equation*}
		\sum_{i \in I_n} \eqe{p}{1}^p\bigl(U_A|_{H_{i,n}}\bigr) \le k_n \int_{H/\sqrt{n}} \norm{x}^p \, \dif x \le n^{-p/2}  \int_H \norm{x}^p \, \dif x  \fstop
	\end{equation*}
	Therefore, \eqref{eq:thm:q2:1} is verified once we prove that
	\begin{equation} \label{eq:thm:q2:4}
		\limsup_{n \to \infty} n^{p/2} \eqe{p}{n-k_n}^p\bigl(U_A|_{A \setminus A_n}\bigr) = 0 \fstop
	\end{equation}

	From now on, we use polar coordinates on~$\overline D$. Consider the function
	\begin{equation} \label{eq:thm:q2:5}
		g(\theta) \coloneqq \abs{T \set{(r,\phi) \in D \setminus T^{-1}(A_n) \, : \, \phi \in [0,\theta]}} \comma \qquad \theta \in [0,2\pi] \comma
	\end{equation}
	and note that~$g$ is continuous,~$g(0) = 0$, and~$g(2\pi) = \abs{T(D) \setminus A_n} = \abs{A \setminus A_n} = 1-\frac{k_n}{n}$. Furthermore,~$g$ is strictly increasing: for~$0 \le \theta_1 < \theta_2 \le 2\pi$, we have
	\begin{align*}
		g(\theta_2) - g(\theta_1) &=  \abs{T\set{(r,\phi) \in D \setminus T^{-1}(A_n) \, : \, \phi \in (\theta_1,\theta_2]}} \\
		&\ge c_T^{-2} \abs{\set{(r,\phi) \in D \setminus T^{-1} (A_n) \, : \, \phi \in (\theta_1,\theta_2]}} \\
		&\ge c_T^{-2} \abs{\set{(r,\phi) \in D \setminus T^{-1}( A_n) \, : \, r \in [1-n^{-1/2},1] \comma \phi \in (\theta_1,\theta_2]}} \\
		&\ge c_T^{-2} \abs{\set{(r,\phi) \in D \, : \, r \in [1-n^{-1/2},1] \comma \phi \in (\theta_1,\theta_2]}} \comma
	\end{align*}
	where the last inequality follows from the definition of~$A_n$ in~\eqref{eq:thm:q2:2}. Indeed, if~$1-r \le n^{-1/2}$ and~$T(r,\phi) \in H_{i,n}$, then
	\[
		\dist\left(H_{i,n}, T(\partial D) \right) \le \norm{T(r,\phi) - T(1,\phi)} \le c_T \norm{(r,\phi) - (1,\phi)} \le \frac{c_T}{\sqrt{n}} \fstop
	\]
	Therefore,
	\begin{equation} \label{eq:thm:q2:6}
		g(\theta_2) - g(\theta_1) \ge c_T^{-2} \frac{\theta_2-\theta_1}{\sqrt{n}} \left( 1-\frac{1}{2\sqrt{n}} \right) \ge c_T^{-2} \frac{\theta_2-\theta_1}{2 \sqrt{n}} \fstop
	\end{equation}

	Let us define
	\begin{equation} \label{eq:thm:q2:7}
	\begin{alignedat}{2} 
		\bar \theta_j &\coloneqq g^{-1}(j/n) \comma &&\qquad  j \in \set{0,\dots,n-k_n} \\
		B_j &\coloneqq T\set{(r,\phi) \in D \setminus T^{-1}(A_n) \, : \, \phi \in (\bar \theta_{j-1}, \bar \theta_j]} \comma &&\qquad j \in \set{1,\dots,n-k_n} \fstop
	\end{alignedat}
	\end{equation}
	\begin{figure}
		\centering
		\includegraphics[trim={0 4cm 0 4cm},clip,width = 0.8\textwidth]{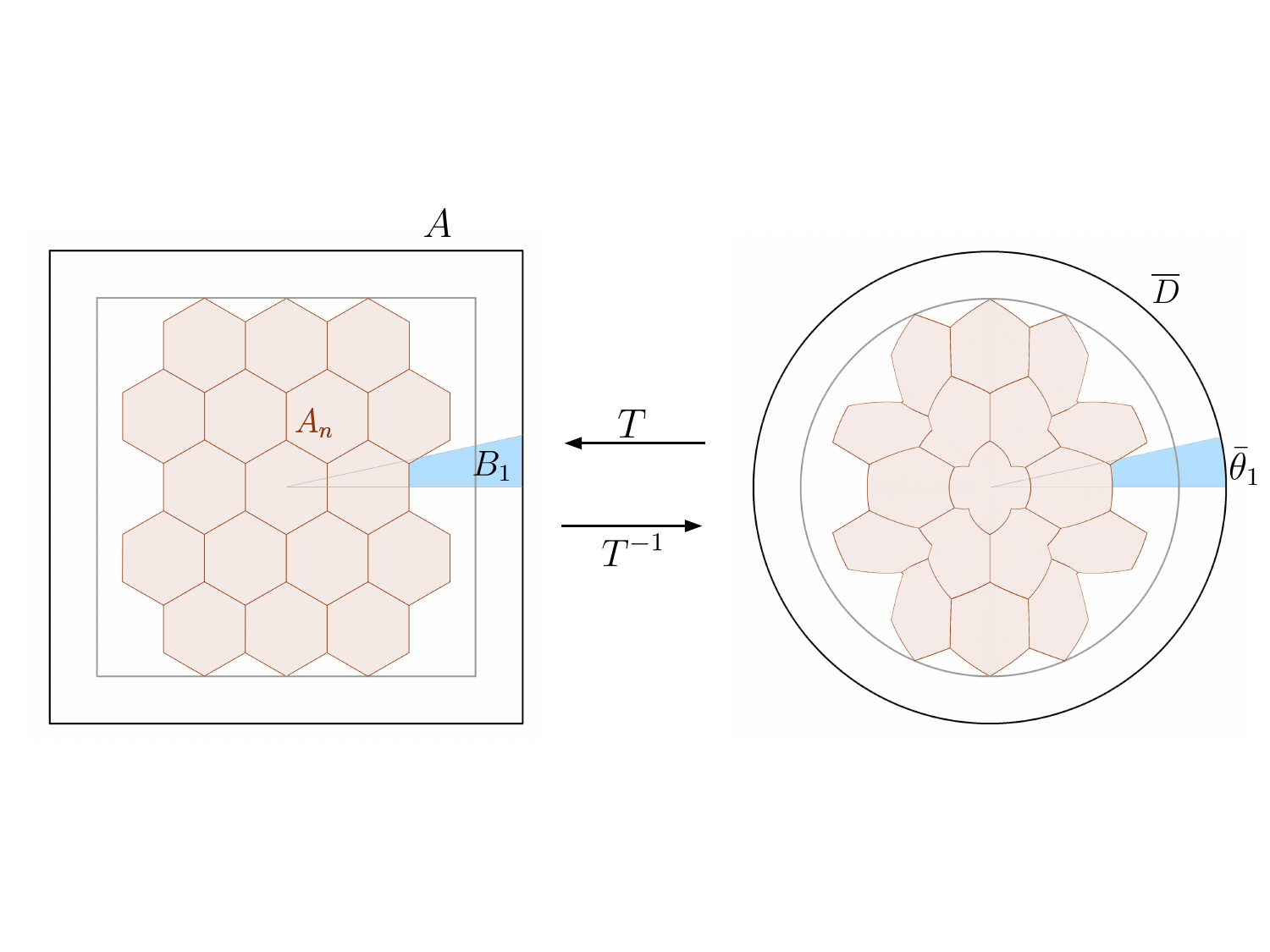}
		\caption{Idea for \Cref{thm:q2}. We select the hexagons inside~$T(D)$ that are sufficiently far from~$T(\partial D)$. We define the sets~$B_i$ as~$T$-images of intersections of~$D \setminus T^{-1}(A_n)$ with angles.}
		\label{fig:hexagons}
	\end{figure}
	These sets enjoy two important properties: firstly, by~\eqref{eq:thm:q2:5} and~\eqref{eq:thm:q2:7},
	\begin{equation} \label{eq:thm:q2:8}
		\abs{B_j} = g(\bar \theta_j) - g(\bar \theta_{j-1}) = \frac{1}{n} \semicolon
	\end{equation}
	secondly, by~\eqref{eq:thm:q2:2bis},
	\begin{equation*}
		B_j %
		\subseteq T\set{(r,\phi) \in D \, : \, r \ge 1-n^{-1/2}\bar c \comma \phi \in (\bar \theta_{j-1}, \bar \theta_j]} \comma
	\end{equation*}
	which implies\footnote{To move from one point of~$ D$ to another, one can (inefficiently) walk radially up to the circle~$\partial D$, then along~$\partial D$, and radially again.}
	\begin{align} \label{eq:thm:q2:9} \begin{split}
		\diam(B_j) &\le c_T(\bar \theta_j - \bar \theta_{j-1}) + 2c_T \frac{\bar c}{\sqrt{n}} \stackrel{\eqref{eq:thm:q2:6}}{\le} 2c_T^3 \sqrt{n} \bigl( g(\bar \theta_j) - g(\bar \theta_{j-1}) \bigr) + 2c_T \frac{\bar c}{\sqrt{n}} \\
		&\stackrel{\eqref{eq:thm:q2:7}}{=} 2c_T^3 \frac{1}{\sqrt{n}} + 2c_T \frac{\bar c}{\sqrt{n}} \lesssim_T n^{-1/2} \fstop \end{split}
	\end{align}
	We can conclude: by \Cref{rmk:subaddErr},
	\begin{align*}
		\eqe{p}{n-k_n}^p\bigl(U_A|_{A \setminus A_n}\bigr) &\stackrel{\eqref{eq:thm:q2:8}}{\le} \sum_{j=1}^{n-k_n} \eqe{p}{1}^p\bigl(U_A|_{B_j}\bigr) \le \sum_{i=1}^{n-k_n} \frac{\diam(B_j)^p}{n} \\
		&\stackrel{\eqref{eq:thm:q2:9}}{\lesssim_{T,p}} \left(1-\frac{k_n}{n}\right) n^{-p/2} = \abs{A \setminus A_n} n^{-p/2} \comma
	\end{align*}
	which, together with~\eqref{eq:thm:q2:3}, implies~\eqref{eq:thm:q2:4}. \qedhere
	
\end{proof}

\bibliographystyle{myplainnat}

\begin{thebibliography}{101}
	\providecommand{\natexlab}[1]{#1}
	\providecommand{\url}[1]{\texttt{#1}}
	\expandafter\ifx\csname urlstyle\endcsname\relax
	\providecommand{\doi}[1]{doi: #1}\else
	\providecommand{\doi}{doi: \begingroup \urlstyle{rm}\Url}\fi
	
	\bibitem[Aistleitner and Dick(2014)]{AistleitnerDick14}
	C.~Aistleitner and J.~Dick.
	\newblock Low-discrepancy point sets for non-uniform measures.
	\newblock \emph{Acta Arith.}, 163\penalty0 (4):\penalty0 345--369, 2014.
	\newblock \doi{10.4064/aa163-4-4}.
	
	\bibitem[Aistleitner and Dick(2015)]{AistleitnerDick15}
	C.~Aistleitner and J.~Dick.
	\newblock Functions of bounded variation, signed measures, and a general
	{K}oksma-{H}lawka inequality.
	\newblock \emph{Acta Arith.}, 167\penalty0 (2):\penalty0 143--171, 2015.
	\newblock \doi{10.4064/aa167-2-4}.
	
	\bibitem[Aistleitner et~al.(2018)Aistleitner, Bilyk, and
	Nikolov]{AistleitnerBilykNikolov18}
	C.~Aistleitner, D.~Bilyk, and A.~Nikolov.
	\newblock Tusn\'{a}dy's problem, the transference principle, and non-uniform
	{QMC} sampling.
	\newblock In \emph{Monte {C}arlo and quasi--{M}onte {C}arlo methods}, volume
	241 of \emph{Springer Proc. Math. Stat.}, pages 169--180. Springer, Cham,
	2018.
	\newblock \doi{10.1007/978-3-319-91436-7\_8}.
	
	\bibitem[Ambrosio et~al.(2019)Ambrosio, Stra, and
	Trevisan]{AmbrosioStraTrevisan19}
	L.~Ambrosio, F.~Stra, and D.~Trevisan.
	\newblock A {PDE} approach to a 2-dimensional matching problem.
	\newblock \emph{Probab. Theory Related Fields}, 173\penalty0 (1-2):\penalty0
	433--477, 2019.
	\newblock \doi{10.1007/s00440-018-0837-x}.
	
	\bibitem[Ambrosio et~al.(2022)Ambrosio, Goldman, and
	Trevisan]{AmbrosioGoldmanTrevisan22}
	L.~Ambrosio, M.~Goldman, and D.~Trevisan.
	\newblock On the quadratic random matching problem in two-dimensional domains.
	\newblock \emph{Electron. J. Probab.}, 27:\penalty0 Paper No. 54, 35, 2022.
	\newblock \doi{10.1214/22-ejp784}.
	
	\bibitem[Aurenhammer et~al.(1998)Aurenhammer, Hoffmann, and
	Aronov]{AurenhammerHoffmannAronov98}
	F.~Aurenhammer, F.~Hoffmann, and B.~Aronov.
	\newblock Minkowski-type theorems and least-squares clustering.
	\newblock \emph{Algorithmica}, 20\penalty0 (1):\penalty0 61--76, 1998.
	\newblock \doi{10.1007/PL00009187}.
	
	\bibitem[Aydin and Iacobelli(2024)]{AydinIacobelli24}
	A.~D. Aydin and M.~Iacobelli.
	\newblock Asymptotic quantization on {R}iemannian manifolds via covering growth
	estimates.
	\newblock \emph{arXiv preprint}, 2024.
	\newblock \doi{10.48550/arXiv.2402.13164}.
	
	\bibitem[Baker(2015)]{Baker15}
	D.~M. Baker.
	\newblock Quantizations of probability measures and preservation of the convex
	order.
	\newblock \emph{Statist. Probab. Lett.}, 107:\penalty0 280--285, 2015.
	\newblock \doi{10.1016/j.spl.2015.09.001}.
	
	\bibitem[Balzer et~al.(2009)Balzer, Schl\"{o}mer, and
	Deussen]{BalzerSchlomerDeussen09}
	M.~Balzer, T.~Schl\"{o}mer, and O.~Deussen.
	\newblock Capacity-constrained point distributions: a variant of {L}loyd's
	method.
	\newblock \emph{ACM Trans. Graph.}, 28\penalty0 (3), jul 2009.
	\newblock \doi{10.1145/1531326.1531392}.
	
	\bibitem[Barthe and Bordenave(2013)]{BartheBordenave13}
	F.~Barthe and C.~Bordenave.
	\newblock Combinatorial optimization over two random point sets.
	\newblock In \emph{S\'{e}minaire de {P}robabilit\'{e}s {XLV}}, volume 2078 of
	\emph{Lecture Notes in Math.}, pages 483--535. Springer, Cham, 2013.
	\newblock \doi{10.1007/978-3-319-00321-4\_19}.
	
	\bibitem[Beardwood et~al.(1959)Beardwood, Halton, and
	Hammersley]{BeardwoodHaltonHammersley59}
	J.~Beardwood, J.~H. Halton, and J.~M. Hammersley.
	\newblock The shortest path through many points.
	\newblock \emph{Proc. Cambridge Philos. Soc.}, 55:\penalty0 299--327, 1959.
	\newblock \doi{10.1017/s0305004100034095}.
	
	\bibitem[Beck(1984)]{Beck84}
	J.~Beck.
	\newblock Some upper bounds in the theory of irregularities of distribution.
	\newblock \emph{Acta Arith.}, 43\penalty0 (2):\penalty0 115--130, 1984.
	\newblock \doi{10.4064/aa-43-2-115-130}.
	
	\bibitem[Bencheikh and Jourdain(2022{\natexlab{a}})]{BencheikhJourdain22}
	O.~Bencheikh and B.~Jourdain.
	\newblock Approximation rate in {W}asserstein distance of probability measures
	on the real line by deterministic empirical measures.
	\newblock \emph{J. Approx. Theory}, 274:\penalty0 Paper No. 105684, 27,
	2022{\natexlab{a}}.
	\newblock \doi{10.1016/j.jat.2021.105684}.
	
	\bibitem[Bencheikh and Jourdain(2022{\natexlab{b}})]{BencheikhJourdain22bis}
	O.~Bencheikh and B.~Jourdain.
	\newblock Weak and strong error analysis for mean-field rank-based particle
	approximations of one-dimensional viscous scalar conservation laws.
	\newblock \emph{Ann. Appl. Probab.}, 32\penalty0 (6):\penalty0 4143--4185,
	2022{\natexlab{b}}.
	\newblock \doi{10.1214/21-aap1776}.
	
	\bibitem[Bennett(1948)]{Bennett48}
	W.~R. Bennett.
	\newblock Spectra of quantized signals.
	\newblock \emph{Bell System Tech. J.}, 27:\penalty0 446--472, 1948.
	\newblock \doi{10.1002/j.1538-7305.1948.tb01340.x}.
	
	\bibitem[Berger and Xu(2020)]{BergerXu20}
	A.~Berger and C.~Xu.
	\newblock Asymptotics of one-dimensional {L}\'{e}vy approximations.
	\newblock \emph{J. Theoret. Probab.}, 33\penalty0 (2):\penalty0 1164--1195,
	2020.
	\newblock \doi{10.1007/s10959-019-00893-1}.
	
	\bibitem[Berger et~al.(2008)Berger, Hill, and Morrison]{BergerHillMorrison08}
	A.~Berger, T.~P. Hill, and K.~E. Morrison.
	\newblock Scale-distortion inequalities for mantissas of finite data sets.
	\newblock \emph{J. Theoret. Probab.}, 21\penalty0 (1):\penalty0 97--117, 2008.
	\newblock \doi{10.1007/s10959-007-0112-z}.
	
	\bibitem[Bobkov and Ledoux(2021)]{BobkovLedoux21}
	S.~G. Bobkov and M.~Ledoux.
	\newblock A simple {F}ourier analytic proof of the {AKT} optimal matching
	theorem.
	\newblock \emph{Ann. Appl. Probab.}, 31\penalty0 (6):\penalty0 2567--2584,
	2021.
	\newblock \doi{10.1214/20-aap1656}.
	
	\bibitem[Bollob\'{a}s(1973)]{Bollobas73}
	B.~Bollob\'{a}s.
	\newblock The optimal arrangement of producers.
	\newblock \emph{J. London Math. Soc. (2)}, 6:\penalty0 605--613, 1973.
	\newblock \doi{10.1112/jlms/s2-6.4.605}.
	
	\bibitem[Bollob\'{a}s and Stern(1972)]{BollobasStern72}
	B.~Bollob\'{a}s and N.~Stern.
	\newblock The optimal structure of market areas.
	\newblock \emph{J. Econom. Theory}, 4\penalty0 (2):\penalty0 174--179, 1972.
	\newblock \doi{10.1016/0022-0531(72)90147-0}.
	
	\bibitem[Bonneel and Digne(2023)]{BonneelDigne23}
	N.~Bonneel and J.~Digne.
	\newblock A survey of optimal transport for computer graphics and computer
	vision.
	\newblock \emph{Computer Graphics Forum}, 42\penalty0 (2):\penalty0 439--460,
	2023.
	\newblock \doi{https://doi.org/10.1111/cgf.14778}.
	
	\bibitem[Brown and Steinerberger(2020)]{BrownSteinerberger20}
	L.~Brown and S.~Steinerberger.
	\newblock Positive-definite functions, exponential sums and the greedy
	algorithm: a curious phenomenon.
	\newblock \emph{J. Complexity}, 61:\penalty0 101485, 17, 2020.
	\newblock \doi{10.1016/j.jco.2020.101485}.
	
	\bibitem[Bucklew and Wise(1982)]{BucklewWise82}
	J.~A. Bucklew and G.~L. Wise.
	\newblock Multidimensional asymptotic quantization theory with {$r$}th power
	distortion measures.
	\newblock \emph{IEEE Trans. Inform. Theory}, 28\penalty0 (2):\penalty0
	239--247, 1982.
	\newblock \doi{10.1109/TIT.1982.1056486}.
	
	\bibitem[Buttazzo and Santambrogio(2009)]{ButtazzoSantambrogio09}
	G.~Buttazzo and F.~Santambrogio.
	\newblock A mass transportation model for the optimal planning of an urban
	region.
	\newblock \emph{SIAM Rev.}, 51\penalty0 (3):\penalty0 593--610, 2009.
	\newblock \doi{10.1137/090759197}.
	
	\bibitem[Caglioti et~al.(2024)Caglioti, Goldman, Pieroni, and
	Trevisan]{CagliotiGoldmanPieroniTrevisan24}
	E.~Caglioti, M.~Goldman, F.~Pieroni, and D.~Trevisan.
	\newblock Subadditivity and optimal matching of unbounded samples.
	\newblock \emph{arXiv preprint}, 2024.
	\newblock \doi{10.48550/arXiv.2407.06352}.
	
	\bibitem[Chen et~al.(2018)Chen, Liu, and Wang]{ChenLiuWang18}
	S.~Chen, J.~Liu, and X.-J. Wang.
	\newblock Boundary regularity for the second boundary-value problem of
	monge-amp\`ere equations in dimension two.
	\newblock \emph{arXiv preprint}, 2018.
	\newblock \doi{10.48550/arXiv.1806.09482}.
	
	\bibitem[Chen et~al.(2019)Chen, Liu, and Wang]{ChenLiuWang19}
	S.~Chen, J.~Liu, and X.-J. Wang.
	\newblock Global regularity of optimal mappings in non-convex domains.
	\newblock \emph{Sci. China Math.}, 62\penalty0 (11):\penalty0 2057--2072, 2019.
	\newblock \doi{10.1007/s11425-018-9465-8}.
	
	\bibitem[Chen et~al.(2021)Chen, Liu, and Wang]{ChenLiuWang21}
	S.~Chen, J.~Liu, and X.-J. Wang.
	\newblock Global regularity for the {M}onge-{A}mp\`ere equation with natural
	boundary condition.
	\newblock \emph{Ann. of Math. (2)}, 194\penalty0 (3):\penalty0 745--793, 2021.
	\newblock \doi{10.4007/annals.2021.194.3.4}.
	
	\bibitem[Chevallier et~al.(2019)Chevallier, Duarte, L\"{o}cherbach, and
	Ost]{ChevallierDuarteLocherbachOst19}
	J.~Chevallier, A.~Duarte, E.~L\"{o}cherbach, and G.~Ost.
	\newblock Mean field limits for nonlinear spatially extended {H}awkes processes
	with exponential memory kernels.
	\newblock \emph{Stochastic Process. Appl.}, 129\penalty0 (1):\penalty0 1--27,
	2019.
	\newblock \doi{10.1016/j.spa.2018.02.007}.
	
	\bibitem[Chevallier(2018)]{Chevallier18}
	J.~Chevallier.
	\newblock Uniform decomposition of probability measures: quantization,
	clustering and rate of convergence.
	\newblock \emph{J. Appl. Probab.}, 55\penalty0 (4):\penalty0 1037--1045, 2018.
	\newblock \doi{10.1017/jpr.2018.69}.
	
	\bibitem[Cohn(2013)]{Cohn13}
	D.~L. Cohn.
	\newblock \emph{Measure theory}.
	\newblock Birkh\"{a}user Advanced Texts: Basler Lehrb\"{u}cher.
	Birkh\"{a}user/Springer, New York, second edition, 2013.
	\newblock \doi{10.1007/978-1-4614-6956-8}.
	
	\bibitem[Cort\'es(2010)]{Cortes10}
	J.~Cort\'es.
	\newblock Coverage optimization and spatial load balancing by robotic sensor
	networks.
	\newblock \emph{IEEE Transactions on Automatic Control}, 55\penalty0
	(3):\penalty0 749--754, 2010.
	\newblock \doi{10.1109/TAC.2010.2040495}.
	
	\bibitem[Cort\'es and Egerstedt(2017)]{CortesEgerstedt17}
	J.~Cort\'es and M.~Egerstedt.
	\newblock Coordinated control of multi-robot systems: A survey.
	\newblock \emph{SICE Journal of Control, Measurement, and System Integration},
	10\penalty0 (6):\penalty0 495--503, 2017.
	\newblock \doi{10.9746/jcmsi.10.495}.
	
	\bibitem[de~Goes et~al.(2012)de~Goes, Breeden, Ostromoukhov, and
	Desbrun]{deGoesBreedenOstromoukhovDesbrun12}
	F.~de~Goes, K.~Breeden, V.~Ostromoukhov, and M.~Desbrun.
	\newblock Blue noise through optimal transport.
	\newblock \emph{ACM Trans. Graph.}, 31\penalty0 (6), nov 2012.
	\newblock \doi{10.1145/2366145.2366190}.
	
	\bibitem[Dereich(2009)]{Dereich09}
	S.~Dereich.
	\newblock Asymptotic formulae for coding problems and intermediate optimization
	problems: a review.
	\newblock In \emph{Trends in stochastic analysis}, volume 353 of \emph{London
		Math. Soc. Lecture Note Ser.}, pages 187--232. Cambridge Univ. Press,
	Cambridge, 2009.
	\newblock \doi{10.1017/CBO9781139107020.010}.
	
	\bibitem[Dereich et~al.(2013)Dereich, Scheutzow, and
	Schottstedt]{DereichScheutzowSchottstedt13}
	S.~Dereich, M.~Scheutzow, and R.~Schottstedt.
	\newblock Constructive quantization: approximation by empirical measures.
	\newblock \emph{Ann. Inst. Henri Poincar\'{e} Probab. Stat.}, 49\penalty0
	(4):\penalty0 1183--1203, 2013.
	\newblock \doi{10.1214/12-AIHP489}.
	
	\bibitem[Ding et~al.(2023)Ding, Zhou, Wang, Feng, Tang, Shang, Xin, Jian, Gude,
	and Xu]{DingEtAl23}
	Y.~Ding, X.~Zhou, J.~Wang, Y.~Feng, J.~Tang, N.~Shang, S.~Xin, X.~Jian,
	M.~Gude, and J.~Xu.
	\newblock A sophisticated periodic micro-model for closed-cell foam based on
	centroidal constraint and capacity constraint.
	\newblock \emph{Composite Structures}, 303:\penalty0 116175, 2023.
	\newblock \doi{https://doi.org/10.1016/j.compstruct.2022.116175}.
	
	\bibitem[Eckstein and Nutz(2024)]{EcksteinNutz24}
	S.~Eckstein and M.~Nutz.
	\newblock Convergence rates for regularized optimal transport via quantization.
	\newblock \emph{Mathematics of Operations Research}, 49\penalty0 (2):\penalty0
	1223--1240, 2024.
	\newblock \doi{10.1287/moor.2022.0245}.
	
	\bibitem[Fairchild et~al.(2021)Fairchild, Goering, and
	Weiss]{FairchildGoeringWeiss21}
	S.~Fairchild, M.~Goering, and C.~Weiss.
	\newblock Families of well approximable measures.
	\newblock \emph{Unif. Distrib. Theory}, 16\penalty0 (1):\penalty0 53--70, 2021.
	\newblock \doi{10.2478/udt-2021-0003}.
	
	\bibitem[Fejes~T\'{o}th(1953)]{FejesToth53}
	L.~Fejes~T\'{o}th.
	\newblock \emph{Lagerungen in der {E}bene, auf der {K}ugel und im {R}aum},
	volume Band LXV of \emph{Die Grundlehren der mathematischen Wissenschaften in
		Einzeldarstellungen mit besonderer Ber\"{u}cksichtigung der
		Anwendungsgebiete}.
	\newblock Springer-Verlag, Berlin-G\"{o}ttingen-Heidelberg, 1953.
	\newblock \doi{10.1007/978-3-642-65234-9}.
	
	\bibitem[Figalli and Gigli(2010)]{FigalliGigli10}
	A.~Figalli and N.~Gigli.
	\newblock A new transportation distance between non-negative measures, with
	applications to gradients flows with {D}irichlet boundary conditions.
	\newblock \emph{J. Math. Pures Appl. (9)}, 94\penalty0 (2):\penalty0 107--130,
	2010.
	\newblock \doi{10.1016/j.matpur.2009.11.005}.
	
	\bibitem[Gersho(1979)]{Gersho79}
	A.~Gersho.
	\newblock Asymptotically optimal block quantization.
	\newblock \emph{IEEE Trans. Inform. Theory}, 25\penalty0 (4):\penalty0
	373--380, 1979.
	\newblock \doi{10.1109/TIT.1979.1056067}.
	
	\bibitem[Gersho and Gray(1992)]{GershoGray92}
	A.~Gersho and R.~M. Gray.
	\newblock \emph{Vector Quantization and Signal Compression}.
	\newblock The Springer International Series in Engineering and Computer
	Science. Springer New York, NY, 1992.
	\newblock \doi{10.1007/978-1-4615-3626-0}.
	
	\bibitem[Giles and Sheridan-Methven(2023)]{GilesSheridan-Methven23}
	M.~Giles and O.~Sheridan-Methven.
	\newblock Approximating inverse cumulative distribution functions to produce
	approximate random variables.
	\newblock \emph{ACM Trans. Math. Software}, 49\penalty0 (3):\penalty0 Art. 26,
	29, 2023.
	\newblock \doi{10.1145/3604935}.
	
	\bibitem[Giles et~al.(2019{\natexlab{a}})Giles, Hefter, Mayer, and
	Ritter]{GilesHefterMayerRitter19}
	M.~B. Giles, M.~Hefter, L.~Mayer, and K.~Ritter.
	\newblock Random bit quadrature and approximation of distributions on {H}ilbert
	spaces.
	\newblock \emph{Found. Comput. Math.}, 19\penalty0 (1):\penalty0 205--238,
	2019{\natexlab{a}}.
	\newblock \doi{10.1007/s10208-018-9382-3}.
	
	\bibitem[Giles et~al.(2019{\natexlab{b}})Giles, Hefter, Mayer, and
	Ritter]{GilesHefterMayerRitter19bis}
	M.~B. Giles, M.~Hefter, L.~Mayer, and K.~Ritter.
	\newblock Random bit multilevel algorithms for stochastic differential
	equations.
	\newblock \emph{J. Complexity}, 54:\penalty0 101395, 21, 2019{\natexlab{b}}.
	\newblock \doi{10.1016/j.jco.2019.01.002}.
	
	\bibitem[Gkogkas et~al.(2023)Gkogkas, Kuehn, and Xu]{GkogkasKuehnXu23}
	M.~A. Gkogkas, C.~Kuehn, and C.~Xu.
	\newblock Mean field limits of co-evolutionary heterogeneous networks.
	\newblock \emph{arXiv preprint}, 2023.
	\newblock \doi{10.48550/arXiv.2202.01742}.
	
	\bibitem[Goldman and Trevisan(2024)]{GoldmanTrevisan24}
	M.~Goldman and D.~Trevisan.
	\newblock Optimal transport methods for combinatorial optimization over two
	random point sets.
	\newblock \emph{Probab. Theory Related Fields}, 188\penalty0 (3-4):\penalty0
	1315--1384, 2024.
	\newblock \doi{10.1007/s00440-023-01245-1}.
	
	\bibitem[Graf and Luschgy(2000)]{GrafLuschgy00}
	S.~Graf and H.~Luschgy.
	\newblock \emph{Foundations of quantization for probability distributions},
	volume 1730 of \emph{Lecture Notes in Mathematics}.
	\newblock Springer-Verlag, Berlin, 2000.
	\newblock \doi{10.1007/BFb0103945}.
	
	\bibitem[Graf et~al.(2012)Graf, Luschgy, and Pag\`es]{GrafLuschgyPages12}
	S.~Graf, H.~Luschgy, and G.~Pag\`es.
	\newblock The local quantization behavior of absolutely continuous
	probabilities.
	\newblock \emph{Ann. Probab.}, 40\penalty0 (4):\penalty0 1795--1828, 2012.
	\newblock \doi{10.1214/11-AOP663}.
	
	\bibitem[Gray and Neuhoff(1998)]{GrayNeuhoff98}
	R.~M. Gray and D.~L. Neuhoff.
	\newblock Quantization.
	\newblock \emph{IEEE Trans. Inform. Theory}, 44\penalty0 (6):\penalty0
	2325--2383, 1998.
	\newblock \doi{10.1109/18.720541}.
	
	\bibitem[Gruber(1999)]{Gruber99}
	P.~M. Gruber.
	\newblock A short analytic proof of {F}ejes {T}\'{o}th's theorem on sums of
	moments.
	\newblock \emph{Aequationes Math.}, 58\penalty0 (3):\penalty0 291--295, 1999.
	\newblock \doi{10.1007/s000100050116}.
	
	\bibitem[Gruber(2001)]{Gruber01}
	P.~M. Gruber.
	\newblock Optimal configurations of finite sets in {R}iemannian 2-manifolds.
	\newblock \emph{Geom. Dedicata}, 84\penalty0 (1-3):\penalty0 271--320, 2001.
	\newblock \doi{10.1023/A:1010358407868}.
	
	\bibitem[Gruber(2004)]{Gruber04}
	P.~M. Gruber.
	\newblock Optimum quantization and its applications.
	\newblock \emph{Adv. Math.}, 186\penalty0 (2):\penalty0 456--497, 2004.
	\newblock \doi{10.1016/j.aim.2003.07.017}.
	
	\bibitem[Iacobelli(2016)]{Iacobelli16}
	M.~Iacobelli.
	\newblock Asymptotic quantization for probability measures on {R}iemannian
	manifolds.
	\newblock \emph{ESAIM Control Optim. Calc. Var.}, 22\penalty0 (3):\penalty0
	770--785, 2016.
	\newblock \doi{10.1051/cocv/2015025}.
	
	\bibitem[Jourdain and Reygner(2016)]{JourdainReygner16}
	B.~Jourdain and J.~Reygner.
	\newblock Optimal convergence rate of the multitype sticky particle
	approximation of one-dimensional diagonal hyperbolic systems with monotonic
	initial data.
	\newblock \emph{Discrete Contin. Dyn. Syst.}, 36\penalty0 (9):\penalty0
	4963--4996, 2016.
	\newblock \doi{10.3934/dcds.2016015}.
	
	\bibitem[Kloeckner(2012)]{Kloeckner12}
	B.~Kloeckner.
	\newblock Approximation by finitely supported measures.
	\newblock \emph{ESAIM Control Optim. Calc. Var.}, 18\penalty0 (2):\penalty0
	343--359, 2012.
	\newblock \doi{10.1051/cocv/2010100}.
	
	\bibitem[Kuehn and Xu(2022)]{KuehnXu22}
	C.~Kuehn and C.~Xu.
	\newblock Vlasov equations on digraph measures.
	\newblock \emph{J. Differential Equations}, 339:\penalty0 261--349, 2022.
	\newblock \doi{10.1016/j.jde.2022.08.023}.
	
	\bibitem[Le~Brigant and Puechmorel(2019{\natexlab{a}})]{LeBrigantPuechmorel19}
	A.~Le~Brigant and S.~Puechmorel.
	\newblock Quantization and clustering on {R}iemannian manifolds with an
	application to air traffic analysis.
	\newblock \emph{J. Multivariate Anal.}, 173:\penalty0 685--703,
	2019{\natexlab{a}}.
	\newblock \doi{10.1016/j.jmva.2019.05.008}.
	
	\bibitem[Le~Brigant and
	Puechmorel(2019{\natexlab{b}})]{LeBrigantPuechmorel19bis}
	A.~Le~Brigant and S.~Puechmorel.
	\newblock Approximation of densities on {R}iemannian manifolds.
	\newblock \emph{Entropy}, 21\penalty0 (1), 2019{\natexlab{b}}.
	\newblock \doi{10.3390/e21010043}.
	
	\bibitem[Ledoux(2023)]{Ledoux23}
	M.~Ledoux.
	\newblock Optimal matching of random samples and rates of convergence of
	empirical measures.
	\newblock In \emph{Mathematics going forward---collected mathematical
		brushstrokes}, volume 2313 of \emph{Lecture Notes in Math.}, pages 615--627.
	Springer, Cham, 2023.
	\newblock \doi{10.1007/978-3-031-12244-6\_43}.
	
	\bibitem[Liu and Pfeiffer(2023)]{LiuPfeiffer23}
	K.~Liu and L.~Pfeiffer.
	\newblock Mean field optimization problems: stability results and {L}agrangian
	discretization.
	\newblock \emph{arXiv preprint}, 2023.
	\newblock \doi{10.48550/arXiv.2310.20037}.
	
	\bibitem[Liu et~al.(2024)Liu, Dong, Xiao, Chen, Hu, Zhu, Zhu, Sakai, and
	Wu]{LiuDongXiaoChenHuZhuZhuSakaiWu24}
	Q.~Liu, X.~Dong, J.~Xiao, N.~Chen, H.~Hu, J.~Zhu, C.~Zhu, T.~Sakai, and X.-M.
	Wu.
	\newblock Vector quantization for recommender systems: A review and outlook.
	\newblock \emph{arXiv preprint}, 2024.
	\newblock \doi{10.48550/arXiv.2405.03110}.
	
	\bibitem[Liu and Pag\`es(2020)]{LiuPages20}
	Y.~Liu and G.~Pag\`es.
	\newblock Convergence rate of optimal quantization and application to the
	clustering performance of the empirical measure.
	\newblock \emph{J. Mach. Learn. Res.}, 21:\penalty0 Paper No. 86, 36, 2020.
	
	\bibitem[Lloyd(1982)]{Lloyd82}
	S.~P. Lloyd.
	\newblock Least squares quantization in {PCM}.
	\newblock \emph{IEEE Trans. Inform. Theory}, 28\penalty0 (2):\penalty0
	129--137, 1982.
	\newblock \doi{10.1109/TIT.1982.1056489}.
	
	\bibitem[Luschgy and Pag\`es(2023)]{LuschgyPages23}
	H.~Luschgy and G.~Pag\`es.
	\newblock \emph{Marginal and functional quantization of stochastic processes},
	volume 105 of \emph{Probability Theory and Stochastic Modelling}.
	\newblock Springer, Cham, 2023.
	\newblock \doi{10.1007/978-3-031-45464-6}.
	
	\bibitem[Mak and Joseph(2018)]{MakJoseph18}
	S.~Mak and V.~R. Joseph.
	\newblock Support points.
	\newblock \emph{Ann. Statist.}, 46\penalty0 (6A):\penalty0 2562--2592, 2018.
	\newblock \doi{10.1214/17-AOS1629}.
	
	\bibitem[Mattila(1995)]{Mattila95}
	P.~Mattila.
	\newblock \emph{Geometry of sets and measures in {E}uclidean spaces}, volume~44
	of \emph{Cambridge Studies in Advanced Mathematics}.
	\newblock Cambridge University Press, Cambridge, 1995.
	\newblock \doi{10.1017/CBO9780511623813}.
	
	\bibitem[M\'{e}rigot and Mirebeau(2016)]{MerigotMirebeau16}
	Q.~M\'{e}rigot and J.-M. Mirebeau.
	\newblock Minimal geodesics along volume-preserving maps, through semidiscrete
	optimal transport.
	\newblock \emph{SIAM J. Numer. Anal.}, 54\penalty0 (6):\penalty0 3465--3492,
	2016.
	\newblock \doi{10.1137/15M1017235}.
	
	\bibitem[M\'{e}rigot and Oudet(2016)]{MerigotOudet16}
	Q.~M\'{e}rigot and E.~Oudet.
	\newblock Discrete optimal transport: complexity, geometry and applications.
	\newblock \emph{Discrete Comput. Geom.}, 55\penalty0 (2):\penalty0 263--283,
	2016.
	\newblock \doi{10.1007/s00454-016-9757-7}.
	
	\bibitem[M\'{e}rigot et~al.(2024)M\'{e}rigot, Santambrogio, and
	Sarrazin]{MerigotSantambrogioSarrazin21}
	Q.~M\'{e}rigot, F.~Santambrogio, and C.~Sarrazin.
	\newblock Non-asymptotic convergence bounds for {W}asserstein approximation
	using point clouds.
	\newblock In \emph{Proceedings of the 35th International Conference on Neural
		Information Processing Systems}, NIPS '21, Red Hook, NY, USA, 2024. Curran
	Associates Inc.
	
	\bibitem[Oliver et~al.(1948)Oliver, Pierce, and Shannon]{OliverPierceShannon48}
	B.~M. Oliver, J.~R. Pierce, and C.~E. Shannon.
	\newblock The philosophy of {PCM}.
	\newblock \emph{Proceedings of the IRE}, 36\penalty0 (11):\penalty0 1324--1331,
	1948.
	\newblock \doi{10.1109/JRPROC.1948.231941}.
	
	\bibitem[Pag\`es(1998)]{Pages98}
	G.~Pag\`es.
	\newblock A space quantization method for numerical integration.
	\newblock \emph{J. Comput. Appl. Math.}, 89\penalty0 (1):\penalty0 1--38, 1998.
	\newblock \doi{10.1016/S0377-0427(97)00190-8}.
	
	\bibitem[Pag\`es(2015)]{Pages15}
	G.~Pag\`es.
	\newblock Introduction to vector quantization and its applications for
	numerics.
	\newblock In \emph{C{EMRACS} 2013---modelling and simulation of complex
		systems: stochastic and deterministic approaches}, volume~48 of \emph{ESAIM
		Proc. Surveys}, pages 29--79. EDP Sci., Les Ulis, 2015.
	\newblock \doi{10.1051/proc/201448002}.
	
	\bibitem[Pag\`es and Pham(2005)]{PagesPham05}
	G.~Pag\`es and H.~Pham.
	\newblock Optimal quantization methods for nonlinear filtering with
	discrete-time observations.
	\newblock \emph{Bernoulli}, 11\penalty0 (5):\penalty0 893--932, 2005.
	\newblock \doi{10.3150/bj/1130077599}.
	
	\bibitem[Pag\`es and Printems(2009)]{PagesPrintems09}
	G.~Pag\`es and J.~Printems.
	\newblock Optimal quantization for finance: From random vectors to stochastic
	processes.
	\newblock In A.~Bensoussan and Q.~Zhang, editors, \emph{Special Volume:
		Mathematical Modeling and Numerical Methods in Finance}, volume~15 of
	\emph{Handbook of Numerical Analysis}, pages 595--648. Elsevier, 2009.
	\newblock \doi{10.1016/S1570-8659(08)00015-x}.
	
	\bibitem[Panter and Dite(1951)]{PanterDite51}
	P.~F. Panter and W.~Dite.
	\newblock Quantization distortion in pulse-count modulation with nonuniform
	spacing of levels.
	\newblock \emph{Proceedings of the IRE}, 39\penalty0 (1):\penalty0 44--48,
	1951.
	\newblock \doi{10.1109/JRPROC.1951.230419}.
	
	\bibitem[Patel et~al.(2014)Patel, Frasca, and Bullo]{PatelFrascaBullo14}
	R.~Patel, P.~Frasca, and F.~Bullo.
	\newblock {Centroidal Area-Constrained Partitioning for Robotic Networks}.
	\newblock \emph{Journal of Dynamic Systems, Measurement, and Control},
	136\penalty0 (3):\penalty0 031024, 03 2014.
	\newblock \doi{10.1115/1.4026344}.
	
	\bibitem[Pierce(1970)]{Pierce70}
	J.~N. Pierce.
	\newblock Asymptotic quantizing error for unbounded random variables.
	\newblock \emph{IEEE Trans. Information Theory}, IT-16:\penalty0 81--83, 1970.
	\newblock \doi{10.1109/tit.1970.1054400}.
	
	\bibitem[Pollard(1982)]{Pollard82}
	D.~Pollard.
	\newblock Quantization and the method of {$k$}-means.
	\newblock \emph{IEEE Trans. Inform. Theory}, 28\penalty0 (2):\penalty0
	199--205, 1982.
	\newblock \doi{10.1109/TIT.1982.1056481}.
	
	\bibitem[Quattrocchi(2021)]{Quattrocchi21}
	F.~Quattrocchi.
	\newblock Quantization of probability measures and the import-export problem.
	\newblock Master's thesis, Universit\`a di Pisa, 2021.
	
	\bibitem[Santambrogio(2015)]{Santambrogio15}
	F.~Santambrogio.
	\newblock \emph{Optimal transport for applied mathematicians}, volume~87 of
	\emph{Progress in Nonlinear Differential Equations and their Applications}.
	\newblock Birkh\"{a}user/Springer, Cham, 2015.
	\newblock \doi{10.1007/978-3-319-20828-2}.
	
	\bibitem[Sarrazin(2022)]{Sarrazin22}
	C.~Sarrazin.
	\newblock Lagrangian discretization of variational mean field games.
	\newblock \emph{SIAM J. Control Optim.}, 60\penalty0 (3):\penalty0 1365--1392,
	2022.
	\newblock \doi{10.1137/20M1377291}.
	
	\bibitem[Shannon(1948)]{Shannon48}
	C.~E. Shannon.
	\newblock A mathematical theory of communication.
	\newblock \emph{Bell System Tech. J.}, 27:\penalty0 379--423, 623--656, 1948.
	\newblock \doi{10.1002/j.1538-7305.1948.tb01338.x}.
	
	\bibitem[Steele(1997)]{Steele97}
	J.~M. Steele.
	\newblock \emph{Probability theory and combinatorial optimization}, volume~69
	of \emph{CBMS-NSF Regional Conference Series in Applied Mathematics}.
	\newblock Society for Industrial and Applied Mathematics (SIAM), Philadelphia,
	PA, 1997.
	\newblock \doi{10.1137/1.9781611970029}.
	
	\bibitem[Steinhaus(1956)]{Steinhaus56}
	H.~Steinhaus.
	\newblock Sur la division des corps mat\'{e}riels en parties.
	\newblock \emph{Bull. Acad. Polon. Sci. Cl. III.}, 4:\penalty0 801--804, 1956.
	
	\bibitem[Thomassen(1992)]{Thomassen92}
	C.~Thomassen.
	\newblock The {J}ordan-{S}ch\"onflies theorem and the classification of
	surfaces.
	\newblock \emph{Amer. Math. Monthly}, 99\penalty0 (2):\penalty0 116--130, 1992.
	\newblock \doi{10.2307/2324180}.
	
	\bibitem[Villani(2009)]{Villani09}
	C.~Villani.
	\newblock \emph{Optimal transport}, volume 338 of \emph{Grundlehren der
		mathematischen Wissenschaften}.
	\newblock Springer-Verlag, Berlin, 2009.
	\newblock \doi{10.1007/978-3-540-71050-9}.
	
	\bibitem[Weed and Bach(2019)]{WeedBach19}
	J.~Weed and F.~Bach.
	\newblock Sharp asymptotic and finite-sample rates of convergence of empirical
	measures in {W}asserstein distance.
	\newblock \emph{Bernoulli}, 25\penalty0 (4A):\penalty0 2620--2648, 2019.
	\newblock \doi{10.3150/18-BEJ1065}.
	
	\bibitem[Weiss(2023)]{Weiss23}
	C.~Weiss.
	\newblock Approximation of discrete measures by finite point sets.
	\newblock \emph{Unif. Distrib. Theory}, 18\penalty0 (1):\penalty0 31--38, 2023.
	\newblock \doi{10.2478/udt-2023-0003}.
	
	\bibitem[Wheeden and Zygmund(2015)]{WheedenZygmund15}
	R.~L. Wheeden and A.~Zygmund.
	\newblock \emph{Measure and integral}.
	\newblock Pure and Applied Mathematics (Boca Raton). CRC Press, Boca Raton, FL,
	second edition, 2015.
	\newblock \doi{10.1201/b15702}.
	
	\bibitem[Williams(1991)]{Williams91}
	D.~Williams.
	\newblock \emph{Probability with martingales}.
	\newblock Cambridge Mathematical Textbooks. Cambridge University Press,
	Cambridge, 1991.
	\newblock \doi{10.1017/CBO9780511813658}.
	
	\bibitem[Xin et~al.(2016)Xin, L\'{e}vy, Chen, Chu, Yu, Tu, and Wang]{XinEtAl16}
	S.-Q. Xin, B.~L\'{e}vy, Z.~Chen, L.~Chu, Y.~Yu, C.~Tu, and W.~Wang.
	\newblock Centroidal power diagrams with capacity constraints: computation,
	applications, and extension.
	\newblock \emph{ACM Trans. Graph.}, 35\penalty0 (6), dec 2016.
	\newblock \doi{10.1145/2980179.2982428}.
	
	\bibitem[Xu and Berger(2019)]{XuBerger19}
	C.~Xu and A.~Berger.
	\newblock Best finite constrained approximations of one-dimensional
	probabilities.
	\newblock \emph{J. Approx. Theory}, 244:\penalty0 1--36, 2019.
	\newblock \doi{10.1016/j.jat.2019.03.005}.
	
	\bibitem[Xu et~al.(2022)Xu, Korba, and Slep\v{c}ev]{XuKorbaSlepcev22}
	L.~Xu, A.~Korba, and D.~Slep\v{c}ev.
	\newblock Accurate quantization of measures via interacting particle-based
	optimization.
	\newblock In K.~Chaudhuri, S.~Jegelka, L.~Song, C.~Szepesvari, G.~Niu, and
	S.~Sabato, editors, \emph{Proceedings of the 39th International Conference on
		Machine Learning}, volume 162 of \emph{Proceedings of Machine Learning
		Research}, pages 24576--24595. PMLR, 17--23 Jul 2022.
	
	\bibitem[Yukich(1998)]{Yukich98}
	J.~E. Yukich.
	\newblock \emph{Probability theory of classical {E}uclidean optimization
		problems}, volume 1675 of \emph{Lecture Notes in Mathematics}.
	\newblock Springer-Verlag, Berlin, 1998.
	\newblock \doi{10.1007/BFb0093472}.
	
	\bibitem[Zador(1964)]{Zador64}
	P.~L. Zador.
	\newblock \emph{Development and evaluation of procedures for quantizing
		multivariate distributions}.
	\newblock Phd thesis, Stanford University, 1964.
	
	\bibitem[Zador(1982)]{Zador82}
	P.~L. Zador.
	\newblock Asymptotic quantization error of continuous signals and the
	quantization dimension.
	\newblock \emph{IEEE Trans. Inform. Theory}, 28\penalty0 (2):\penalty0
	139--149, 1982.
	\newblock \doi{10.1109/TIT.1982.1056490}.
	
	\bibitem[Zhou et~al.(2024)Zhou, Yang, Shen, Xie, Tang, Zhao, Feng, Li, Chen,
	Lai, and Xu]{ZhouEtAl24}
	X.~Zhou, J.~Yang, L.~Shen, X.~Xie, J.~Tang, G.~Zhao, P.~Feng, S.~Li, Z.~Chen,
	N.~Y.~G. Lai, and J.~Xu.
	\newblock A new modelling method for open cell foam structure based on
	centroidal and capacity constraint power diagram.
	\newblock \emph{Journal of Sandwich Structures \& Materials}, 26\penalty0
	(5):\penalty0 662--678, 2024.
	\newblock \doi{10.1177/10996362231226334}.
	
	\bibitem[Zhu(2011)]{Zhu11}
	S.~Zhu.
	\newblock Asymptotic uniformity of the quantization error of self-similar
	measures.
	\newblock \emph{Math. Z.}, 267\penalty0 (3-4):\penalty0 915--929, 2011.
	\newblock \doi{10.1007/s00209-009-0653-1}.
	
	\bibitem[Zhu(2020)]{Zhu20}
	S.~Zhu.
	\newblock Asymptotic uniformity of the quantization error for the
	{A}hlfors-{D}avid probability measures.
	\newblock \emph{Sci. China Math.}, 63\penalty0 (6):\penalty0 1039--1056, 2020.
	\newblock \doi{10.1007/s11425-017-9436-4}.
	
\end{thebibliography}

\end{document}